\documentclass[11pt]{article}
\usepackage{amssymb,latexsym,amsmath, amsthm}
\usepackage{amsfonts, graphics,color}
\usepackage{mathtools}
\usepackage{hyperref} 
\usepackage{breakurl} 
\usepackage[american]{babel}
\usepackage{csquotes}
\usepackage{enumitem}
\usepackage{dsfont}
\usepackage[utf8]{inputenc}

\usepackage[top=1in, bottom=1.25in, left=1.25in, right=1.25in]{geometry}

\usepackage{subcaption} 
\usepackage{float}

\def\R{\mathbb R}

\def\N{\mathbb N}

\def\T{{\mathcal T}}

\def\H{{\mathcal H}}
\def\Rmax{R_{\mathrm{max}}}
\def\Rmin{R_{\mathrm{min}}}

\def\B{{\mathcal B}}

\def\Ri{{\mathcal R }}

\def\L{{\mathcal L}}
\def\M{{\mathcal M}}

\let\originalleft\left
\let\originalright\right
\renewcommand{\left}{\mathopen{}\mathclose\bgroup\originalleft}
\renewcommand{\right}{\aftergroup\egroup\originalright}
\newcommand{\ft}[0]{\footnotesize}

\newcommand{\im}{\mathrm{im}}

\usepackage{color}
\usepackage[normalem]{ulem} 
\def\bcr{\begin{color}{red}}
\def\bcb{\begin{color}{blue}}
\definecolor{darkgreen}{RGB}{0,150,0}
\def\bcg{\begin{color}{darkgreen}}
\def\ec{\end{color}}
\def\be{\begin{equation}}
\def\ee{\end{equation}}

\hypersetup{pdftitle={GRS-BSEV}}

\newtheorem{theorem}{Theorem}[section]
\newtheorem{definition}[theorem]{Definition}
\newtheorem{proposition}[theorem]{Proposition}
\newtheorem{defnprop}[theorem]{Definition \& Proposition}

\newtheorem{cor}[theorem]{Corollary}

\newtheorem{lemma}[theorem]{Lemma}
\newtheorem{defnlem}[theorem]{Definition \& Lemma}
\newtheorem{defnrem}[theorem]{Definition \& Remark}
\newtheorem{remark}[theorem]{Remark}

\numberwithin{equation}{section}

\title{A Birman-Schwinger Principle in General Relativity: Linearly Stable Shells of Collisionless Matter Surrounding a Black Hole}

\author{Sebastian~G\"unther, Gerhard~Rein, Christopher~Straub \vspace{0.4cm}   \\ 
  Department of Mathematics, University of Bayreuth, Germany}

\begin{document}

\maketitle

\begin{abstract}
We develop a Birman-Schwinger principle for the spherically symmetric, asymptotically flat Einstein-Vlasov system. It characterizes stability properties of steady states such as the positive definiteness of an Antonov-type operator or the existence of exponentially growing modes in terms of a one-dimensional variational problem for a Hilbert-Schmidt operator. This requires a refined analysis of the operators arising from linearizing the system, which uses action-angle type variables. For the latter, a single-well structure of the effective potential for the particle flow of the steady state is required. This natural property can be verified for a broad class of singularity-free steady states. As a particular example for the application of our Birman-Schwinger principle we consider steady states where a Schwarzschild black hole is surrounded by a shell of Vlasov matter. 
We prove the existence of such steady states and derive linear stability if the mass of the Vlasov shell is small compared to the mass of the black hole.
\end{abstract}

\vspace{-.5cm}

\tableofcontents

\section{Introduction}

In the framework of general relativity, we consider a large ensemble of massive particles which interact only through the Einstein equations 
\begin{equation}\label{eq:einstein_equations}
	G_{\alpha \beta} = 8\pi T_{\alpha \beta}.
\end{equation}
Here the Einstein tensor $G_{\alpha\beta}$ is induced by a Lorentzian metric $g_{\alpha \beta}$ with signature $(-,+,+,+)$ and determines the curvature of spacetime.  The energy-momentum tensor $T_{\alpha \beta}$ represents the matter and energy content of spacetime. Greek indices always run from $0$ to $3$. We write the metric in local coordinates $x^\alpha$, 
\[ 
ds^2 = g_{\alpha \beta} dx^\alpha dx^\beta,
\]
where the Einstein summation convention is employed. 
We assume that all the particles have rest mass equal to $1$ and
move forward in time. Thus, the particle density $f$ is supported on the
corresponding mass shell, and we write $f = f (t, x^i , p^j )$, where
we use $t:=x^0$ as a timelike coordinate,  Latin indices run from
$1$ to $3$, $p^\alpha$ are the canonical momentum variables
corresponding to $x^\alpha$, and $p^0$ is eliminated via the
mass shell condition $g_{\alpha\beta} p^\alpha p^\beta  = -1$.


The energy-momentum tensor induced by $f$ is
\begin{equation}\label{eq:energy_mom_tensor}
	T_{\alpha\beta} = \int_{\R^3} p_\alpha p_\beta f |g|^{\frac 12} \frac{dp^1 dp^2 dp^3}{-p_0},
\end{equation}
where we define $|g|$ as the modulus of the determinant of the metric. The evolution of $f$ is determined by the Vlasov equation
\begin{equation}\label{eq:vlasoveq_general}
	\partial_t f + \frac{p^i}{p^0} \partial_{x^i} f - \frac{1}{p^0} \Gamma^i_{\alpha \beta} p^\alpha p^\beta \partial_{p^i} f=0 ,
\end{equation}
also called the collisionless Boltzmann equation, where $\Gamma^\gamma_{\alpha \beta}$ are the Christoffel symbols corresponding to the metric $g_{\alpha\beta}$. The Einstein-Vlasov system is given by \eqref{eq:einstein_equations}, \eqref{eq:energy_mom_tensor}, \eqref{eq:vlasoveq_general} together with suitable boundary conditions and initial data for $f$. We consider isolated systems and thus impose asymptotic flatness of the spacetime. In its general form, the Einstein-Vlasov system is very difficult to handle which is why we restrict our study to the spherically symmetric case. We refer to \cite{An2011,Rein95} for an in-depth discussion.

\subsection{The Einstein-Vlasov system in Schwarzschild coordinates}
We simplify the Einstein-Vlasov system by introducing so-called Schwarzschild coordinates where the metric takes the form
\begin{equation}\label{eq:schwarzschild_metric}
	ds^2 = -e^{2\mu(t,r)} dt^2 + e^{2\lambda(t,r)} dr^2 + r^2 ( d\theta ^2 + \sin^2(\theta) d\phi^2). 
\end{equation}
The metric coefficients $\mu$ and $\lambda$ are functions on $I \times [0,\infty[$ for an interval $I\subset \R$
and depend on the time-coordinate $t\in I$ and the area radius $r\geq0$; $t$ can be thought of as the proper time of an observer located at spatial infinity.
The coordinates $\theta \in [0,\pi]$ and $\phi \in [0,2\pi]$ correspond to the angular coordinates in spherical symmetry. We introduce Cartesian coordinates
\[ 
	x = (x^1, x^2, x^3) = r(\sin(\theta) \cos(\phi), \sin(\theta) \sin(\phi), \cos(\theta)),
\]
and non-canonical momentum variables 
\[ 
	v^i = p^i + (e^\lambda-1) \frac{x\cdot p}{r} \frac{x^i}{r}, \quad i=1,2,3. 
\]
Here $x \cdot p$ denotes the Euclidean scalar product and we define $|v|^2 \coloneqq v\cdot v$. Derivatives with respect to $t$ and $r$ are denoted with $\dot{\phantom{.}}$ and $'$ respectively. Spherical symmetry of $f$ corresponds to 
\begin{equation}\label{eq:sphsym}
		f(t,x,v) = f(t,Ax, Av), \quad x,v\in \R^3, \quad A \in \mathrm{SO}(3).
\end{equation}
In these variables the Einstein-Vlasov system in spherical symmetry and Schwarzschild coordinates reads as follows: 
\begin{align}
	\partial_t f + e^{\mu-\lambda} \frac{v}{\sqrt{1+|v|^2}} \cdot \partial_x f- \left (\dot \lambda \frac{x\cdot v}{r} + e^{\mu-\lambda} \mu' \sqrt{1+|v|^2} \right  ) \frac{x}{r} \cdot \partial_v f=0, \label{eq:vlasov}
\end{align}
\begin{align}
	e^{-2\lambda}(2r\lambda'-1)+1 &= 8\pi r^2 \rho_f \label{eq:fieldeq1}, \\
	e^{-2\lambda}(2r\mu'+1)-1 &= 8\pi r^2 p_f	\label{eq:fieldeq2} , \\
	\dot \lambda= -4\pi r& e^{\lambda+\mu} j_f \label{eq:fieldeq3}, \\
	e^{-2\lambda} \left (\mu'' + (\mu'-\lambda')(\mu'+\frac 1r)\right ) - &e^{-2\mu} \left (\ddot{\lambda} + \dot \lambda(\dot \lambda-\dot \mu)\right ) = 8\pi q_f \label{eq:fieldeq4} ,
\end{align}
\begin{align}
	\rho_f(t,r) &= \rho_f(t,x)= \int_{\R^3}   \sqrt{1+|v|^2} \, f(t,x,v)\, dv ,\label{eq:rho}\\
	p_f(t,r) &= p_f(t,x)= \int_{\R^3}  \left  (\frac{x\cdot v}{r}\right )^2 \, f(t,x,v)\, \frac{dv}{ \sqrt{1+|v|^2} } , \label{eq:p} \\
	j_f(t,r )&= j_f(t,x)=  \int_{\R^3}   \frac{x \cdot v}{r} \,f(t,x,v)\, dv , \label{eq:j} \\
	q_f(t,r) &= q_f(t,x) = \frac 12 \int_{\R^3} \left |\frac{x \times v}{r}\right |^2 \,f(t,x,v) \frac{dv}{ \sqrt{1+|v|^2} }. \label{eq:q}
\end{align}
Here, \eqref{eq:vlasov} is the Vlasov equation, \eqref{eq:fieldeq1}--\eqref{eq:fieldeq4} are the field equations, and~\eqref{eq:rho}--\eqref{eq:q} represent the energy momentum tensor. The system is not complete without boundary conditions and initial data. In both settings we consider, we impose an asymptotically flat spacetime, i.e.,
\begin{equation}\label{eq:asym_flat}
	\lim_{r\to\infty}\mu(t,r) = 0  = 	\lim_{r\to\infty}\lambda(t,r) ,
\end{equation}
and that $f(t)$ has compact support. For the remaining boundary and initial conditions, we distinguish between two situations. On the one hand, we consider singularity-free spacetimes with non-negative, spherically symmetric, initial distributions $\mathring f \in C^1_c(\R^3\times \R^3)$ satisfying
\begin{equation}\label{eq:sing_free_2mr}
	4\pi \int_0^r \rho_{\mathring f}(s)s^2 \, ds < \frac{r}{2}, \quad r  > 0,
\end{equation} 
and impose 
\begin{equation}\label{eq:sing_free_boundary}
	\lambda(t,0) = 0 .
\end{equation}
We call \eqref{eq:vlasov}--\eqref{eq:sing_free_boundary} the {\em singularity-free Einstein-Vlasov system}. On the other hand, we consider the setting where a Schwarzschild black hole of mass $M>0$ is situated at the center of the spacetime. Schwarzschild coordinates can only cover points of the spacetime where $r>2M$. Therefore, we allow non-negative, spherically symmetric, initial distributions $\mathring f \in C^1_c\left (\{x\in \R^3\mid r>2M\}\times \R^3\right  )$ with
\begin{equation}\label{eq:sing_2mr}
	M + 	4\pi \int_{2M}^r \rho_{\mathring f}(s)s^2 \, ds<  \frac{r}{2}, \quad r > 2M,
\end{equation} 
and prescribe
\begin{equation}\label{eq:blackhole_boundary}
	\lim_{r\to 2M} e^{-2\lambda(t,r)} = 0.
\end{equation}
Accordingly, we call \eqref{eq:vlasov}--\eqref{eq:asym_flat}, \eqref{eq:sing_2mr}, \eqref{eq:blackhole_boundary} the {\em Einstein-Vlasov system with Schwarzschild-singularity of mass $M$}. In both settings ($M=0$ and $M>0$), there are conserved quantities. Firstly, the ADM-mass given by
\[ 
M_{ADM} \coloneqq M +  4 \pi \int_{2M}^\infty \rho_f(r) r^2 \, dr
\]
and, secondly, the Casimir functionals 
\[ 
N_{\zeta} \coloneqq  \int_{\{r>2M\}} \int_{\R^3} e^{\lambda} \, \zeta(f) \, dvdx,
\]
where $\zeta \in C^1(\R)$ with $\zeta(0)=0$. The special case $N \coloneqq N_{\mathrm{id}}$ can be interpreted as the number of particles or equivalently the rest mass of the system. 

A comment on terminology is at order: when talking about the Einstein-Vlasov system, we generally mean both settings introduced above. If a statement is only applicable in one setting, we always refer to the singularity-free or the Schwarzschild-singularity situation, respectively. 

In the singularity-free setting there exist unique, local in-time solutions for smooth, compactly supported initial data together with a continuation criterion \cite{Rein95,RR92a}. 
Similar results are known in other coordinates, e.g., in maximal-areal coordinates~\cite{GueRe21} or maximal isotropic coordinates~\cite{Re97}. 
In the case of a Schwarzschild singularity at the center, it can be shown that the methods from \cite{ReReSch95} yield a global existence result in Schwarzschild time. 

\subsection{Steady states and previous stability results}\label{sc:introsc}

The Einstein-Vlasov system possesses a plethora of physically reasonable equilibria whose stability behavior is analyzed in this work.

In the singularity-free setting, a classical way~\cite{RaRe2013,Rein94,ReRe00} of constructing steady states is due to the following observation: Any sufficiently regular function of the form $f_0=\varphi(E,L)$ solves the Vlasov equation~\eqref{eq:vlasov} since the {\em particle energy}
\begin{equation}\label{eq:ststparticleenergy}
	E=E(x,v)= e^{\mu_0(r)} \sqrt{1+|v|^2},
\end{equation}
with $\mu_0$ induced by $f_0$ via \eqref{eq:fieldeq1}--\eqref{eq:p}, and the {\em angular momentum} $L \coloneqq |x\times v|^2$ are preserved along the characteristic flow of the Vlasov equation. 
Consequently, the Einstein-Vlasov system reduces to the field equations. 
It is well-known that for a large class of microscopic equations of state~$\varphi$ there exist solutions of the reduced system which correspond to compactly supported steady states with finite mass; we recall these arguments in Section~\ref{sc:stst_classic}.

The first work analyzing stability in the context of the singularity-free Einstein-Vlasov system is~\cite{Ip1980}, where the system is linearized around suitable equilibria and it is shown that linear stability investigations come down to the spectral analysis of an associated unbounded operator. 
This is similar as for the spherically symmetric, gravitational Vlasov-Poisson system---the non-relativistic counterpart to the Einstein-Vlasov system---where both linear and non-linear stability of all physically relevant steady states is known by now~\cite{An1961,DoFeBa,GuRe99,KS,LeMeRa12}. In the relativistic setting, 
it was shown on a linearized level in~\cite{HaRe2013,HaRe2014} that suitable singularity-free steady states are stable for small values of the central redshift, i.e., if the setting is not too relativistic. In~\cite{HaLiRe2020} the converse statement was proven, i.e., steady states become linearly unstable for large values of the central redshift; note that this is in sharp contrast to the non-relativistic situation. In addition, it is shown that there is a trichotomy in phase space into a stable, unstable, and center space. In work towards non-linear stability,  it was tried to obtain steady states as minimizers of an appropriate energy-Casimir functional in~\cite{Wo}, but~\cite{Wo} contains serious flaws~\cite{AnKu20}. In~\cite{AnKu} the corresponding Euler-Lagrange equation is solved, but non-linear stability is still very much elusive. 

Non-linear stability in the singularity-free setting has, however, been studied numerically for decades. A question of particular interest is to determine the point(s) where stability changes along suitable families of equilibria, e.g., steady states with fixed equation of state parametrized by a redshift factor. It has been conjectured and supported in various numerical investigations~\cite{AnRe2006, Praktikum20, Ip1980, ShTe1985_2, Ze1971, Ze1966} that changes in stability correspond to critical points of the so-called binding energy
\[ 
E_b = \frac{N-M_{ADM}}{N}.
\]
However, recent numerical evidence~\cite{GueStRe21} strongly contradicts this hypothesis and shows that stability behaviors can be much more diverse than previously thought. As already stated by Ipser~and~Thorne~\cite{Ip1980}, new versatile criteria are needed in order to gain more understanding of stability issues in general relativity.

The existence of static shells around a Schwarzschild black hole has been shown in~\cite{Jab2021,Rein94}, but we develop a different approach adapted to our linear stability analysis. More precisely, we consider the modified ansatz
\begin{equation}\label{eq:ststansatzlochINTRO}
	f^\delta(x,v)=\delta\,\tilde\chi(r)\,\varphi\left(E(x,v),L(x,v)\right),
\end{equation}
where $\delta>0$ controls the size of the static solution, $\tilde\chi\colon\R\to\R$ is a radial cut-off function, and~$\varphi$ is similar to the singularity-free setting. Again, $E$ is the particle energy~\eqref{eq:ststparticleenergy} associated to $f^\delta$ and $L$ is the angular momentum. If~$\tilde{\chi}$ and~$\varphi$ are suitably chosen, we  show in Section~\ref{sc:stst_hole} that~\eqref{eq:ststansatzlochINTRO} indeed defines a physically reasonable static solution of the Einstein-Vlasov system with Schwarzschild-singularity  by combining the techniques from~\cite{RaRe2013} and~\cite{Rein94}; we will have $\tilde\chi\equiv1$ on the steady state support.

To the authors' knowledge, there are no previous results---neither analytical nor numerical---concerning stability of shells of Vlasov matter around a black hole.
Recently, related steady states were constructed for the massless Einstein-Vlasov system~\cite{An2021}.
Moreover, non-linear stability of certain static solutions of the spherically symmetric Vlasov-Poisson system with a fixed central point mass---which can be interpreted as the non-relativistic analogue of the Einstein-Vlasov system with Schwarzschild-singularity---was shown in~\cite{Sc09}.

\subsection{Main results}
Our first main result introduces a reduced, one-dimensional variational principle for the linear stability of static solutions to the Einstein-Vlasov system and gives a new sufficient condition for linear stability. We emphasize that the theorem is applicable to steady states as described above both with and without a Schwarzschild-singularity; we discuss the required assumptions below and refer to Section~\ref{sc:ststconditions} for a detailed specification of the classes of steady states we handle in our investigation.

\begin{theorem}[A reduced variational principle]\label{thm:mathur_stability}
	Consider a static solution to the Einstein-Vlasov system as above. Then, there exists a semi-explicit integral kernel $K \in L^2([0,\infty[^2)$ which is compactly supported and depends on the steady state, such that the following holds:
	\begin{enumerate}[label=(\alph*)]
		\item\label{it:stab1} The steady state is linearly stable if, and only if,
		\[
		\sup \limits_{ \substack{G\in L^2([0,\infty[) \\ \|G\|_{2}=1} }  \int_{0}^{\infty}\int_{0}^{\infty} K(r,s) G(r)G(s) \, ds dr < 1.
		\]
		If equality holds, there exists a zero-frequency mode but no exponentially growing mode.
		\item\label{it:stab2} The number of exponentially growing modes of the steady state is finite and strictly bounded by $\|K\|_{L^2([0,\infty[^2)}^2$. 
		\item\label{it:stab3} The steady state is linearly stable if $\|K\|_{L^2([0,\infty[^2)}<1$.
	\end{enumerate}
\end{theorem}
We define linear stability through the (strict) positivity of the second-order variation of the energy-Casimir functional, see Definition~\ref{def:linear_stab} and Remark~\ref{remark:linear_stab}; there we also describe what we mean by exponentially growing modes and zero-frequency modes. 

The terminology {\em zero-frequency mode} is due to~\cite[\S{}~IV{\textit{f}}]{IT68}, where it was shown that a zero-frequency mode carries one steady state to another nearby equilibrium. Another point of view is that the situation of a zero-frequency mode but no exponentially growing modes corresponds to the point where linear stability might change along a family of steady states depending on some parameter in a continuous way. Thus, the criterion provided by~\ref{it:stab1} might be useful to understand the onset of instability along families of equilibria.

Part~\ref{it:stab2} can be interpreted as a {\em Birman-Schwinger type bound} on the number of exponentially growing modes. Birman-Schwinger bounds are classical in quantum mechanics, where they are, e.g., used to bound the number of negative eigenvalues of Schr\"odinger operators by integrals of the potential, cf.~\cite[Sc.~4.3]{LiSe10} or~\cite[Thm.~XIII.10]{ReSi4}.

There are two limitations of our result. Firstly, our analysis requires that the steady state under investigation is of single-well structure (cf.~Definition~\ref{def:jeans}) and that the periods of the  particle motions (cf.~Definition~\ref{def:periodfunction}) are bounded and bounded away from zero within the equilibrium configuration. These properties are rigorously verified for static shells around a Schwarzschild black hole provided that the mass of the shell is sufficiently small compared to the black hole. In the singularity-free setting, we show that these properties are satisfied for isotropic equilibria which are not too relativistic, but emphasize that numerical simulations indicate that they are true for a much larger class of steady states.
We elaborate more on these assumptions in Remark~\ref{remark:single_well_numerics} as well as in the following subsection.

Secondly, the integral kernel $K$ is not fully explicit; see~\eqref{eq:kernel_K} for its definition. It contains a projection onto the kernel of an important operator; while the latter operator and its kernel are known explicitly, the orthogonal projection onto this space is, unfortunately, not explicitly known to us.

Nonetheless, we are able to apply Theorem~\ref{thm:mathur_stability} rigorously to small matter shells surrounding a black hole.

\begin{theorem}[Linear stability of small matter shells around a Schwarzschild black hole]\label{thm:shell_stab}
	Consider the (spherically symmetric) Einstein-Vlasov system with a Schwarzschild-singularity of mass $M>0$ at the center. There exist families of steady states $(f^\delta)_{\delta>0}$ which are linearly stable for $0<\delta \ll1$, where the parameter $\delta>0$ controls the size of the Vlasov shell.	As $\delta$ goes to zero, the metric coefficients converge uniformly on $]2M,\infty[$ to the vacuum Schwarzschild metric coefficients of mass $M$ and the densities $f^\delta$ converge pointwise to zero on $\{r>2M\} \times \R^3$. 
\end{theorem}
In contrast to Theorem~\ref{thm:mathur_stability}, all properties required for the result above are rigorously proven for a large variety of families of steady states; for a detailed description of these equilibria we refer to Theorem~\ref{thm:shell_stab_2}.

We emphasize the fact that in previous linear stability results \cite{HaRe2013, HaRe2014}  it was essential that the corresponding steady state is close to Newtonian. Here the metric of the steady state under consideration is close to the Schwarzschild metric so that from a physics point of view we are studying a very different situation.

The mere existence of these stationary solutions shows that small, spherically symmetric perturbations of Schwarzschild spacetime consisting of Vlasov matter do not necessarily converge asymptotically to a new vacuum Schwarzschild spacetime, since small values of $\delta$ correspond to small mass of the Vlasov matter compared to the mass of the central singularity. The linear stability result above suggests that the same is also true for sufficiently weak perturbations of these small static shells, i.e., the class of spherically symmetric solutions which start close to the Schwarzschild spacetime but do not converge to a new vacuum Schwarzschild spacetime does not seem to be pathological. This is in sharp contrast to the celebrated result in~\cite{DaHoRoTa} where it is shown on a non-linear level and without symmetry assumptions that small, vacuum perturbations of Schwarzschild spacetime converge asymptotically to another member of the Schwarzschild family, modulo the Kerr solutions.

In the singularity-free case the situation is different. It is known that small initial data of the Einstein-Vlasov system converges asymptotically to Minkowski space~\cite{FaJoSm21, LiTa19}; in the special case of spherical symmetry this was already proven in~\cite{RR92a}. 

\subsection{Methodology and outline of the paper}\label{sc:methodology}
We now present the techniques used to derive our main results since we believe they are quite flexible and can be adapted to various problems in future work.

In Section~\ref{sc:stst} a plethora of steady states for the Einstein-Vlasov system is constructed using the ansatz described in Section~\ref{sc:introsc}. In Section~\ref{sc:stst_hole} we prove  the existence of a new set of static matter shells of the form~\eqref{eq:ststansatzlochINTRO} around a Schwarzschild black hole.

Crucial parts of our investigation are based on {\em action-angle type variables}, which we introduce in Section~\ref{sc:aacoords}. While action-angle variables are a classical tool in Hamiltonian mechanics~\cite{Arnold,LaLi,LB1994}, they have been used recently to derive a Birman-Schwinger principle in Newtonian galactic dynamics~\cite{HaReSt21,Kunze} and to analyze phase mixing~\cite{RiSa2020}.

In order to define action-angle type variables, it is necessary that every particle orbit within some fixed equilibrium configuration can be uniquely characterized by its particle energy~$E$ and angular momentum~$L$; the latter two quantities are integrals of the characteristic system. This corresponds to the effective potential associated with the steady state having a {\em single-well structure}. A rigorous description of this property is given in Section~\ref{sc:jeans_sc}, where we also discuss its validity. In the case of a Schwarzschild-singularity, we can rigorously show that small static shells have this property by considering the limit $\delta\to0$ in~\eqref{eq:ststansatzlochINTRO}; a related result has been obtained in~\cite{Jab2021}. In the singularity-free setting, we prove that a steady state has single-well structure if it is isotropic, i.e., $f_0=\varphi(E)$, and satisfies 
\begin{equation}\label{eq:INTROtwomoverrloweronethird}
	\frac{2m(r)}r\leq\frac13,\quad r>0,
\end{equation}
where $m$ is the quasi-local mass of the equilibrium. The interpretation of the condition~\eqref{eq:INTROtwomoverrloweronethird} is that it corresponds to the steady being not too relativistic. However, this is not yet satisfying since a desired application is to analyze steady states as the redshift gets larger. Although we cannot rigorously show the single-well structure in the latter setting, we note that numerical simulations clearly show its presence for large values of the central redshift for a wide class of steady states, e.g., for general isotropic equilibria.

We emphasize that the single-well structure of the effective potential is related to Jeans' theorem, which is known to be violated for certain steady states of the singularity-free system, see~\cite{Sch99} and the numerical study in~\cite{AnRe2007}. This is in sharp contrast to the non-relativistic situation, where all relevant steady states of the radial Vlasov-Poisson system are of single-well structure~\cite[Lemma~2.1]{LeMeRa11} and satisfy Jeans' theorem~\cite{Batt86}. 

Another point of view is that the single-well structure of the effective potential leads to the associated characteristic flow to be ergodic, cf.~\cite[Sc.~II.5]{ReSi1}.

In the case of a single-well structure, all particle motions within the associated equilibrium configuration are trapped and every particle orbit is either stationary or time-periodic. A necessary property for the following analysis is that the periods of these orbits are bounded and bounded away from zero on the steady state support. 
In Section~\ref{sc:periodfunction} we prove these bounds for the same static solutions for which we show the presence of a single-well structure, but emphasize that numerical simulations show their validity for a much larger class of equilibria including general isotropic steady states. Similar bounds on the particle periods are crucial to derive a Birman-Schwinger principle in the non-relativistic setting, cf.~\cite[Prop.~B.1]{HaReSt21} and~\cite[Theorems~3.2 and~3.5]{Kunze}.

In Section~\ref{sc:linearization} the Einstein-Vlasov system is formally linearized around a fixed steady state with the properties discussed above. As in~\cite{HaLiRe2020,IT68} we apply Antonov's trick~\cite{An1961} to arrive at the second-order evolution equation
\[ 
\partial^2_t f + \L f = 0
\]
for the odd-in-$v$ part $f$ of the perturbation. $\L$ is called the {\em linearized operator} or {\em Antonov operator} and it is of the form
\begin{equation}
	\L = -\B^2 -\Ri=-\left(\T+\mathcal S\right)^2-\Ri,
\end{equation}
where $\Ri$ and $\mathcal S$ are non-local operators and $\T$ is the transport operator associated to the characteristic flow of the equilibrium.
In Section~\ref{sc:def_operators} we carefully define these operators on a suitable Hilbert space $H$, which is the $L^2$ space over the steady state support with weight~$\frac{e^{\lambda_0}}{-\partial_E\varphi}$. This weight causes $\Ri$ to be symmetric, $\mathcal S$ to be skew-symmetric, and both of these operators to be bounded on $H$. Moreover, it is shown in~\cite{ReSt20} that $\T$ can be defined on a dense subset of $H$ such that the resulting operator is skew-adjoint. Section~\ref{sc:properties_operators} is entirely devoted to an in-depth analysis of the operators $\T$, $\B$, and $\Ri$. Overall, this leads to~$\L$ being an unbounded, self-adjoint operator on the subspace $\H$ of odd in $v$ functions.
Observe that in order for $H$ to be a Hilbert space, we have to require that
\begin{equation}\label{eq:ansatzmonotonic}
	\varphi'\coloneqq\partial_E\varphi<0
\end{equation}
on the steady state support, i.e., the concentration of ever more energetic particles is decreasing within the equilibrium configuration. Nonetheless, we emphasize that this condition is natural from a physics of view~\cite{Ze1971}. 
Moreover, eqn.~\eqref{eq:ansatzmonotonic} is the reason why we do not consider the steady states of the Einstein-Vlasov system with a Schwarzschild-singularity constructed in~\cite{Jab2021}, as the distribution functions of steady states derived there are not monotonic in the particle energy.

Linear stability corresponds to the positivity of the spectrum of $\L$, i.e., $\inf\left(\sigma(\L)\right)>0$. We show that the essential spectrum of $\L$ is always positive, which implies that every non-positive element in the spectrum of $\L$ has to be an eigenvalue. It thus remains to characterize these non-positive eigenvalues of $\L$, which we do by deriving a {\em Birman-Schwinger principle} in Section~\ref{sc:birman_schwinger_principle}. This principle originates from quantum mechanics, where it is, e.g., used to investigate the presence of eigenvalues below a given energy level of time-independent Schr\"odinger operators of the form~$-\Delta-V$ with prescribed potential $V\geq0$. Classical references covering the Birman-Schwinger principle in quantum mechanics are~\cite[Sc.~12.4]{LiLo01}, \cite[Sc.~4.3]{LiSe10}, \cite[Sc.~XIII.3]{ReSi4} or~\cite[Sc.~III.3]{Si}; we present it in our specific situation:

A formal calculation shows that $0$ is an eigenvalue of $\L_\gamma \coloneqq - \B^2 - \frac 1\gamma \Ri$ for $\gamma>0$ if and only if $\gamma$ is an eigenvalue of the {\em Birman-Schwinger operator}
\begin{equation}
	Q \coloneqq - \sqrt \Ri\,\B^{-2} \sqrt\Ri;
\end{equation}
the existence of $\sqrt\Ri$ and $\B^{-2}$ on suitable spaces is derived---with considerable effort---in Section~\ref{sc:properties_operators}.
In Section~\ref{appendix} we show that $\Ri\geq0$ implies that eigenvalues of $\L_\gamma$ are non-decreasing and continuous in~$\gamma$ and that the spectrum of $\L_\gamma$ gets positive for sufficiently large~$\gamma$. Hence, the number of non-positive eigenvalues of $\L=\L_1$ equals the number of $\gamma\geq1$ such that $\L_\gamma$ has the eigenvalue~$0$, which is the same as the number of eigenvalues~$\geq1$ of~$Q$. These identities are proven in Section~\ref{ssc:birman_schwinger_operator}, where it is also shown that the (algebraic) multiplicities of eigenvalues carry over from one operator onto the other.

The spectral analysis of the operator $Q$ now simplifies by observing that $\im(Q)\subset\im(\sqrt\Ri)$ and that every function in the image of $\sqrt\Ri$ is of the form $|\varphi'(E,L)|\,w\,\alpha_0(r)\,F(r)$ for some $F\in L^2([0,\infty[)$ and fixed $\alpha_0$ depending on the underlying steady state. Thus, we investigate the operator $\M\colon L^2([0,\infty[)\to L^2([0,\infty[)$ defined by
\begin{equation}
	Q\left(|\varphi'(E,L)|\,\frac{x\cdot v}r\,\alpha_0(r)\,F(r)\right)=|\varphi'(E,L)|\,\frac{x\cdot v}r\,\alpha_0(r)\,\left(\M F\right)(r),
\end{equation}
which we call the {\em reduced operator} or {\em Mathur operator}, as the reduction process goes back to Mathur~\cite{Ma} who studied a related problem in the context of the Vlasov-Poisson system with an external potential. 
The operator $\M$ is reduced in the sense that it acts only on radial functions and not on functions on the full phase-space, like $Q$ and $\L$ do.
Still, the non-zero eigenvalues of~$Q$ and~$\M$ correspond to one another, which allows us to limit the spectral analysis to~$\M$.

In Section~\ref{sc:mathur_operator} we prove that the reduced operator~$\M$ is symmetric, non-negative, and Hilbert-Schmidt with integral kernel $K$, cf.~\cite[Thm.~VI.22 et~seq.]{ReSi1}. Moreover, in Section~\ref{ssc:mathur_operator_explicit} we derive a semi-explicit representation of~$K$ based on the in-depth understanding of $\B^{-2}$ and several further properties of~$\B$ established in Section~\ref{sc:properties_B}. 
We thus arrive at a variational principle for $\M$ consisting of integration over the radial, semi-explicitly known integral kernel $K$ which fully describes the presence of non-positive eigenvalues of $\L$ in a quantitative way.

As an application of this general Birman-Schwinger principle, we show in Section~\ref{sc:stab_blackhole_solutions} that matter shells around a Schwarzschild black hole at the center are linearly stable by making these steady states sufficiently small compared to the mass of the central black hole. 

\vspace*{.5cm}
\noindent
{\bf Acknowledgments.}
The authors thank Mahir~Had\v zi\'c for inspiring discussions. Some parts of this work were developed in the stimulating atmosphere of the Erwin Schr\"odinger International Institute for Mathematics and Physics during the thematic program \enquote{Mathematical Perspectives of Gravitation beyond the Vacuum Regime}, which we thank for its hospitality. 

\section{Steady states}\label{sc:stst}

In this section we introduce the static solutions whose linear stability properties are analyzed in our work. We distinguish between two conceptually different situations:
Firstly, a singularity-free situation where the Schwarzschild metric \eqref{eq:schwarzschild_metric} can cover all possible radii, i.e., $2m(r)<r$ for the quasi-local mass $m$ and $r>0$. We recall the construction of such steady states in Section~\ref{sc:stst_classic}. Secondly, our theory also works in the case of a spacetime with a Schwarzschild-singularity, i.e., a black hole, of mass $M>0$ at the center. In Section~\ref{sc:stst_hole} we show how to construct suitable steady states in that situation.

\subsection{Singularity-free stationary solutions}\label{sc:stst_classic}

For the construction of stationary solutions to the singularity-free Einstein-Vlasov system in the spherically symmetric case we briefly recall the arguments in~\cite{RaRe2013}; see \cite{Rein94,ReRe00} for slightly different approaches. We consider a microscopic equation of state of the form
\begin{equation} \label{eq:ansatzsingularityfreestst}
	f(x,v) = \varphi(E,L) = \Phi\left (1-\frac{E}{E_0}\right ) (L-L_0)^l_+,
\end{equation} 
where $l> -\frac12$ and an index $+$ denotes the positive part of a function. Moreover, $L_0 \geq 0$ gives a lower bound for the angular momentum, i.e., $L_0>0$ corresponds to solutions with a vacuum region at the center of the steady state. 
In particular, the choice $l=0=L_0$ causes~$\varphi$ to depend solely on the particle energy $E$; such static solutions are called {\em isotropic}.
As an aside, we note that the explicit form of the $L$-dependency in~\eqref{eq:ansatzsingularityfreestst} is solely chosen for the sake of simplicity and it is straight-forward to extend our analysis to steady states with more general $L$-dependencies.
We impose that $\Phi= \Phi(\alpha)$ fulfills the following conditions:
\begin{enumerate}[label=($\Phi\arabic*$)]
	\item\label{it:assphi1} $\Phi \colon \R \to [0,\infty[$, $\Phi\in L^\infty_{\mathrm{loc}}(\R)$, and $\Phi(\alpha)=0$ for $\alpha\leq0$.
	\item\label{it:assphi2} There exist constants $c_1,c_2>0$, $\alpha_0>0$, and $0\leq k<l+\frac{3}{2}$ such that 
	\[ 
	c_1\alpha^k \leq \Phi(\alpha)  \leq c_2 \alpha^k, \quad \alpha \in ]0,\alpha_0].
	\]
\end{enumerate}
Common examples for this function are $\Phi(\alpha)=\left(e^\alpha-1\right)_+$ or $\Phi(\alpha)=\alpha_+^k$ with $0\leq k<l+\frac32$; the resulting steady states are known as a King model or polytropes, respectively.
These properties of $\Phi$ together with the presence of the cut-off energy $E_0\in]0,1[$ in~\eqref{eq:ansatzsingularityfreestst} will guarantee a compact support and finite mass.
Inserting the ansatz $f(x,v) = \varphi(E(x,v),L(x,v))$ with $E$ given by~\eqref{eq:ststparticleenergy} into the singularity-free Einstein-Vlasov system reduces the system to an equation for the metric coefficient $\mu$. It turns out that it is more convenient to make $E_0$ part of the unknowns and to consider $y\coloneqq \ln(E_0) - \mu$ instead of $\mu$. The equation for $y$ reads
\begin{align}\label{eq:ssdgl}
	y'(r) = - \frac1{1-\frac{8\pi}{r}\int_0^rs^2G(s,y(s))\,ds} \left( \frac{4\pi}{r^2} \int_0^rs^2G(s,y(s))\, ds + 4\pi r H(r,y(r)) \right), \quad y(0)=y_0,
\end{align}
for a prescribed initial value $y_0>0$, where 
\begin{align}
	G(r,y) &\coloneqq 2\pi c_l\,r^{2l} e^{3y}  \int_0^{1-e^{-y}\sqrt{1+\frac{L_0}{r^2}}} \Phi(\alpha)  (1-\alpha)^2 \left (e^{2y} (1-\alpha)^2-1-\frac{L_0}{r^2}\right )^{l+\frac{1}{2}} \, d\alpha, \label{eq:G_phi_singfree}\\
	H(r,y) &\coloneqq 2\pi d_l\,r^{2l} e^y \int_0^{1-e^{-y}\sqrt{1+\frac{L_0}{r^2}}} \Phi(\alpha)  \left (e^{2y} (1-\alpha)^2-1-\frac{L_0}{r^2}\right )^{l+\frac{3}{2}} \, d\alpha, \label{eq:H_phi_singfree}
\end{align}
for $(r,y)\in]0,\infty[\times\R$ with $e^{-y}\sqrt{1+\frac{L_0}{r^2}}<1$ and $G(r,y)\coloneqq0\eqqcolon H(r,y)$ otherwise. Here,
\begin{equation*}
	c_l\coloneqq \int_0^1 \frac{s^l}{\sqrt{1-s}}\,ds ,\quad d_l\coloneqq \int_0^1 s^l \sqrt{1-s}\, ds.
\end{equation*}
These quantities are related to the density and pressure induced by~$f$ via 
\[ 
\rho_f(r) = G(r,y(r)), \quad p_f(r) = H(r,y(r)),\quad r>0.
\]
In \cite{RaRe2013} it is shown that under the assumptions \ref{it:assphi1},\ref{it:assphi2} there exists a unique solution $y\in C^1([0,\infty[)$ of \eqref{eq:ssdgl} with $y_{\infty} \coloneqq \lim_{r\to\infty} y(r)<0$. Setting $E_0 = e^{y_{\infty}}$,  $\mu = \ln(E_0) - y$, and 
\[
\lambda(r) = -\frac{1}{2} \ln\left ( 1-\frac{8\pi}{r}\int_0^rs^2G(s,y(s)) \, ds \right ),\quad r>0,
\]
then defines a non-trivial stationary solution of the singularity-free Einstein-Vlasov system with finite mass and compact support. We denote by
\begin{equation}\label{eq:twomoverr}
	m(r)\coloneqq4\pi\int_0^r\rho_f(s)s^2\,ds = \frac{r}{2}\left (1-e^{-2\lambda(r)}\right ),\quad r>0,
\end{equation}
the quasi-local mass induced by the density~$\rho$ of the equilibrium.

To summarize,  for fixed $L_0$, $l$, and $\Phi$ as above, we obtain a family of static solutions parameterized by the initial value $y_0>0$ 
which is related to the central redshift factor, see~\cite[(2.11)]{GueStRe21}. For the study of linear stability we will later impose some further assumptions on the stationary solution which are stated in Section~\ref{sc:ststconditions}.

\subsection{Matter shells surrounding a Schwarzschild black hole}\label{sc:stst_hole}

We look for a time-independent solution of the Einstein-Vlasov system with a Schwarzschild-singularity of given mass $M>0$ of the form
\begin{equation}\label{eq:lochdf}
	f^\delta(x,v)=\delta\,\chi(r-r_0)\,\varphi(E(x,v),L(x,v)),
\end{equation}
where $\delta\geq0$ and $\chi\in C^\infty(\R)$ is a non-negative, radial cut-off function with $\chi(s)=0$ for $s\leq0$ and $\chi(s)=1$ for $s\geq\eta_0$ with suitable $\eta_0>0$. The microscopic equation of state $\varphi$ is again of the form~\eqref{eq:ansatzsingularityfreestst} with $E_0$ replaced by a $\delta$-dependent cut-off energy $E^\delta$. Here, $l>-\frac12$ and we assume that $\Phi$ satisfies~\ref{it:assphi1} together with
\begin{enumerate}[label=($\Phi\arabic*$)]\addtocounter{enumi}{2}
	\item\label{it:assphi3} There exists a constant $\alpha_0>0$ such that $\Phi>0$ on $]0,\alpha_0[$.
\end{enumerate}
Due to the presence of the radial cut-off function, $f^\delta$ will only be a solution of the Vlasov equation if the parameters $r_0,\eta_0,L_0>0$ and $E^\delta\in]0,1[$ are chosen suitably. We will derive such a choice of parameters by analyzing the metric quantity $\mu^\delta$ induced by $f^\delta$. The equation for $\mu^\delta$ reads
\begin{equation}\label{eq:lochmudgl}
	(\mu^\delta)'(r)=\frac1{1-\frac2r\left(M+m^\delta(r)\right)}\left(\frac{M+m^\delta(r)}{r^2}+4\pi rp^\delta(r)\right),\quad r>2M,
\end{equation}
together with the boundary condition $\lim_{r\to\infty}\mu^\delta(r)=0$. Here, $m^\delta$ is given by
\begin{equation}\label{eq:lochlocalmass}
	m^\delta(r)\coloneqq4\pi\int_{2M}^rs^2\rho^\delta(s)\,ds,\quad r>2M,
\end{equation}
and $\rho^\delta\coloneqq\rho_{f^\delta}$, $p^\delta\coloneqq p_{f^\delta}$ are the density and pressure induced by $f^\delta$ via \eqref{eq:rho} and \eqref{eq:p}, respectively. Note that the latter quantities only refer to the Vlasov part of the static solution. In the case $\delta=0$ we just obtain the Schwarzschild solution, i.e., the solution of \eqref{eq:lochmudgl} is of the form
\begin{equation}\label{eq:lochmuss}
	\mu^0(r)=\frac12\ln\left(1-\frac{2M}r\right),\quad r>2M.
\end{equation}
In the pure Schwarzschild case, the {\em effective potential} is given by
\begin{equation}\label{eq:locheffpotss}
	\Psi_L^0(r)\coloneqq e^{\mu^0(r)}\sqrt{1+\frac L{r^2}} = \sqrt{1-\frac{2M}r}\sqrt{1+\frac L{r^2}}, \quad L\geq0,\,r>2M.
\end{equation}
We state some properties of $\Psi_L^0$ in the following lemma and refer the reader to~\cite[\S~19]{Chandrasekhar} or~\cite[Appendix~A]{Jab2021} for more details.
\begin{lemma}\label{lemma:lochpropertieseffpotss}
	The effective potential $\Psi_L^0$ in the pure Schwarzschild case has the following properties:
	\begin{enumerate}[label=(\alph*)]
		\item For every $L\geq0$ we have that $\lim_{r\to2M}\Psi_L^0(r)=0$ and $\lim_{r\to\infty}\Psi_L^0(r)=1$.
		\item For every $L>12M^2$ there exist two unique zeros $r_L^0>s_L^0>2M$ of $\left(\Psi_L^0\right)'$. Furthermore, $\left(\Psi_L^0\right)''(s_L^0)<0<\left(\Psi_L^0\right)''(r_L^0)$, i.e., $\Psi_L^0$ attains a strict local maximum in $s_L^0$ and a strict local minimum in $r_L^0$. Hence, $\Psi_L^0(r_L^0)< \min\left \{ 1, \Psi_L^0(s_L^0)\right \}$, and $\Psi_L^0(s_L^0)>1$ is equivalent to $L>16M^2$.
		\item For every $L>12M^2$ and $E\in]\Psi_L^0(r_L^0),\min\{1,\Psi_L^0(s_L^0)\}[$ there exist three unique radii 
		\[2M<r_0^0(E,L)<s_L^0<r_-^0(E,L)<r_L^0<r_+^0(E,L)\]
		such that
		\[ \Psi_L^0(r_0^0(E,L))=E=\Psi_L^0(r_\pm^0(E,L)). \]
		Moreover, $r_-^0(E,L)>4M$.
	\end{enumerate}
\end{lemma}
We can now specify how we choose the remaining parameters in~\eqref{eq:lochdf}. First, let $L_0>12M^2$ be arbitrary. Then, fix some intermediate parameter $E^0\in]\Psi_{L_0}^0(r_{L_0}^0),\min\{1,\Psi_{L_0}^0(s_{L_0}^0)\}[$. Next, let $r_0\coloneqq s_{L_0}^0$ and choose $\eta_0>0$ sufficiently small such that $r_0+\eta_0<r_-^0(E^0,L_0)$; we note that the resulting static solution will not depend on~$r_0$ and~$\eta_0$. 
The choice of all these parameters together with the behavior of $\Psi_L^0$ is illustrated in Figure~\ref{img:eff_pot}.

\begin{figure}[h]
	\begin{center}
		\centering
		\resizebox{!}{.72\textwidth}{\input{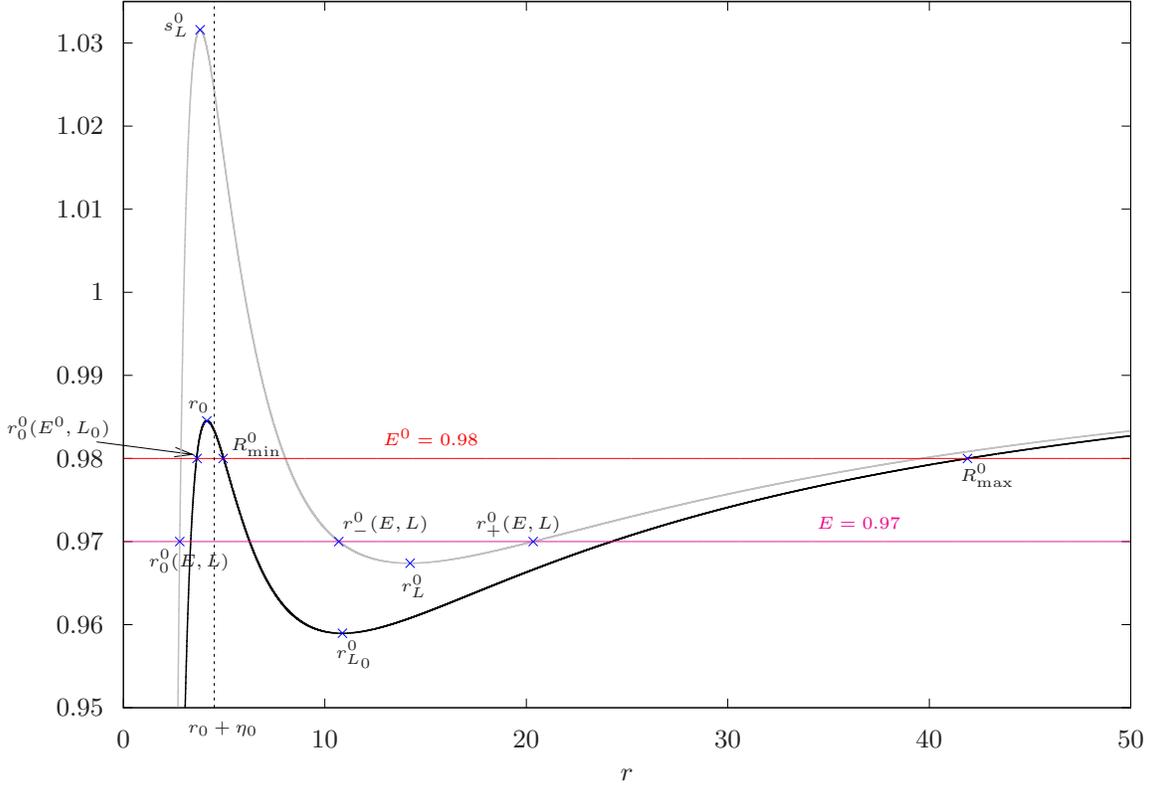}}\vspace*{-.16cm}
	\end{center}
	\vspace*{-.5cm}
	\caption{An illustration of a possible choice of parameters and the effective potential in the pure Schwarzschild case for $M=1$, $L_0=15$, and $E^0=0.98$. The black line corresponds to $\Psi_{L_0}^0$ and the grey line to $\Psi^0_{L}$ with $L=18$. 
	}
	\label{img:eff_pot}
\end{figure}

The only parameter which is still left free is the cut-off energy $E^\delta$. By considering $y^\delta\coloneqq\ln(E^\delta)-\mu^\delta$, we turn $E^\delta$ into an unknown variable and arrive at the following equation for $y^\delta$ on $]2M,\infty[$:
\begin{equation}\label{eq:lochydeltagl}
	(y^\delta)'(r)=-\frac1{1-\frac2r\left(M+m^\delta(r)\right)}\left(\frac{M+m^\delta(r)}{r^2}+4\pi rp^\delta(r)\right),\quad y^\delta(4M)=y_0,
\end{equation}
for which we choose the initial value
\begin{equation}
	y_0\coloneqq \ln\left(\sqrt2\,E^0\right).
\end{equation}
Eqn.~\eqref{eq:lochydeltagl} is a closed system for $y^\delta$ since we can express $\rho^\delta$ and $p^\delta$ in terms of $y^\delta$ by plugging~\eqref{eq:lochdf} into~\eqref{eq:rho} and~\eqref{eq:p}:
\begin{align}\label{eq:lochrhoandp}
	\rho^\delta(r) = \delta\,\chi(r-r_0)\,G(r,y^\delta(r)), \quad p^\delta(r) = \delta\,\chi(r-r_0)\,H(r,y^\delta(r)),\quad r>2M,
\end{align}
where $G$ and $H$ are given by~\eqref{eq:G_phi_singfree} and~\eqref{eq:H_phi_singfree}.
A first simple observation is that in the pure Schwarzschild case, i.e., $\delta=0$, the solution of~\eqref{eq:lochydeltagl} takes on the form
\begin{equation}
	y^0(r)=y_0-\ln\left(\sqrt2\right)-\frac12\ln\left(1-\frac{2M}r\right)=\ln\left(E^0\right)-\mu^0(r),\quad r>2M,
\end{equation}
where $\mu^0$ is defined in~\eqref{eq:lochmuss}. We next investigate the existence of a solution for positive~$\delta$:
\begin{lemma}\label{lemma:lochydeltaproperties}
	For every $\delta\geq0$ there exists a unique solution $y^\delta\in C^1(]2M,\infty[)$ of~\eqref{eq:lochydeltagl} with~\eqref{eq:lochlocalmass} and~\eqref{eq:lochrhoandp} inserted into the right hand side of the differential equation.
	In particular, $2\left(M+m^\delta(r)\right)<r$ for $r>2M$.
	Furthermore, 
	\begin{equation}\label{eq:lochydeltaestimatey0}
		y^\delta\leq y^0 \quad\text{ on } \left]2M, \infty\right[,
	\end{equation}
	as well as
	\begin{equation}\label{eq:lochydeltaequaly0}
		y^\delta(r)=y^0(r) \quad\text{ for } r\in\left]2M, \Rmin^0\right],
	\end{equation}
	where
	\begin{equation}\label{eq:lochdefRminRmax}
		\Rmin^0\coloneqq r_-^0(E^0,L_0)>4M,\quad \Rmax^0\coloneqq r_+^0(E^0,L_0).
	\end{equation}
	In addition, $\rho^\delta$ and $p^\delta$ do not vanish on the whole domain $]2M,\infty[$ if $\delta>0$, and
	\begin{equation}\label{eq:lochrhopcompactsupport}
		\rho^\delta(r)=0=p^\delta(r) \quad\text{ for } r\in\left]2M,\Rmin^0\right]\cup\left[\Rmax^0,\infty\right[.
	\end{equation}
	Lastly, $y_\infty^\delta\coloneqq \lim_{r\to\infty}y^\delta(r)$ exists with $y_\infty^\delta\in]-\infty,0[$ and the relation
	\begin{equation}\label{eq:lochyinfdelta}
		y_\infty^\delta=y^\delta(\Rmax^0)+\frac12\ln\left(1-\frac2{\Rmax^0}\left(M+M^\delta\right)\right)
	\end{equation}
	holds, where $M^\delta$ is the total Vlasov mass given by
	\begin{equation}\label{eq:lochMdeltadef}
		M^\delta\coloneqq\lim_{r\to\infty}m^\delta(r)=4\pi\int_{\Rmin^0}^{\Rmax^0}s^2\rho^\delta(s)\,ds.
	\end{equation}
\end{lemma}
\begin{proof}
	Uniqueness and local existence follow by basic ODE theory since $G,H\in C^1(]0,\infty[\times\R)$, cf.~\cite[Lemma~2.2]{ReRe00}.
	Next, observe that $y^\delta=y^0$ defines a solution of~\eqref{eq:lochydeltagl} on $]2M,\Rmin^0]$, which is due to the two following facts: First, the presence of the radial cut-off function $\chi$ in~\eqref{eq:lochrhoandp} leads to $\rho^\delta(r)=0=p^\delta(r)$ for $2M<r\leq r_0=s_{L_0}^0$. Second, $\Psi_{L_0}^0(r)\geq E^0$ for $r\in[s_{L_0}^0,\Rmin^0]$ by Lemma~\ref{lemma:lochpropertieseffpotss}, which implies that $e^{-y^0(r)}\sqrt{1+\frac{L_0}{r^2}}\geq1$. Hence, $G(r,y^0(r))=0=H(r,y^0(r))$ for $r\in[s_{L_0}^0,\Rmin^0]$. We have thus proven~\eqref{eq:lochydeltaequaly0} and that the solution of~\eqref{eq:lochydeltagl} can be extended to the left up to the radius $2M$.
	
	For a proof that the solution can be extended to arbitrarily large radii we refer to~\cite[Theorem~3.4]{Rein94}. The main difficulty is to show that the denominator in the right hand side of~\eqref{eq:lochydeltagl} does not vanish, which is achieved by using the Tolman-Oppenheimer-Volkov equation~\cite[Lemma~3.3]{Rein94}.
	
	Now let $y^\delta\in C^1(]2M,\infty[)$ be the resulting solution. Since $\rho^\delta,p^\delta\geq0$, we obtain that $(y^\delta)'(r)\leq(y^0)'(r)$ for $r>2M$. Together with~\eqref{eq:lochydeltaequaly0} this implies~\eqref{eq:lochydeltaestimatey0}. Hence, 
	\[e^{-y^\delta(r)}\sqrt{1+\frac{L_0}{r^2}}\geq e^{-y^0(r)}\sqrt{1+\frac{L_0}{r^2}}=\frac{\Psi_{L_0}^0(r)}{E^0}\geq1 \quad\text{ for }r\geq\Rmax^0,\]
	i.e., $G(r,y^\delta(r))=0=H(r,y^\delta(r))$ for $r\geq\Rmax^0$ and~\eqref{eq:lochrhopcompactsupport} is shown. Moreover, in the case $\delta>0$, $y^\delta$ cannot equal $y^0$ on the whole domain $]2M,\infty[$ since, e.g., $G(r_{L_0}^0,y^0(r_{L_0}^0))$ and $H(r_{L_0}^0,y^0(r_{L_0}^0))$ are positive; recall that $E^0>\Psi_{L_0}^0(r_{L_0}^0)$ and~\ref{it:assphi1},~\ref{it:assphi3}. This proves that $\rho^\delta$ and $p^\delta$ cannot vanish identically.
	Lastly, \eqref{eq:lochrhopcompactsupport} implies that
	\[(y^\delta)'(r)=-\frac1{1-\frac2r\left(M+M^\delta\right)}\,\frac{M+M^\delta}{r^2},\quad r\geq\Rmax^0,\]
	from which we deduce that
	\[y^\delta(r)=y^\delta(\Rmax^0)+\frac12\ln\left(1-\frac2{\Rmax^0}\left(M+M^\delta\right)\right)-\frac12\ln\left(1-\frac2r\left(M+M^\delta\right)\right),\quad r\geq\Rmax^0.\]
	We therefore obtain~\eqref{eq:lochyinfdelta}. $y_\infty^\delta<0$ is a consequence of~\eqref{eq:lochydeltaestimatey0} and $y_\infty^0=\ln(E^0)<0$.
\end{proof}

Once the solution $y^\delta$ of~\eqref{eq:lochydeltagl} is known, we can recover the original functions as follows:

\begin{proposition}\label{prop:lochexistencestst}
	Let $\chi$, $\varphi$, $l$, and $\Phi$ be as stated at the start of this subsection with $L_0$, $r_0$, $\eta_0$, and $y_0$ as specified after Lemma~\ref{lemma:lochpropertieseffpotss}. For $\delta>0$ let $y^\delta$ be the solution of~\eqref{eq:lochydeltagl} provided by Lemma~\ref{lemma:lochydeltaproperties}. Then, letting $E^\delta\coloneqq e^{y_\infty^\delta}\in]0,1[$, $\mu^\delta(r)\coloneqq\ln(E^\delta)-y^\delta(r)$, and
	\[\lambda^\delta(r)\coloneqq-\frac12\ln\left(1-\frac{8\pi}r\left(M+m^\delta(r)\right)\right),\quad r>2M, \]
	as well as defining $\rho^\delta$, $p^\delta$ by~\eqref{eq:lochrhoandp} and $f^\delta$ by~\eqref{eq:lochdf} gives a time-independent solution of the Einstein-Vlasov system with Schwarzschild-singularity of mass $M$. More precisely, the Vlasov equation~\eqref{eq:vlasov} is solved in the sense that
	\begin{equation}\label{eq:lochdfdependsonEandL}
		f^\delta(x,v)=\delta\,\varphi(E(x,v),L(x,v)) \text{ for } (x,v) \text{ such that }f^\delta(x,v)>0, 
	\end{equation}
	i.e., $f^\delta$ is constant along characteristics of the stationary Vlasov equation. The solution is compactly supported and has finite mass. 
\end{proposition}
\begin{proof}
	Due to the properties of $y^\delta$ derived in Lemma~\ref{lemma:lochydeltaproperties}, it remains to show~\eqref{eq:lochdfdependsonEandL}. 
	This might only fail for $r\in]r_0,r_0+\eta_0[$ since $f^\delta(x,v)=0$ for $2M<r\leq r_0$ and $f^\delta(x,v)=\delta\,\varphi(E(x,v),L(x,v))$ for $r\geq r_0+\eta_0$. However,~\eqref{eq:lochrhopcompactsupport} together with $r_0+\eta_0<r_-^0(E^0,L_0)=\Rmin^0$ immediately implies that $f^\delta(x,v)=0$ if $2M<r<r_0+\eta_0$.
\end{proof}

Note that if the microscopic equation of state is chosen to be sufficiently smooth, e.g., $\Phi\in C^1(\R)$ and $l>1$, then the solution obtained in the above proposition solves the Vlasov equation in the classical sense. Further assumptions on the static solution which are needed for the succeeding linear stability analysis are stated in Section~\ref{sc:ststconditions}.

\begin{remark}\label{remark:nested_shells}
	After obtaining such a static shell we can repeat the above method to construct another shell situated at larger radii. For the second shell, we have to choose $L_0$ large enough and consider $M+M^\delta$ as the interior mass.  
	Iteratively, we can obtain arbitrarily many, nested matter shells. This is analogous to the massless case~\cite{An2021}.
\end{remark}

\section{Action-angle type variables and single-well structure}\label{sc:jeans}

Spherical symmetry implies that the distribution function $f$ can be written in the form $f=f(t,r,w,L)$, where $w= \frac{x\cdot v}{r}$ and $L=|x\times v|^2$ can be interpreted as the radial and angular momentum, respectively. Integrals change via
\begin{equation}
	dx dv = 4\pi^2 \, drdwdL, \quad dv = \frac{\pi}{r^2}\, dw dL. 
\end{equation}
We now introduce action-angle type variables, which are a fundamental ingredient in our analysis. The main idea is to express a point $(r,w,L)$ in the spherically symmetric phase space in terms of the two integrals $(E,L)$ of the characteristic flow of a fixed steady state together with one angle-variable $\theta$ which determines the position along the orbit fixed by $(E,L)$. 

Before introducing action-angle type variables in Section~\ref{sc:aacoords}, we discuss a property of the underlying equilibrium which is necessary for such variables to be well-defined in Section~\ref{sc:jeans_sc} and study the characteristic flow in Section~\ref{sc:periodfunction}.

\subsection{Single-well structure}\label{sc:jeans_sc}


\begin{defnlem}\label{def:jeans}
	Let $f_0$ be a steady state of the Einstein-Vlasov system with or without a Schwarzschild-singularity as derived in Proposition~\ref{prop:lochexistencestst}  or in Section~\ref{sc:stst_classic}. Let $E_0\in]0,1[$ be the corresponding cut-off energy, $L_0\geq0$ the lower bound on the angular momentum, and $\mu_0$ and $\rho_0$ be the induced metric coefficient and density, respectively. Furthermore, for $L\geq0$ let the {\em effective potential} be given by 
	\begin{equation}\label{eq:defeffpot}
		\Psi_L(r)\coloneqq e^{\mu_0(r)}\sqrt{1+\frac L{r^2}}.
	\end{equation}
	Then the steady state is said to have {\em single-well structure} if for every $L>0$ with $L\geq L_0$ and $I_L\coloneqq\{\Psi_L<E_0\}\cap\{\rho_0>0\}\neq\emptyset$ there exists a unique radius $r_L\in I_L$ such that $\Psi_L'(r_L)=0$.
	
	In this case, we deduce the following properties for every $L$ as above:
	\begin{enumerate}[label=(\alph*)]
		\item\label{it:jeansdef1} $\Psi_L(r_L)=\min\limits_{I_L}\Psi_L$.
		\item\label{it:jeansdef2} For every $E\in]\Psi_L(r_L),E_0[$ there exist two unique radii $r_\pm(E,L)\in I_L$ such that $r_-(E,L)<r_L<r_+(E,L)$ and $\Psi_L(r_\pm(E,L))=E$.
	\end{enumerate}
\end{defnlem}
\begin{proof}
	Observe that $\{\rho_0>0\}=\{\Psi_{L_0}<E_0\}$ in the singularity-free situation and $\{\rho_0>0\}=\{\Psi_{L_0}<E_0\}\cap]s_{L_0}^0,\infty[$ in the case of a central Schwarzschild-singularity with $s_{L_0}^0$ given by Lemma~\ref{lemma:lochpropertieseffpotss}. In both settings, $\Psi_L$ thus equals $E_0$ at the boundaries of $I_L$, which yields the claimed properties.  
\end{proof}

In particular, note that $I_L$ and the interior of the radial steady state support $\{\rho_0>0\}$ are connected if the steady state has single-well structure.

We now discuss the presence of a single-well structure for the classes of steady states derived in Section~\ref{sc:stst}. We refer to~\cite[Appendix~D]{RiSa2020} for a related analysis.

\subsubsection{Singularity-free, isotropic steady states}\label{ssc:jeans_singfree}
The equilibria constructed in~\cite{Sch99} show that there are steady states of the singularity-free Einstein-Vlasov system as constructed in Section~\ref{sc:stst_classic} which do not have single-well structure. The explicit $L$-dependency of the static solutions is crucial for the results from~\cite{Sch99}, and it is an open question whether general isotropic steady states have single-well structure. We prove this statement for steady states satisfying an additional condition.

\begin{lemma}\label{lemma:jeansfortwomoverrloweronethird}
	Let $f$ be a static solution of the singularity-free Einstein-Vlasov system as constructed in Section~\ref{sc:stst_classic}. We require that the steady state is isotropic, i.e., the microscopic equation of state $\varphi=\varphi(E)$ is independent of $L$.
	If
	\begin{equation}\label{eq:twomoverrloweronethird}
		\frac{2m(r)}{r} \leq \frac 1 3,\quad r>0,
	\end{equation}
	where $m$ is defined in \eqref{eq:twomoverr}, then $f$ has single-well structure.
\end{lemma}
\begin{proof}
	Consider $L>0$ such that $I_L\neq\emptyset$. Since $\Psi_L(r)>E_0$ as $r\to0$ and as $r\to\infty$, there exists at least one zero of $\Psi_L'$. In order to show that there is no more than one zero, we first observe that $\Psi_L'(r)=0$ is equivalent to
	\begin{equation}\label{eq:effpotzerocondition}
		\frac1L=\frac1{r^3\mu'(r)}-\frac1{r^2},\quad r>0.
	\end{equation} 
	We prove that the right hand side of~\eqref{eq:effpotzerocondition} is strictly monotonic on the steady state support.
	By isotropy, the radial pressure~$p$ and tangential pressure~$q$ of the steady state are equal, cf.~\cite[p.~563]{Rein94}; we suppress the index $f$ for better readability. Using~\eqref{eq:fieldeq1}, \eqref{eq:fieldeq2}, and \eqref{eq:fieldeq4} thus yields that
	\begin{multline*}
		\left(\frac1{r^3\mu'}-\frac1{r^2}\right)'=-\frac{e^{2\lambda}}{r^3(\mu')^2}\Bigg[16\pi p+4\pi\rho+\frac m{r^3}\\
		-e^{2\lambda}\left(48\pi^2r^2p^2-16\pi^2r^2\rho p+28\pi p\frac mr-4\pi\rho\frac mr+4\frac{m^2}{r^2}\right)\Bigg].
	\end{multline*}
	Let $d(r)$ denote the term in the square brackets. The isotropy of the steady state implies that $3p\leq\rho$ and we obtain that
	\begin{equation*}
		d(r)\geq 16\pi p+4\pi\rho+\frac m{r^3}-2\left(e^{2\lambda}-1\right)\left(4\pi p+\frac m{r^3}\right),
	\end{equation*}
	where we used that $e^{-2\lambda}=1-\frac{2m}r$. Lastly, observe that~\eqref{eq:twomoverrloweronethird} is equivalent to $e^{2\lambda}-1\leq\frac12$, from which we deduce that $d$ is positive on $\{\rho>0\}$.
\end{proof}

If one replaces $\frac 13$ with $\frac89$ in the right hand side of~\eqref{eq:twomoverrloweronethird}, the resulting inequality is satisfied for every equilibrium of the Einstein-Vlasov system by the extended Buchdahl inequality shown in~\cite{An2008}.
The more restrictive inequality~\eqref{eq:twomoverrloweronethird} is true if the isotropic steady state $f$ is not too relativistic. In context of Section~\ref{sc:stst_classic}, this can be seen as follows: If one fixes an isotropic ansatz as in Section~\ref{sc:stst_classic}, then~\eqref{eq:twomoverrloweronethird} is satisfied as long as $y_0$ is not too large, cf.~\cite[Cor.~3.2]{HaRe2014}.

\subsubsection{Steady states with a Schwarzschild-singularity}\label{ssc:jeans_loch}

We now prove that suitable static solutions as constructed in Section~\ref{sc:stst_hole} have single-well structure. The first step is to investigate their behavior as the size of the Vlasov part gets smaller.

\begin{lemma}\label{lemma:lochyconvergence}
	For fixed choices of $\chi$, $l$, $\Phi$, $L_0$, $r_0$, $\eta_0$, and $E^0$ as specified in Section~\ref{sc:stst_hole} let $(f^\delta)_{\delta>0}$ be the resulting family of static solutions provided by Proposition~\ref{prop:lochexistencestst} with corresponding $\mu^\delta$, $E^\delta$, $\rho^\delta$, $p^\delta$, and $M^\delta$ for $\delta>0$. As $\delta\to0$, these quantities behave as follows:
	$\rho^\delta\to0$, $p^\delta\to0$, and $\lambda^\delta\to\lambda^0$ uniformly on $]2M,\infty[$, $M^\delta\to0$, $E^\delta\to E^0$, and
	\begin{equation}\label{eq:lochmudeltaconvergence}
		(\mu^\delta)^{(k)}\to(\mu^0)^{(k)}\text{ uniformly on }]2M,\infty[,\quad k\in\{0,1,2\}.
	\end{equation}
\end{lemma}
\begin{proof}
	We start by showing the convergence of the densities $\rho^\delta$, which, due to~\eqref{eq:lochrhopcompactsupport}, can be restricted to $[\Rmin^0,\Rmax^0]$. As $\Rmin^0>4M$ and $y^\delta$ is decreasing, we obtain that $G(r,y^\delta(r))\leq G(r,y_0)$ for $r\in[\Rmin^0,\Rmax^0]$ since $G$ is an increasing function in the $y$-component, see~\cite[Lemma~3.1]{Rein94}. Hence, there exists $C>0$ such that $G(r,y^\delta(r))\leq C$ for $r\in[\Rmin^0,\Rmax^0]$ and every $\delta>0$, from which we deduce that $\rho^\delta\leq C\delta$. 
	This immediately implies that $M^\delta\to0$; recall~\eqref{eq:lochMdeltadef}. 
	Analogous arguments are valid with $\rho^\delta$ replaced by $p^\delta$. These convergences lead to $(y^\delta)'\to(y^0)'$ uniformly on $]2M,\infty[$; recall~\eqref{eq:lochydeltaequaly0} and note that $(y^\delta)'(r)\to0$ as $r\to\infty$ uniformly in $\delta>0$ since $2M^\delta<\Rmax^0-2M$. In particular, $\lambda^\delta\to\lambda^0$. After integration, we then deduce that $y^\delta\to y^0$ uniformly on $]2M,\infty[$ using similar arguments. Thus, $E^\delta\to E^0$, and~\eqref{eq:lochmudeltaconvergence} for $k\in\{0,1\}$ follows. Lastly, the uniform convergence of the second derivative of $\mu^\delta$ can be seen by differentiating~\eqref{eq:lochydeltagl} w.r.t.~$r$ and inserting the Tolman-Oppenheimer-Volkov equation
	\[
	(p^\delta)'(r)=(y^\delta)'(r)\left(p^\delta(r)+\rho^\delta(r)\right)-\frac2r\left(p^\delta(r)-q^\delta(r)\right),\quad r>2M,
	\]
	where $q^\delta$ has the same properties as~$\rho^\delta$ and~$p^\delta$; cf.~\cite[Lemma~3.3]{Rein94} for the latter claims. 
\end{proof}

These convergences now imply the single-well structure (cf.~Definition~\ref{def:jeans}) for $0<\delta\ll1$:

\begin{proposition}\label{prop:lochjeans}
	For fixed choices of $\chi$, $l$, $\Phi$, $L_0$, $r_0$, $\eta_0$, and $E^0$ as specified in Section~\ref{sc:stst_hole} let $(f^\delta)_{\delta>0}$ be the resulting family of static solutions provided by Proposition~\ref{prop:lochexistencestst}. Then there exists $\delta_0>0$ such that $f^\delta$ has single-well structure for any $0<\delta\leq\delta_0$. 
\end{proposition}
\begin{proof}
	Let $\rho^\delta$ and $\mu^\delta$ be the density and metric coefficient induced by $f^\delta$, respectively, and let~$\Psi^\delta_L(r)$ be the corresponding effective potential given by~\eqref{eq:defeffpot} for $\delta>0$, $L>0$, and $r>2M$. Then,~\eqref{eq:lochydeltaestimatey0} and~\eqref{eq:lochrhopcompactsupport} imply that
	\[
	I_L^\delta\coloneqq\{\Psi_L^\delta<E^\delta\}\cap\{\rho^\delta>0\}\subset\{\Psi_L^0\leq E^0\}\cap[\Rmin^0,\Rmax^0]\eqqcolon J_L^0
	\]
	for $\delta>0$ and $L>0$; recall~\eqref{eq:lochdefRminRmax}. If $J_L^0\neq\emptyset$, then $r_L^0\in J_L^0$, where $r_L^0$ is defined in Lemma~\ref{lemma:lochpropertieseffpotss}. In this case, $\Psi_L^0$ behaves as follows on $J_L^0$: There exists an open interval around $r_L^0$ where $(\Psi_L^0)''$ is positive. 
	To the left of that interval, $(\Psi_L^0)'$ is negative, and to the right of this interval, $(\Psi_L^0)'$ is positive; in both of the latter regions, $(\Psi_L^0)'$ is bounded away from zero. Using the convergences~\eqref{eq:lochmudeltaconvergence} implies that $\Psi_L^\delta$ has the same properties on $J_L^0$ for $\delta\in]0,\delta_0]$ with some suitable $\delta_0>0$; note that $J_L^0$ is empty for large $L$, which is why we only have to consider a compact $L$-interval and can thus choose $\delta_0$ independent of $L$. 
\end{proof}

\subsection{The periodic particle motions and the period function}\label{sc:periodfunction}

Let $f_0$ be a steady state of the Einstein-Vlasov system with or without a Schwarzschild-singularity at the center and let $\mu_0$, $E_0$, $\lambda_0$, and $\rho_0$ be the corresponding static quantities as defined in Section~\ref{sc:stst}. In addition, we assume that the steady state has single-well structure as stated in Definition~\ref{def:jeans} and also employ the notation from that definition. 

We now analyze the characteristic flow of the steady state within (the interior of) its support
\begin{equation}\label{eq:defOmeganull}
	\Omega_0\coloneqq\{(r,w,L)\mid f_0(r,w,L) > 0\},	
\end{equation} 
which is given by the Hamiltonian-like ODE system
\begin{subequations}\label{eq:charsystem}
	\begin{align}
		\dot r&=e^{-\lambda_0(r)}\,\partial_wE(r,w,L)=e^{\mu_0(r)-\lambda_0(r)}\,\frac w{\sqrt{1+w^2+\frac L{r^2}}},\label{eq:charsystemr}\\
		\dot w&=-e^{-\lambda_0(r)}\,\partial_rE(r,w,L)=e^{\mu_0(r)-\lambda_0(r)}\left(\frac L{r^3\sqrt{1+w^2+\frac L{r^2}}}-\mu_0'(r)\,\sqrt{1+w^2+\frac L{r^2}}\right).\label{eq:charsystemw}
	\end{align}
\end{subequations}
The angular momentum $L\geq0$ can be interpreted as a parameter of this system, while the particle energy $E$ given by~\eqref{eq:ststparticleenergy} is clearly a conserved quantity. For fixed $(r,w,L)\in\Omega_0$ let $(R,W)=(R,W)(\cdot,r,w,L)$ be the unique maximal solution of \eqref{eq:charsystem} with parameter~$L$ satisfying the initial condition $(R,W)(0,r,w,L)=(r,w)$. We restrict ourselves to the case $L>0$ as $\Omega_0\cap\{L=0\}$ forms a set of measure zero which will not be of importance later on. 
Now observe that the radial component $R$ of the solution always stays in the interval~$I_L$ since $\Psi_L(r)\leq E(r,w,L)$; in particular, $I_L\neq\emptyset$, cf.~Definition~\ref{def:jeans}. If $(r,w)=(r_L,0)$, the solution is constant with energy $E=\Psi_L(r_L)$. In any other case, the solution is non-constant, bounded, exists on $\R$, and is time-periodic with orbit $\{(\tilde r,\tilde w)\mid E(\tilde r,\tilde w,L)=E(r,w,L),~\rho_0(\tilde r)>0\}$; note that the latter set is bounded and connected, because of the single-well structure of the steady state, and does not contain any stationary solution of~\eqref{eq:charsystem} since $E(r,w,L)>\Psi_L(r_L)$ and $\Psi_L'(s)\neq0$ for $s\in I_L\setminus\{r_L\}$ by Definition~\ref{def:jeans}. In particular, the orbit of any such solution is uniquely determined by $L$ and its energy value $E=E(r,w,L)$.

An important quantity in the context of action-angle variables is the period of the solutions, which can be derived by applying the inverse function theorem, see~\cite[\S3.1]{BiTr}.

\begin{defnlem}\label{def:periodfunction}
	Let $f_0$ be a steady state with single-well structure and define
	\begin{equation}\label{eq:defOmeganulltildeEL}
		\tilde\Omega_0^{EL}\coloneqq \{(E(r,w,L),L)\mid(r,w,L)\in\Omega_0,\,L>0,\,E(r,w,L)>\Psi_L(r_L)\}.
	\end{equation}
	For $(E,L)\in\tilde\Omega_0^{EL}$ let $(R,W)=(R,W)(\cdot,E,L)\colon\R\to I_L\times\R$ be the maximal solution of~\eqref{eq:charsystem} with parameter~$L$ satisfying the initial condition $(R,W)(0,E,L)=(r_-(E,L),0)$, where $r_L$ and $r_\pm(E,L)$ are defined in Definition~\ref{def:jeans}. Then the solution is time-periodic with period 
	\begin{equation}\label{eq:periodfunctiondef}
		T(E,L)\coloneqq2\int_{r_-(E,L)}^{r_+(E,L)}e^{\lambda_0(r)-\mu_0(r)}\frac E{\sqrt{E^2-\Psi_L^2(r)}}\,dr.
	\end{equation}
	The induced function $T\colon\tilde\Omega_0^{EL}\to]0,\infty[$ is called the {\em (radial) period function}.
\end{defnlem}

As explained above, any solution $(R,W)=(R,W)(\cdot,r,w,L)$ of the characteristic system~\eqref{eq:charsystem} with parameter $L$ and initial condition $(R,W)(0,r,w,L)=(r,w)$ for $(E(r,w,L),L)\in\tilde\Omega_0^{EL}$ is time-periodic with period $T(E(r,w,L),L)$. Moreover, the integral in~\eqref{eq:periodfunctiondef} is well-defined and finite since $\Psi_L<E$ on $]r_-(E,L),r_+(E,L)[$ and $\Psi_L'(r_\pm(E,L))\neq0$ by the single-well structure as stated in Definition~\ref{def:jeans}. Note that the characteristic motions are only periodic in the coordinates $(r,w,L)$ adapted to spherically symmetry, but are not necessarily periodic in Cartesian coordinates $(x,v)$.

As in the non-relativistic setting~\cite{HaReSt21,Kunze}, the radial period function $T$ and its properties are crucial to derive a Birman-Schwinger principle. We next show that this function is bounded and bounded away from zero on the steady state support for sufficiently small shells around a Schwarzschild black hole.

\begin{lemma}\label{lemma:lochperiodsbounded}
	Let $f^\delta$ be a fixed steady state of the Einstein-Vlasov system with a Schwarzschild-singularity at the center with $0<\delta\leq\delta_0$ as in Proposition~\ref{prop:lochjeans}. Then the corresponding period function $T$ is bounded and bounded away from zero on $\tilde\Omega_0^{EL}$.
\end{lemma}
\begin{proof}
	Let $\mu_0$, $E_0$, $\lambda_0$, $\Psi_L$, $\rho_0$, etc., be the quantities induced by the steady state $f_0=f^\delta$ as described in Proposition~\ref{prop:lochexistencestst} and Definition~\ref{def:jeans}.
	First, note that $E\in]\Psi_{L_0}(r_{L_0}),E_0[$ for $(E,L)\in\tilde\Omega_0^{EL}$ and that $\lambda_0$ and $\mu_0$ are bounded on the compact radial steady state support, which is why it suffices to prove that the function defined by
	\begin{equation*}
		S(E,L)\coloneqq\int_{r_-(E,L)}^{r_+(E,L)}\frac {dr}{\sqrt{E-\Psi_L(r)}}, \quad (E,L)\in\tilde\Omega_0^{EL},
	\end{equation*}
	is bounded and bounded away from zero on $\tilde\Omega_0^{EL}$. We start by establishing the lower bound by proceeding similarly as in the proof of~\cite[Lemma~B.4]{HaReSt21}. More precisely, for every $(E,L)\in\tilde\Omega_0^{EL}$ we obtain that
	\begin{equation}\label{eq:TboundedSlowerbound}
		S(E,L)\geq\frac{r_+(E,L)-r_L}{\sqrt{E-\Psi_L(r_L)}}=\frac{\sqrt2}{\sqrt{\Psi_L''(s)}}
	\end{equation}
	for some $s\in ]r_L,r_+(E,L)[$ with $\Psi_L''(s)>0$ by
	a second-order Taylor expansion. Since $\Psi_L''(s)$ is bounded for $s\in[\Rmin^0,\Rmax^0]$ and~$L$ is within the bounded range where $I_L\neq\emptyset$, the estimate~\eqref{eq:TboundedSlowerbound} yields that $\inf_{\tilde\Omega_0^{EL}}(S)>0$.
	
	In order to establish the upper bound recall that by the construction of $\delta_0$ in the proof of Proposition~\ref{prop:lochjeans} there exist $a,q>0$ such that for every $L\geq L_0$ with $I_L\neq\emptyset$ we have that $\Psi_L''\geq a$ on $[r_L-q,r_L+q]\cap\bar I_L$ as well as $\Psi_L'\leq-a$ on $]2M,r_L-q]\cap\bar I_L$ and $\Psi_L'\geq a$ on $[r_L+q,\infty[\cap\bar I_L$; note that the latter two domains are empty if $\Psi_L(r_L)$ gets close to $E_0$. We now split the integral $S(E,L)$ as follows:
	\begin{align*}
		S(E,L)&=\int_{r_-(E,L)}^{r_L-q}\frac{dr}{\sqrt{E-\Psi_L(r)}}+\int_{r_L-q}^{r_L+q}\frac{dr}{\sqrt{E-\Psi_L(r)}} + \int_{r_L+q}^{r_+(E,L)}\frac{dr}{\sqrt{E-\Psi_L(r)}}\\
		&\eqqcolon S_l(E,L)+S_m(E,L)+S_r(E,L),
	\end{align*}
	and estimate each part separately. Here, we leave out $S_l$ or $S_r$ if their domain of integration is empty, e.g., the $S_l$ term is present iff $r_-(E,L)\leq r_L-q$. When $S_l$ is present, $\Psi_L'\leq-a$ on $[r_-(E,L),r_L-q]\neq\emptyset$, which yields that
	\[
	S_l(E,L)\leq\frac1{\sqrt a}\int_{r_-(E,L)}^{r_L-q}\frac{dr}{\sqrt{r-r_-(E,L)}}\leq\frac2{\sqrt a}\sqrt{\Rmax^0}
	 \]
	by the mean value theorem. $S_r$ can be estimated in the same manner. This leaves us with $S_m$, which we split into the two parts
	\[
	S_m(E,L)=\int_{r_L-q}^{r_L}\frac{dr}{\sqrt{E-\Psi_L(r)}}+\int_{r_L}^{r_L+q}\frac{dr}{\sqrt{E-\Psi_L(r)}}\eqqcolon S_m^-(E,L)+S_m^+(E,L).
	\]
	Here, we again assumed that $r_L-q\geq r_-(E,L)$ and $r_L+q\leq r_+(E,L)$; otherwise replace $r_L\pm q$ by $r_\pm(E,L)$. The estimates we employ for $S_m^\pm$ are related to those used in the proof of~\cite[Prop.~2.8]{HaReSt21} and we will illustrate these techniques for $S_m^+$. First, changing variables via $\eta=\Psi_L(r)$ yields that
	\[
	S_m^+(E,L)=\int_{\Psi_L(r_L)}^{\Psi_L(r_L+q)}\frac{d\eta}{\sqrt{\left(\Psi_L'(r_+(\eta,L))\right)^2(E-\eta)}};
	\]
	note that $\Psi_L'(r)>0$ for $r_L<r<r_+(E,L)$. By the extended mean value theorem, for every $\eta\in]\Psi_L(r_L),\Psi_L(r_L+q)[$ there exists some $s\in]r_L,r_L+q[$ such that 
	\[
	\frac{\left(\Psi_L'(r_+(\eta,L))\right)^2}{\eta-\Psi_L(r_L)}=2 \Psi_L''(s)\geq2a;
	\]
	recall that $\Psi_L'(r_L)=0$. Hence,
	\[
	S_m^+(E,L)\leq\frac1{\sqrt{2a}}\int_{\Psi_L(r_L)}^E\frac{d\eta}{\sqrt{(\eta-\Psi_L(r_L))(E-\eta) }}=\frac\pi{\sqrt{2a}}.
	\]
	A similar calculation shows that $S_m^-(E,L)$ is bounded independently of $(E,L)$ as well. 
\end{proof}

The same result is true for singularity-free, isotropic steady states which are not too relativistic.

\begin{lemma}\label{lemma:Tboundedfortwomoverrloweronethird}
	Let $f$ be a steady state of the singularity-free Einstein-Vlasov system as in Lemma~\ref{lemma:jeansfortwomoverrloweronethird}, i.e., $f$ is isotropic and satisfies the inequality~\eqref{eq:twomoverrloweronethird}. Then the corresponding period function~$T$ is bounded and bounded away from zero on $\tilde\Omega_0^{EL}$.
\end{lemma}
\begin{proof} We suppress the index $f$ for better readability. The lower bound on $T$ can be established similarly to Lemma~\ref{lemma:lochperiodsbounded}. The non-trivial part is to prove that $\Psi_L''(s)$ with $s\in]r_L,r_+(E,L)[$ is bounded, in particular as $L\to0$. This can be done by explicitly computing the second-order radial derivative of the effective potential and observing that $\frac L{r_L^4}$ is bounded as $L\to0$ since $\Psi_L'(r_L)=0$ implies that
	\begin{equation}
		\frac L{r_L^4}=\frac{\mu'(r_L)}{r_L}\left(1+\frac L{r_L^2}\right)\leq e^{2\lambda(r_L)}\left(4\pi p(r_L)+\frac{m(r_L)}{r_L^3}\right)E_0^2e^{-2\mu(r_L)}.
	\end{equation}
	Establishing the upper bound is more involved. It is achieved similarly as for the Vlasov-Poisson system in~\cite[Sc.~B.1]{HaReSt21}. The substitute for the Poisson equation in the Newtonian setting is the following elliptic equation
	\begin{multline}\label{eq:ellipticeffpot}
		\Delta_x(\Psi_L^2)=\frac1{r^2}\left(r^2(\Psi_L^2)'\right)'\\ =2e^{2\mu+2\lambda}\left[4\pi(\rho+3p)+4\pi r(p+\rho)\mu'+\frac L{r^2}\left(4\pi r(p+\rho)\mu'+4\pi(\rho-p)+\frac{e^{-2\lambda}}{r^2}-4\frac m{r^3}\right)\right],
	\end{multline}
	which can be derived using~\eqref{eq:fieldeq1}, \eqref{eq:fieldeq2}, \eqref{eq:fieldeq4}, and $p=q$. The crucial observation is that the right hand side of~\eqref{eq:ellipticeffpot} is positive on $\{\rho>0\}$ since $e^{-2\lambda}=1-\frac{2m}r$ by~\eqref{eq:twomoverr}, $\frac{2m}r\leq\frac13$ by~\eqref{eq:twomoverrloweronethird}, and $3p\leq\rho$ by isotropy.
	
	The boundedness of~$T$ on $\tilde\Omega_0^{EL}\cap\{L\geq L_1\}$ for any $L_1>0$ then follows as in Lemma~\ref{lemma:lochperiodsbounded} since $(\Psi_L^2)''(r_L)>0$ by~\eqref{eq:ellipticeffpot}. 
	Similar to~\cite[Lemma~B.2]{HaReSt21}, the maximum principle for~\eqref{eq:ellipticeffpot} yields that~$T$ is bounded on $\tilde\Omega_0^{EL}\cap\{E\leq E_1\}$ for any $E_1<E_0$ since the right hand side of~\eqref{eq:ellipticeffpot} is bounded away from zero on orbits corresponding to an energy value $E\leq E_1$.
	The remaining gap can be closed analogously to~\cite[Lemma~B.3]{HaReSt21}; observe that~\eqref{eq:ellipticeffpot} implies that $\Psi_L^2$ is convex on $]r_-(E_0,L),r_L[$.
%
%
\end{proof}

An alternate way to show the boundedness of $T$ is to extend it continuously onto the boundary of $\tilde\Omega_0^{EL}$ using the techniques from~\cite[Theorem~3.13]{Kunze}.

\subsection{Action-angle type variables}\label{sc:aacoords}

We come back to the original goal of this section---introducing action-angle type variables similar to~\cite[Sc.~5.1]{HaReSt21}. As in the previous subsection, let $f_0$ be a steady state of the Einstein-Vlasov system with or without a Schwarzschild-singularity at the center which has single-well structure as stated in Definition~\ref{def:jeans}. For $(E,L)\in\tilde\Omega_0^{EL}$ let $(R,W)=(R,W)(\cdot,E,L)$ be the unique global solution to the characteristic system~\eqref{eq:charsystem} with parameter~$L$ satisfying the initial condition $(R,W)(0,E,L)=(r_-(E,L),0)$. As shown in Lemma~\ref{def:periodfunction}, this solution is time-periodic with period $T(E,L)\in]0,\infty[$ and orbit $\{(r,w)\mid E(r,w,L)=E,~\rho_0(r)>0\}$. Now let
\begin{equation}
	\tilde\Omega_0\coloneqq\{(r,w,L)\in\Omega_0\mid (E(r,w,L),L)\in\tilde\Omega_0^{EL}\},
\end{equation}
where $\Omega_0$ and $\tilde\Omega_0^{EL}$ are defined in~\eqref{eq:defOmeganull} and~\eqref{eq:defOmeganulltildeEL}, respectively.
The action-angle type variables are then given by expressing some $(r,w,L)\in\tilde\Omega_0$ as
\begin{equation*}
	(r,w,L)=\left((R,W)(\theta\,T(E,L),E,L),L\right)
\end{equation*}
with $E=E(r,w,L)$ and suitable $\theta\in[0,1[$. Similar to Lemma~\ref{def:periodfunction}, $\theta$ is explicitly given by
\begin{equation}\label{eq:defthetarEL}
	\theta(r,E,L)\coloneqq\frac1{T(E,L)}\int_{r_-(E,L)}^re^{\lambda_0(s)-\mu_0(s)}\frac E{\sqrt{E^2-\Psi_L^2(s)}}\,ds \in\left[0,\frac12\right]
\end{equation}
if $w\geq0$, and $1-\theta(r,E,L)$ if $w<0$. In particular, the change of variables
\begin{equation}\label{eq:AAvariables}
	\tilde\Omega_0\ni(r,w,L)\mapsto (\theta,E,L)\in[0,1[\times\tilde\Omega_0^{EL}
\end{equation}
defined in this way is one-to-one. The interpretation of these new variables is that $(E,L)$---the \enquote{actions}---fix an orbit of the characteristic flow of the steady state and $\theta$---the \enquote{angle}---determines the position along this orbit. However, in comparison to \enquote{true} action-angle variables~\cite{Arnold,LaLi,LB1994}, our change of variables is not volume preserving as integrals change via  
\begin{equation}\label{eq:aacoordsvolumeelements}
	drdwdL = e^{-\lambda_0}\,T(E,L)\,d\theta dEdL;
\end{equation}
this is the reason why we refer to $(\theta,E,L)$ as action-angle {\em type} variables.
Lastly, note that the sets $\Omega_0$ and $\tilde\Omega_0$ as well as the sets 
\begin{equation}
	\Omega_0^{EL}\coloneqq\{(E(r,w,L),L)\mid(r,w,L)\in\Omega_0\}
\end{equation} 
and $\tilde\Omega_0^{EL}$ are equal up to sets of measure zero, respectively, which is why for the succeeding analysis it suffices to establish the change of variables on the smaller sets.

\section{Linearization of the Einstein-Vlasov system and the linearized operator}\label{sc:linearization}

In this section we introduce the linearized Einstein-Vlasov system which we use to investigate linear stability.

\subsection{The steady states under consideration}\label{sc:ststconditions}
We start by rigorously stating the class of steady states which we analyze in the following. 
Let $f_0$ be a static solution of the Einstein-Vlasov system as constructed in Section~\ref{sc:stst} with corresponding metric quantities $\lambda_0$ and $\mu_0$. If there is a Schwarzschild-singularity at the center (cf.~Section~\ref{sc:stst_hole}), we denote its mass by $M>0$; the case $M=0$ stands for the singularity-free situation (cf.~Section~\ref{sc:stst_classic}).
In both cases, $f_0$ is of the form
\begin{equation*}
	f_0(r,w,L)=\varphi(E(r,w,L),L),\quad (r,w,L)\in\Omega_0,
\end{equation*}
for some appropriate microscopic equation of state $\varphi\colon\R^2\to[0,\infty[$, where $E$ is the particle energy induced by $\mu_0$ via~\eqref{eq:ststparticleenergy} and $\Omega_0$ is the interior of the steady state support defined in~\eqref{eq:defOmeganull}. 
In addition, let $\Rmin\coloneqq\inf\{r\mid(r,w,L)\in\Omega_0\}$ and $\Rmax \coloneqq \sup \{r \mid (r,w,L)\in \Omega_0\}$ be the radial bounds of the steady state. 
Note that $\varphi$  includes a $\delta$-dependency in the situation of Section~\ref{sc:stst_hole}.
We further impose the following conditions on $f_0$:
\begin{enumerate}[label=(S\arabic*)]
	\item\label{it:stst1} The steady state has single-well structure as defined in Definition~\ref{def:jeans}.
	\item\label{it:stst2} The radial period function $T$ defined in Definition~\ref{def:periodfunction} is bounded and bounded away from zero on $\tilde\Omega_0^{EL}$; the latter set is defined in~\eqref{eq:defOmeganulltildeEL}.
	\item\label{it:stst3} The microscopic equation of state $\varphi$ is continuously differentiable with respect to $E$ on $\tilde\Omega_0^{EL}$ with $\varphi'\coloneqq\partial_E\varphi<0$ on $\tilde\Omega_0^{EL}$. On $\R^2\setminus\tilde\Omega_0^{EL}$ we set $\varphi'\coloneqq0$.
	\item\label{it:stst4} There exists $C>0$ such that 
	\begin{equation}\label{eq:phi_prime_bound}
	\int_{\R^3} |\varphi'(E,L)| \,dv =	\frac{\pi}{r^2}\int_0^{\infty} \int_\R |\varphi'(E, L)| \, dw dL \leq C, \quad r \in ]\Rmin,\infty[.  
	\end{equation}
\end{enumerate}

\begin{remark}\phantomsection\label{remark:single_well_numerics}
	\begin{enumerate}[label=(\alph*)]
		\item In Lemmas~\ref{lemma:jeansfortwomoverrloweronethird} and~\ref{lemma:Tboundedfortwomoverrloweronethird} we have proven that the conditions~\ref{it:stst1} and~\ref{it:stst2} are satisfied for isotropic steady states provided that they are not too relativistic. We emphasize, however, that numerical simulations clearly indicate that they are true for a much larger class of static solutions. For example, in the isotropic case this always seems to be the case; we have verified this numerically, e.g., for polytropes, the King model, and ansatzfunctions as used in \cite{GueStRe21}.	
		\item In the case of steady states with a Schwarzschild-singularity at the center as constructed in Section~\ref{sc:stst_hole}, the validity of conditions~\ref{it:stst1} and~\ref{it:stst2} has been shown in Proposition~\ref{prop:lochjeans} and Lemma~\ref{lemma:lochperiodsbounded}, respectively, if the self-consistent part is sufficiently small compared to the black hole. Numerical evidence points towards the validity~\ref{it:stst1} and~\ref{it:stst2} even for large values of $\delta$, at least for the ansatz functions we have employed.  However, this does not hold in full generality since multi-shells exist as commented on in Remark~\ref{remark:nested_shells}.
		\item \label{remark:S4condition}	Conditions~\ref{it:stst3} and~\ref{it:stst4} are satisfied if the energy-dependency $\Phi$ for the steady states from Section~\ref{sc:stst} is chosen suitably. The technical assumption~\ref{it:stst4} is, e.g., true if $\Phi'$ is bounded. In addition, \ref{it:stst4} can be verified for the steady states constructed in Section~\ref{sc:stst_hole} by explicitly calculating the integral over $L$ since $\Phi \in L^\infty_{\mathrm{loc}}(\R)$ and $\Rmin > 0$.
	\end{enumerate}
 \end{remark}
In both the singularity-free and Schwarzschild-singularity case, condition~\ref{it:stst1} is mandatory in order to introduce action-angle type variables as in Section~\ref{sc:aacoords} which are crucial for the following analysis. Note that~\ref{it:stst3} together with the assumption~\ref{it:assphi2} or~\ref{it:assphi3} on $\Phi$ causes $\Omega_0$ to be open.

\subsection{The first-order linearized system}\label{sc:firstorder_lin}

We linearize the system as in \cite{HaLiRe2020,HaRe2013,HaRe2014,IT68}, i.e., for $0<\varepsilon\ll1$ we plug $f_0 + \varepsilon f + \mathcal O(\varepsilon^2)$ into the Einstein-Vlasov system and dispense with terms of order $\mathcal O(\varepsilon^2)$. We omit the details of this calculation but point to the references above for more details. We arrive at the following linearized system. The linearized Vlasov equation reads
\begin{equation}\label{eq:vlasov_lin}
	\partial_t f = - e^{-\lambda_0} \{f,E\} + 4\pi r |\varphi'| e^{3\mu_0+\lambda_0} \frac{w^2}{E} j_f - e^{2\mu_0-\lambda_0} |\varphi'| w \mu_f',
\end{equation} 
where  $\{g,h\} \coloneqq \partial_x g \cdot \partial_v h - \partial_v g\cdot \partial_x h= \partial_r g\, \partial_w h - \partial_w g\, \partial_r h$ is the Poisson bracket of two differentiable functions $g(x,v)=g(r,w,L)$ and $h(x,v)=h(r,w,L)$. The linearized field equations are given by 
\begin{align}
			(re^{-2\lambda_0} \lambda_f )' = 4\pi r^2\rho_f	\label{eq:fieldeq1_lin} \\
			\mu_f' = 4\pi r e^{2\lambda_0} p_f +\left  (2\mu_0'+\frac{1}{r}\right ) \lambda_f \label{eq:fieldeq2_lin}
\end{align}
and the source terms are the same as in \eqref{eq:rho}--\eqref{eq:j}. As in the non-linear case we prescribe
\begin{equation}\label{eq:asym_flat_lin}
	\lim_{r\to\infty}\mu_f(t,r) = 0  = 	\lim_{r\to\infty}\lambda_f(t,r).
\end{equation}
In the singularity-free case we impose
\begin{equation}\label{eq:sing_free_boundary_lin}
	\lambda_f(0) = 0,
\end{equation}
in order to obtain a regular center, while in the setting with a Schwarzschild-singularity of mass $M$ the corresponding boundary condition is
\begin{equation}\label{eq:blackhole_boundary_lin}
	\lambda_f(4M) = 0;
\end{equation}
recall that the radial support $f_0$ is compactly contained in $[4M, \infty[$, see Lemma~\ref{lemma:lochydeltaproperties}, i.e., it suffices to impose a boundary condition at $r =4M$. In particular, integrating~\eqref{eq:fieldeq1_lin} yields that
\begin{equation}\label{eq:lambdaeqexplicit}
	\lambda_f (r) = \frac{4\pi e^{2\lambda_0(r)}}{r}\int_{\Rmin}^r \rho_f(s) s^2 \, ds, \quad r\in [4M,\infty[,
\end{equation}
for $M\geq 0$, where $M=0$ represents the singularity-free case. In summary, \eqref{eq:vlasov_lin}--\eqref{eq:sing_free_boundary_lin} constitutes the linearized, singularity-free Einstein-Vlasov system on $[0,\infty[\times \R \times [0,\infty[$ while \eqref{eq:vlasov_lin}--\eqref{eq:asym_flat_lin}, \eqref{eq:blackhole_boundary_lin} is referred to as the linearized Einstein-Vlasov system with a Schwarzschild-singularity of mass $M$ on  $[4M,\infty[\times \R \times [0,\infty[$. 

The theory of global in-time solutions to the linearized, singularity-free case was dealt with in \cite[Theorem~5.1]{HaRe2013}. This result can be translated to the case of a Schwarzschild-singularity as well. Global existence is interesting in itself, but not needed for our work, since we study the stability via spectral analysis. 

For linear stability considerations it is more convenient to write the linearized system as a second-order system in time. For this, several operators are needed in order to keep notation short. We will define them in the next subsection.
 
\subsection{Definition of the function spaces, operators, and linear stability}\label{sc:def_operators}

The setup that follows is similar to the one used in \cite{HaLiRe2020, HaReSt21}. The operators will be defined on the weighted $L^2$-space 
\begin{equation}\label{eq:defHilbertspace}
	 H \coloneqq \left \{f\colon \Omega_0 \to \R \text{ measurable} \;\Big|\; \|f\|_H < \infty \right  \},  
\end{equation}
where the norm is given by
 \[ 
 	 \|f\|_H^2\coloneqq 4\pi^2 \iiint_{\Omega_0} \frac{e^{\lambda_0(r)}}{|\varphi'(E,L)|}\, |f(r,w,L)|^2 \, dr dw dL;
 \]
recall that $\varphi'<0$ almost everywhere on $\Omega_0^{EL}$.
When it is clear that $E$ has to be interpreted as a function of $(r,w,L)$, we will not always write this dependence explicitly. Together with the associated scalar product
\[ 
	\langle f,g\rangle_{H} \coloneqq  4\pi^2 \iiint_{\Omega_0} \frac{e^{\lambda_0(r)}}{|\varphi'(E,L)|}\,f(r,w,L)\, g(r,w,L) \, dr dw dL, \quad g,h\in H,
\]
we obtain the real Hilbert space $(H,\langle \cdot,\cdot\rangle_H)$ by identifying functions which are equal almost everywhere ({\em a.e.}) as usual. We split $f\in H$ into its odd-in-$w$ part~$f_-$ and even-in-$w$ part~$f_+$ given by
\[f_\pm(r,w,L)=\frac12\left(f(r,w,L)\pm f(r,-w,L)\right),\quad\text{for a.e. }(r,w,L)\in\Omega_0,\]
i.e., $f = f_+ + f_-$ with $f_\pm(r,w,L)=\pm f_\pm(r,-w,L)$ for a.e.\ $(r,w,L) \in \Omega_0$; note that $\Omega_0$ is symmetric with respect to $w$ since $f_0$ is even in $w$. We define the subspace of $H$ consisting of odd-in-$w$ functions as
\[ 
	\H \coloneqq \{f\in H \, |\, f \text{ is odd in $w$ a.e.\ on } \Omega_0 \}.
\]
Similar to \cite[Remark~5.3]{HaReSt21}, parity in $w$ of $f\in H$ can be translated into action-angle variables---introduced in Section~\ref{sc:aacoords}---as follows:
\begin{align}
	f \text{ is even in $w$ a.e.} \,\, &\Leftrightarrow  \,\, f(\theta,E,L)=f(1-\theta,E,L) \text{ for a.e.\ } (\theta,E,L) \in ]0,1[\times\tilde\Omega^{EL}_0, \label{eq:even_in_w_theta} \\
	f \text{ is odd in $w$ a.e.}  \,\, &\Leftrightarrow  \,\, f(\theta,E,L)=-f(1-\theta,E,L)\text{ for a.e.\ } (\theta,E,L) \in ]0,1[\times \tilde\Omega^{EL}_0; \label{eq:odd_in_w_theta}
\end{align}
by a slight abuse of notation we do not distinguish between the function $f$ depending on the variables $(r,w,L)$ or on $(\theta,E,L)$.

An important quantity in the context of linear stability is the transport operator associated with the characteristic flow of the steady state. For a smooth function $f\in C^1(\Omega_0)$ it is given by
\begin{align}
	\T f \coloneqq& -e^{-\lambda_0}\{f,E\} \notag  \\
	=& -e^{\mu_0-\lambda_0} \left (\partial_r f \,\frac{w }{\sqrt{1+w^2+\frac{L}{r^2}}} - \partial_w f \,\left ( \mu_0'\,\sqrt{1+w^2+\frac{L}{r^2}} - \frac{L}{r^3 \sqrt{1+w^2+\frac{L}{r^2}}} \right )\right ). \label{eq:Tdef}
\end{align} 
We now extend this definition to a weak sense similar to~\cite[Definition~2.1]{ReSt20}. Furthermore, we introduce a related operator~$\B$ as in~\cite[Definition~4.11]{HaLiRe2020} which will be a crucial operator which we have to handle later. 
\begin{definition}\phantomsection\label{def:TandB}
\begin{enumerate}[label=(\alph*)]
	\item\label{it:TBdef1} For a function $f\in H$ the {\em transport term} $\T f$ {\em exists weakly} if there exists some $h \in H$ such that for every test function $\xi \in C^1_{c}(\Omega_0)$,
	\[ 
		\langle f, \T \xi \rangle_H = - \langle h, \xi \rangle_H. 
	\]
	In this case, we set $\T f = h $ in a weak sense. The domain of $\T$ is defined as 
	\[ 
		\mathrm D(\T) \coloneqq \{f \in H \, | \, \T f \text{ exists weakly}\},
	\]
	and the resulting operator $\T\colon \mathrm D(\T)\to H$ is called the {\em transport operator}. 
	\item\label{it:TBdef2} The operator $\B\colon \mathrm D(\T) \to H$ is defined by
	\[ 
		\B f \coloneqq \T f - 4\pi r |\varphi'| e^{2\mu_0+\lambda_0} \left (wp_f - \frac{w^2}{\sqrt{1+w^2+\frac{L}{r^2}}}\,j_f \right ).
	\]
	\item\label{it:TBdef3} The {\em residual operator} $\mathcal{R}\colon H \to H$ is defined by 
	\[ 
		\mathcal{R}f \coloneqq 4\pi |\varphi'|\,e^{3\mu_0} (2r\mu_0'+1)wj_f.
	\]
	\item\label{it:TBdef4} The {\em Antonov operator} $\L \colon \mathrm D(\L) \cap \H \to \H$ is defined on
	\[ 
		\mathrm D(\L)\coloneqq \mathrm D(\T^2) \coloneqq \{f \in H \, |\, f \in \mathrm D(\T), \, \T f \in \mathrm D(\T) \}
	\]
	and given by
	\[ 
		\L\coloneqq -\B^2 - \Ri.
	\]
\end{enumerate}
\end{definition}
We prove that the operators are well defined in the next section; for $\T$ in the weak sense we refer to~\cite[Remark~2]{ReSt20}. In order to obtain a second-order formulation of the linearized Einstein-Vlasov system, we split $f$ into its even-in-$w$ part and odd-in-$w$ part. This method is due to Antonov who first used it in the context of non-relativistic galactic dynamics~\cite{An1961}; for the Einstein-Vlasov system it was used in~\cite{IT68} and~\cite{HaLiRe2020}.
\begin{lemma}[\cite{HaLiRe2020}, Lemma~4.21]
	A formal linearization of the spherically symmetric Einstein-Vlasov system takes the form
	\begin{equation}\label{eq:second_order_lin}
		\partial^2_{t} f_- + \L f_- = 0
	\end{equation}
	where $\L$ is the Antonov operator. 
\end{lemma}
Note that, as $\L$ covers the evolution of the odd-in-$w$ part of the linear perturbation only, we have defined it only on the subspace of odd-in-$w$ functions $\H$.  We now define what we actually mean by linear stability: 
\begin{definition}\label{def:linear_stab}
	A steady state of the Einstein-Vlasov system as specified in Section~\ref{sc:ststconditions} is called {\em linearly stable} if the spectrum of~$\L$ is positive, i.e., 
	\[ 
		\inf(\sigma(\L)) > 0. 
	\]
	The number of linearly independent eigenfunctions corresponding to negative eigenvalues is called the \emph{number of exponentially growing modes} of the steady state. If zero is an eigenvalue of $\L$ we say that the steady state has a \emph{zero-frequency mode}.
\end{definition}
We will later show that $\L$ is self-adjoint which implies that $\L$ has real spectrum. Let us comment on why we choose the terminology above: 
\begin{remark}\phantomsection\label{remark:linear_stab}
	\begin{enumerate}[label=(\alph*)]
		\item If $\gamma = \inf(\sigma(\L)) >0$, by \cite[Prop.~5.12]{HiSi} the Antonov-type inequality
		\begin{equation}\label{AntIneqforL}
			\langle f, \L f\rangle_H \geq \gamma \|f\|_H^2,\
			f\in \mathrm D (\L)
		\end{equation}
		holds. On the other hand, it was shown in \cite{HaLiRe2020, IT68} that the energy
		\begin{equation}\label{AntIneq}
			\|\partial_t f_-\|_H^2 + \langle f_-, {\cal L} f_-\rangle_H
		\end{equation}
		is conserved along solutions of the linearized equation	$\partial_t^2 f_- + {\cal L} f_- = 0$. This implies linear	stability in the corresponding norm. But \eqref{AntIneqforL} is also a natural first step towards non-linear stability of the steady state $f_0$, since the latter is a critical point of a	suitably defined energy-Casimir functional which is conserved along the non-linear dynamics and whose second variation at $f_0$	corresponds to the quadratic form induced by ${\cal L}$.
		\item Consider an eigenvalue $\alpha<0$ of $\L$ with eigenfunction $f \in \H$, i.e., $\L f=\alpha f$. Then $g \coloneqq e^{\sqrt{-\alpha}\, t} f$  solves \eqref{eq:second_order_lin} and we get a solution of the linearized Einstein-Vlasov system which grows exponentially in time. We thus call $g$ an exponentially growing mode.
	\end{enumerate}
\end{remark}

\section{Properties of the operators}\label{sc:properties_operators}

In this section we consider steady states as stated in Section~\ref{sc:ststconditions}. We now gather some properties of the operators that were introduced above. The important identity
\begin{align}
	\frac{\pi}{r^2}\int_0^\infty \int_\R  w^2 |\varphi'(E(r,w,L),L)| \, dw dL =  \frac{e^{-2\lambda_0(r)-\mu_0(r)}}{4\pi r} (\lambda_0'+\mu_0') (r), \label{eq:HLRidentity}
\end{align}
for $r\in ]2M,\infty[$ will be used repeatedly. It can be derived by a simple integration by parts, see \cite[Lemma~4.4]{HaLiRe2020} and recall~\ref{it:stst3}. 
\subsection{The transport operator $\T$}\label{sc:properties_T}
The main advantage of the action-angle type variables (cf.\ Section~\ref{sc:aacoords}) is that the transport operator $\T$ is transformed into a one-dimensional derivative along the angle variable~$\theta$. For this we introduce the spaces
\begin{align}
	H^1_\theta &\coloneqq \{ y\in H^1(]0,1[) \, | \, y(0)=y(1)\} \label{eq:h1theta}, \\
	H^2_\theta &\coloneqq \{ y\in H^2(]0,1[) \, | \, y(0)=y(1) \text{ and } \dot y(0)=\dot y(1) \}  = \{ y\in H^1_{\theta} \, | \, \dot y \in H^1_\theta \}  \label{eq:h2theta},
\end{align}
where the boundary conditions are imposed for the continuous representatives which exist by the Sobolev embeddings $H^1(]0,1[) \hookrightarrow C([0,1])$ and  $H^2(]0,1[) \hookrightarrow C^1([0,1])$. We collect the following properties of the transport operator $\T$ as in~\cite{HaReSt21}: 
\begin{proposition}\phantomsection\label{prop:transport_char}
	\begin{enumerate}[label=(\alph*)]
		\item\label{it:Tprop1} $\T \colon \mathrm D(\T) \to H$ is well-defined and skew-adjoint as a densely defined operator on $H$, i.e., $\T^\ast=-\T$. Moreover, $\T^2 \colon \mathrm D(\T^2) \to H$ is self-adjoint.
		\item\label{it:Tprop2} The domains of $\T$ and $\T^2$ can be characterized by 
		\begin{align*}
			\mathrm{D}(\T) = \Big \{f \in H  \mid\,&  f(\cdot, E,L) \in H^1_{\theta} \text{ for a.e. } (E,L) \in \Omega_{0}^{EL} \\
			& \text{and } \iint_{\Omega_0^{EL}} \frac{T(E,L)^{-1}}{|\varphi'(E,L)|} \int_0^1 |\partial_\theta f(\theta, E,L)|^2 \, d\theta dEdL< \infty\Big\}, \\
			\mathrm{D}(\T^2) = \Big\{f \in H   \mid\,& f(\cdot, E,L) \in H^2_{\theta} \text{ for a.e. } (E,L) \in \Omega_{0}^{EL} \\
			& \text{and } \sum_{j=1}^{2}\iint_{\Omega_0^{EL}} \frac{T(E,L)^{1-2j}}{|\varphi'(E,L)|} \int_0^1 |\partial_\theta^j f(\theta, E,L)|^2 \, d\theta dEdL< \infty\Big\}.
		\end{align*} 
		In addition, for $f \in \mathrm D(\T)$,
		\[ 
		(\T f)(\theta,E,L) = -\frac{1}{T(E,L)}\,(\partial_\theta f)(\theta,E,L),
		\]
		and for $f\in \mathrm D(\T^2)$,
		\[ 
		(\T^2 f)(\theta,E,L) = \frac{1}{T(E,L)^2}\,(\partial_\theta^2 f)(\theta,E,L)
		\]
		for a.e.\ $(\theta,E,L) \in \Omega_0^{EL}$.
		\item\label{it:Tprop3} The kernel of $\T$ consists of functions only depending on $(E,L)$, i.e., 
		\begin{multline}\label{eq:kernel_T}
				\ker(\T) = \left\{f \in H \mid \exists g\colon \R^2 \to \R \text{ s.t. } f(r,w,L) = g(E(r,w,L),L)\text{ a.e.\ on } \Omega_0\right\}.
		\end{multline}
		\item\label{it:Tprop4} The range and the orthogonal complement of the kernel of $\T$ are equal and are given by 
		\begin{equation}
			\im(\T) = \ker(\T)^\perp =  \left \{ f\in H \mid \int_0^1 f(\theta,E,L) \, d\theta = 0  \text{ for a.e. } (E,L) \in \Omega_0^{EL}\right\}.
		\end{equation}
		\item\label{it:Tprop5} For every $f\in \mathrm D(\T)$ there exists a sequence $(f_n)_{n\in\N} \subset C^\infty_c (\Omega_0)$ such that 
		\[ 
			f_n \to f, \quad \T f_n \to \T f \quad \text{in } H\text{ as }n \to \infty.
		\]
		\item\label{it:Tprop6} $\T$ reverses $w$-parity, i.e., $(\T f)_{\pm} = \T(f_{\mp})$ for $f\in \mathrm D(\T)$, in particular, $f\in \mathrm D(\T)$ is equivalent to $f_{\pm} \in \mathrm D(\T)$. Moreover, the restricted operator $\T^2\colon\mathrm D(\T^2)\cap\H\to\H$ is self-adjoint as a densely defined operator on $\H$.
		\item\label{it:Tprop7}  $\T \colon \mathrm D(\T) \cap \ker(\T)^\bot \to \im(\T)$ is bijective. Its inverse ${\T}^{-1} \colon \im(\T)  \to \mathrm D(\T) \cap \ker(\T)^\bot $ is given by
		\[ 
		({\T}^{-1} f) (\theta,E,L) = - T(E,L) \left (\int_0^\theta f(\tau,E,L)\, d\tau  - \int_0^{1} \int_0^\sigma f(\tau,E,L)\, d\tau d\sigma \right )
		\]
		for a.e.\ $(\theta,E,L) \in [0,1]\times \Omega_0^{EL}$, is bounded, and reverses $w$-parity.
	\end{enumerate}
\end{proposition}
\begin{proof}
	The first statement in~\ref{it:Tprop1} is proven in~\cite[Thm.~2.2]{ReSt20}; the proofs from \cite{ReSt20} also work in the present situation with a differing class of steady states. The second part follows by von~Neumann's theorem, cf.~\cite[Thm.~X.25]{ReSi2}. The characterizations in~\ref{it:Tprop2} can be shown similarly to~\cite[Lemma~5.2 and Corollary~5.4]{HaReSt21} using~\eqref{eq:aacoordsvolumeelements}; the only difference being one sign change in the transport operator and a different weight in the integration, but the weight difference $e^{\pm\lambda_0}$ is bounded on the steady state support. Part~\ref{it:Tprop3} follows from~\ref{it:Tprop2}; note that the kernel of $\partial_\theta\colon H^1_\theta\to L^2(]0,1[)$ consists of functions which are constant almost everywhere.	The equality of $\im(\T)$ and $\ker(\T)^\perp$ as well as the explicit characterization of this set stated in~\ref{it:Tprop4} are an application of~\ref{it:Tprop1}--\ref{it:Tprop3}, see~\cite[Lemma~5.5]{HaReSt21}. As for part~\ref{it:Tprop5} we refer to~\cite[Prop.~2]{ReSt20}. The claim~\ref{it:Tprop6} follows immediately from parity considerations and from the weak definition of $\T$. The formula for the inverse can be easily verified by using~\ref{it:Tprop2} and~\ref{it:Tprop3}. The fact that $\T^{-1}$ reverses $w$-parity follows by~\ref{it:Tprop6}. Note that the boundedness of the period function $T$ assumed in~\ref{it:stst2} is, e.g., needed for parts~\ref{it:Tprop3} and~\ref{it:Tprop7}.
\end{proof}

In passing we note that not all of the properties derived above require the existence of action-angle type variables via the single-well structure as stated in Definition~\ref{def:jeans}. 

\subsection{The operator $\B$}\label{sc:properties_B}
In order to derive a Birman-Schwinger principle for the operator $\L=-\B^2-\mathcal R$ we need in-depth knowledge about the operator $\B$ defined in Definition~\ref{def:TandB}~\ref{it:TBdef2}. Loosely speaking, we want to derive similar properties for $\B$ as the ones for $\T$ stated in Proposition~\ref{prop:transport_char}, which turns out to be much more difficult. In particular, we need to characterize its kernel, its image, and have to find its inverse. For this reason this section is quite technical and peppered with many non-trivial calculations. At the end of this section, we gather all important properties of $\B$ in Proposition~\ref{prop:Bproperties} in case the reader wants to skip the tedious but insightful technicalities.
First, we analyze the source terms.
\begin{lemma}\phantomsection\label{lemma:source_bounded}
	\begin{enumerate}[label=(\alph*)]
		\item\label{it:scbd1} The mappings 
		\begin{align*}
			H \ni f \mapsto |\varphi'| \rho_f \in H, \quad	H \ni f \mapsto |\varphi'| p_f \in H, \quad H \ni f \mapsto |\varphi'| j_f \in H
		\end{align*}
		are bounded, where $\rho_f$, $p_f$, $j_f$ are defined in~\eqref{eq:rho}--\eqref{eq:j}. The mapping 
		\begin{equation*}
			H \ni f \mapsto \lambda_f \in L^2([\Rmin,\Rmax])
		\end{equation*}
		is compact with $\lambda_f$ given by~\eqref{eq:lambdaeqexplicit}.
		\item\label{it:scbd2} For $f\in C^1_c(\Omega_0)$,
		\begin{align}
			p_f'&=-\mu_0'\,\left(p_f+\rho_f\right)-\frac2r\left(p_f-q_f\right)-e^{\lambda_0-\mu_0} j_{\T f},\label{eq:pfprime}\\
			j_f'&=- 2\left (\mu_0' + \frac1r\right ) j_f - e^{\lambda_0-\mu_0} \rho_{\T f} \label{eq:jfprime}.
		\end{align}
		The mappings
		\[ 
			\mathrm D(\T) \ni f \mapsto |\varphi'|\,(rp_f)' \in H, \quad \mathrm D(\T) \ni f \mapsto |\varphi'|\,(rj_f)' \in H
		\] 
		are well-defined and bounded, if $\mathrm D(\T)$ is equipped with the norm $\|\cdot\|_H + \| \T \cdot\|_H$ and the derivatives are taken in the weak sense. 
	\end{enumerate}
\end{lemma} 
\begin{proof}
	The Cauchy-Schwarz inequality together with \eqref{eq:phi_prime_bound} implies the boundedness in part~\ref{it:scbd1}. For the compactness property, consider $(f_n)_{n\in \N}\subset H$ which converges weakly to $0$ in $H$. 
	Using~\eqref{eq:lambdaeqexplicit}, we obtain that
	\begin{equation}\label{eq:lambda_compact}
		\|   \lambda_{f_n} \|_{L^2([\Rmin,\Rmax])}^2 = \int_{\Rmin}^{\Rmax}\frac{e^{4\lambda_0(r)}}{r^2}\,\left\langle e^{-\lambda_0}|\varphi'|\sqrt{1+w^2+\frac L{s^2}}\,\mathds1_{[\Rmin,r]},f_n\right\rangle_H^2\,dr.
	\end{equation}	
	Weak convergence of $(f_n)$ in $H$ implies that the scalar product in~\eqref{eq:lambda_compact} converges pointwise for $r\in [\Rmin,\Rmax]$. By using the Cauchy-Schwarz inequality, \eqref{eq:phi_prime_bound}, and the boundedness of $(f_n)$, we obtain that the integrand in~\eqref{eq:lambda_compact} is uniformly bounded by an integrable function. Lebesgue's dominated convergence theorem thus yields the convergence of \eqref{eq:lambda_compact} to zero as desired. 
	
	The formulas for $p_f'$ and $j_f'$ can be deduced by a lengthy integration by parts. The last claim then follows with similar arguments as~\ref{it:scbd1} after approximating $f \in \mathrm D(\T)$ according to Proposition~\ref{prop:transport_char}~\ref{it:Tprop5}.
\end{proof}
The identity \eqref{eq:pfprime} can be interpreted as a generalized version of the Tolman-Oppenheimer-Volkov equation~\cite[Lemma~3.3]{Rein94}. We now prove that $\B$ is well defined. 
\begin{lemma}\phantomsection\label{lemma:Bselfadjoint}
	\begin{enumerate}[label=(\alph*)]
		\item\label{it:Bsa1} $\B \colon \mathrm D(\T) \to H$ is well-defined, skew-adjoint as a densely defined operator on~$H$, and reverses $w$-parity. Moreover, $\B^2 \colon \mathrm D(\T^2) \to H$ and the restricted operator $\B^2 \colon  \mathrm D(\T^2) \cap \H\to \H$ are well-defined and self-adjoint; the set $\mathrm D(\T^2)$ is defined in Definition~\ref{def:TandB}~\ref{it:TBdef4}.
		\item\label{it:Bsa2} For every $f\in \mathrm D(\T)$ there exists a sequence $(f_n)_{n\in\N} \in C^\infty_c (\Omega_0)$ such that 
		\[ 
		f_n \to f, \quad \B f_n \to \B f \quad \text{in } H\text{ as } n \to \infty.
		\]
	\end{enumerate}
\end{lemma}
\begin{proof}
	We can write $\B= \T + \mathcal S$ where $\mathcal S$ is bounded on $H$ according to Lemma~\ref{lemma:source_bounded}~\ref{it:scbd1} and $\mathcal S$ is skew-symmetric since for $f,g\in H$, 
	\begin{align*}
		\langle \mathcal S f, g \rangle_H &= \left \langle - 4\pi r  |\varphi'| e^{2\mu_0+\lambda_0} \left (w p_f - \frac{w^2}{\sqrt{1+w^2+\frac{L}{r^2}}} j_f \right ) , g \right \rangle_H \\
		&=  (4\pi)^2 \int_{\Rmin}^{\Rmax} e^{2\mu_0+2\lambda_0} r^3 (p_g j_f - p_f j_g) \, dr. 
	\end{align*}
	Hence, $\B=\T+\mathcal S$ is skew-adjoint by the Kato-Rellich theorem~\cite[Thm.~X.12]{ReSi2} with domain $\mathrm D(\B)\coloneqq \mathrm D(\T)$. Thus, von Neumann's theorem~\cite[Thm.~X.25]{ReSi2} implies that $\B^2$ is self-adjoint on the domain $\mathrm D(\B^2)\coloneqq\{ f\in\mathrm D(\B)\mid \B f\in \mathrm D(\B)\}$. Moreover, $\mathrm D(\B^2)=\mathrm D(\T^2)$. In order to see this equality, it remains to check that $\mathcal S f \in \mathrm D(\T)$ for $f\in\mathrm  D(\T)$, i.e.,
	\[ 
		\mathcal Sf=-4\pi r |\varphi'| e^{2\mu_0+\lambda_0} \left (w p_f - \frac{w^2}{\sqrt{1+w^2+\frac{L}{r^2}}} j_f \right ) \in \mathrm D(\T),
	\]
	which follows from Lemma~\ref{lemma:source_bounded}~\ref{it:scbd2}. Since $\T$ and $\mathcal S$ reverse $w$-parity, the same is true for $\B$. In particular, $\B^2$ preserves $w$-parity which implies the last statement in~\ref{it:Bsa1}. Furthermore, Proposition~\ref{prop:transport_char}~\ref{it:Tprop5} and Lemma~\ref{lemma:source_bounded}~\ref{it:scbd1} imply part~\ref{it:Bsa2} because $\mathcal S$ is bounded.
\end{proof}
Before analyzing the operators further, we need some auxiliary results and identities which will be important throughout the work. 
\begin{lemma}\label{lemma:lambdadot}
	Let $f \in \mathrm D(\mathcal{\T})$. Then the following identities hold for a.e.\ $r\in [\Rmin,\infty[$: 
	\begin{align}
		\lambda_{\mathcal{B}f}(r) &= -4\pi r e^{(\lambda_0 + \mu_0)(r)} j_f(r) \label{eq:lambdadotB}, \\
		\lambda_{e^{\mu_0+\lambda_0} \T f}(r) &= -4\pi r e^{(2\mu_0+2\lambda_0)(r)} j_f(r). \label{eq:lambdadotmod}
	\end{align}
\end{lemma}
\begin{proof}
	Combining the approximation results from Proposition~\ref{prop:transport_char}~\ref{it:Tprop5} and Lemma~\ref{lemma:Bselfadjoint}~\ref{it:Bsa2} with Lemma~\ref{lemma:source_bounded}~\ref{it:scbd1} allows us to assume that $f\in C^\infty_c(\Omega_0)$.
	We start with~\eqref{eq:lambdadotB} by writing $\T f$ as in~\eqref{eq:Tdef} and integrate by parts in \eqref{eq:lambdaeqexplicit} to obtain that
	\begin{multline*}
	\lambda_{\B f}(r) = -4\pi r e^{(\mu_0+\lambda_0)(r)}  j_{f}(r) - \frac{4\pi e^{2\lambda_0(r)}}{r} \int_{\Rmin}^{r} (\mu_0'+\lambda_0')(s)\,e^{(\mu_0-\lambda_0)(s)} j_{f}(s)  s^2 \, ds \\
		   + \frac{16\pi^3 e^{2\lambda_0(r)}}{r} \int_{\Rmin}^r s e^{(2\mu_0+\lambda_0)(s)}j_f(s) \int_{0}^{\infty}  \int_\R  w^2 |\varphi'| \, dw dL ds=  -4\pi r  e^{(\mu_0+\lambda_0)(r)} j_{f}(r),
	\end{multline*}
	where we have used \eqref{eq:HLRidentity} in the last step. 	For~\eqref{eq:lambdadotmod} similar arguments can be carried out. 
\end{proof}
\subsubsection{Characterization of $\ker(\B)$ and $\ker(\B)^\bot$}\label{ssc:kerB_kerBbot}
We now characterize the kernel of $\B$ using our knowledge about the kernel of $\T$, see Proposition~\ref{prop:transport_char}~\ref{it:Tprop3}. A comment on notation is in order. From now on we write $R=R(\theta,E,L)$ and $W=W(\theta,E,L)$ when expressing the radial coordinate and the radial momentum, respectively, as a function of the action-angle type variables $(\theta,E,L)$ introduced in Section~\ref{sc:aacoords}. More precisely,
\begin{align*}
	\partial_\theta R&=T(E,L)e^{\mu_0(R)-\lambda_0(R)}\,\frac W{\sqrt{1+W^2+\frac L{R^2}}},\\
	\partial_\theta W&=T(E,L)e^{\mu_0(R)-\lambda_0(R)}\left(\frac L{R^3\sqrt{1+W^2+\frac L{R^2}}}-\mu_0'(R)\,\sqrt{1+W^2+\frac L{R^2}}\right),
\end{align*}
where $(R,W)(0,E,L)=(r_-(E,L),0)$ and $(R,W)(\frac12,E,L)=(r_+(E,L),0)$ with $0<r_-(E,L)<r_+(E,L)$ defined for a.e.~$(E,L)\in\Omega_0^{EL}$ by Definition~\ref{def:jeans}. 
\begin{lemma}\phantomsection\label{lemma:kernel_B}
	\begin{enumerate}[label=(\alph*)]
		\item\label{it:kerB1} The kernel of $\B$ is given by 
		\begin{align*}
			\ker(\B )&= \left\{ g + 4\pi |\varphi'|E e^{-\lambda_0-\mu_0} \int_{r}^{\Rmax} e^{(3\lambda_0+\mu_0)(s)} p_g(s) s \,ds \mid g = g(E,L) \in \ker \T \right\} .
		\end{align*}
		When $f\in \ker(\B)$ is of the form above, we refer to $g$ as the {\em generator} of $f$. This generator is given by
		\[ 
		g(E,L)=f\left (\frac 1 2,E,L\right ) - 4\pi |\varphi'(E,L)| E \int_{r_+(E,L)}^{\Rmax} e^{2\lambda_0(s)} p_f(s) s \, ds
		\] 
		for a.e.\ $(E,L) \in \Omega^{EL}_0$. 		
		\item\label{it:kerB2} The mappings 
		\begin{align*}
			\ker(\T) \ni g & \mapsto g + 4\pi |\varphi'| E e^{-\lambda_0-\mu_0} \int_r^{\Rmax} e^{(3\lambda_0+\mu_0)(s)} p_g(s) s \, ds \in \ker(\B), \\
			\ker(\B) \ni f & \mapsto f\left (\frac 1 2,E,L \right )- 4\pi |\varphi'| E \int_{r_+(E,L)}^{\Rmax} e^{2\lambda_0(s)} p_f(s) s \, ds \in  \ker(\T)
		\end{align*}
		are well-defined, bijective, and inverse to each other. In particular, the generator of $f\in \ker(\B)$ is uniquely determined.
	\end{enumerate} 
\end{lemma}
\begin{proof}
		We first show that every element in the kernel of $\B$ has the form claimed in~\ref{it:kerB1}. For $f\in \ker \B$, eqn.~\eqref{eq:lambdadotB} implies that $0 = \lambda_{\B f} = -4\pi e^{\lambda_0+\mu_0}r j_{f}$ and hence $0=j_f=j_{f_-}$.
		Since $\B$ and $\T$ reverse $w$-parity and $f_\pm \in \mathrm D(\T)$, see Proposition~\ref{prop:transport_char}~\ref{it:Tprop6}, we have
	\begin{align}
		\B f_+ &= \T f_+ - 4\pi r |\varphi'| e^{2\mu_0+\lambda_0}w p_{f_+} = 0, \label{eq:kernelsplit1} \\
		\B f_- &= \T f_- + 4\pi r |\varphi'| e^{2\mu_0+\lambda_0} \frac{w^2}{\sqrt{1+w^2+\frac{L}{r^2}}} j_{f_-} = 0, \label{eq:kernelsplit2}
	\end{align}
	which yields that $\T f_- = 0$ and, thus, $f_- \in \ker(\T)$. However, the kernel of $\T$ consists only of even-in-$w$ functions and we obtain that $f=f_+$. 
	Hence, only eqn.~\eqref{eq:kernelsplit1} remains, which, written in $(\theta,E,L)$-variables using Proposition~\ref{prop:transport_char}~\ref{it:Tprop2}, reads
	\begin{equation}\label{eq:kernel_theta_eq}
		\frac{1}{T(E,L)} \partial_\theta f(\theta,E,L) =- 4\pi R |\varphi'(E,L)| e^{(2\mu_0+\lambda_0)(R)} W p_f(R).
	\end{equation}
	For $\theta \in [0, \frac 1 2]$ integrating~\eqref{eq:kernel_theta_eq} in $\theta$ yields that
	\[ 
	f(\theta,E,L) = f\left (\frac 1 2,E,L\right ) + 4\pi |\varphi'(E,L)| T(E,L) \int_\theta^{\frac 1 2 }  \left (  R e^{(2\mu_0+\lambda_0)(R)} W p_f(R) \right ) (\tau, E,L)\, d\tau;
	\]
	recall that  $f(\cdot, E,L)\in H^1_\theta$ for a.e.\ $(E,L)\in \Omega_0^{EL}$, i.e., the evaluation at $\theta=\frac 12$ is well-defined for the continuous-in-$\theta$ representative. 
	We next change variables via $s=R(\tau,E,L)$ and get
	\begin{align*}
		f(\theta,E,L) &= f\left (\frac 1 2,E,L\right ) + 4\pi |\varphi'(E,L)| E \int_{R(\theta,E,L)}^{r_+(E,L)} e^{2\lambda_0(s)} p_f(s) s \, ds\\
		&= g(E,L) + 4\pi |\varphi'(E,L)| E \int_{R(\theta,E,L)}^{\Rmax} e^{2\lambda_0(s)} p_f(s) s \, ds,
	\end{align*}
	where $g$ is defined as above, i.e., $g$ is the generator of $f$. We now know that $f$ has to be of the form
	\[ 
	f(\theta,E,L) = g(E,L) + |\varphi'(E,L)| E \, H(R(\theta,E,L)),
	\]
	for some function $H\in C^1( [\Rmin,\Rmax])$ which depends on $f$ with $H(\Rmax)=0$. By applying the chain rule and using \eqref{eq:kernel_theta_eq}, we obtain the following differential equation:
	\begin{equation}\label{eq:Hdiffequation}
		\partial_r H  = -4\pi r e^{2\lambda_0} (p_g + p_{|\varphi'|EH}), \quad H(\Rmax)=0.
	\end{equation} 
	We calculate via~\eqref{eq:HLRidentity} that
	\[ 
	p_{|\varphi'|EH}(r) = H(r) \frac{e^{-2\lambda_0(r)} }{4\pi r} (\lambda_0'+\mu_0')(r),
	\]
	and therefore the unique solution of~\eqref{eq:Hdiffequation} is
	\[ 
	H(r) = 4\pi e^{-\mu_0 -\lambda_0} \int_{r}^{\Rmax} e^{(3\lambda_0+\mu_0)(s)} p_g(s) s \, ds, \quad r\in[\Rmin,\Rmax],
	\]
	which proves the first inclusion in~\ref{it:kerB1}. 

	We now show that, for every $g=g(E,L) \in \ker(\T)$, the function
	\[ 
	f(\theta,E,L) = g(E,L) + 4\pi |\varphi'(E,L)|E e^{(-\lambda_0-\mu_0)(R)} \int_{R}^{\Rmax} e^{3\lambda_0+\mu_0} p_g s \,ds 
	\]
	is an element of the kernel of $\B$. The fact that $f\in \mathrm D(\T)$ can be seen from the characterization of $\mathrm D (\T)$ in Proposition~\ref{prop:transport_char}~\ref{it:Tprop2} together with Lemma~\ref{lemma:source_bounded}~\ref{it:scbd1}. Since $\T g=0$, we obtain by the chain rule that
	\begin{equation}\label{eq:kernel_Tf}
			\T f = 4\pi r |\varphi'| e^{2\mu_0+\lambda_0}  w \left ( p_g + \frac{ e^{-\mu_0-3\lambda_0}}{r} (\mu_0'+\lambda_0')  \int_r^{\Rmax} e^{3\lambda_0+ \mu_0} p_g s \, ds\right  ) .
	\end{equation}
	Furthermore, for  $\B f$ we  calculate that
	\begin{align*}
		p_f = p_g + \frac{e^{-3\lambda_0-\mu_0}}{r} (\lambda_0'+\mu_0')  \int_r^{\Rmax} e^{3\lambda_0+ \mu_0} p_g s \, ds
	\end{align*}
	after using \eqref{eq:HLRidentity} again. Since $f$ is even in $w$, this together with \eqref{eq:kernel_Tf} yields that
	\begin{equation*}
		\B f = \T f - 4\pi r |\varphi'| e^{2\mu_0+\lambda_0} w p_f = 0,
	\end{equation*}
	and completes the proof of~\ref{it:kerB1}. Part~\ref{it:kerB2} results from a straightforward calculation, the details of which we do not go into here.
\end{proof}
To summarize, we can characterize $\ker(\B)$ similarly to $\ker(\T)$ but need to provide an extra term additional to a function that only depends on $(E,L)$.  We want to stress that the integration from $\Rmax$ instead of $\Rmin$ in the formula for the generator of $f\in \ker(\B)$ offers a significant advantage in the following: It allows for a simpler characterization of $\ker(\mathcal{B})^\bot$.
\begin{lemma}\label{lemma:characterkerBbot}
	Let $f \in H$. Then, $f \in \ker(\mathcal{B})^\bot$ is equivalent to
	\begin{equation}\label{eq:imageB}
		\int_0^1 \left (f(\theta,E,L) + |\varphi'(E,L)| e^{2\mu_0(R)} \lambda_f(R) \frac{W^2}{E}\right ) \, d\theta = 0 \, \text{ for a.e. } (E,L)\in \Omega_0^{EL},
	\end{equation}
	i.e., $f+|\varphi'|e^{2\mu_0}\lambda_f\frac{w^2}E\in\ker(\T)^\perp$.
	In particular, $\H \subset \ker(\B)^\perp$.
\end{lemma}
\begin{proof}
	From Lemma~\ref{lemma:kernel_B} we know that $f \in  \ker(\mathcal{B})^\bot$ if and only if 
	\begin{align}
		0 &= \iiint_{\Omega_0}  \frac{e^{\lambda_0}}{|\varphi'|} f \left ( g  + 4\pi |\varphi'|E e^{-\lambda_0-\mu_0} \int_{r}^{\Rmax} e^{(3\lambda_0+\mu_0)(s)} p_g(s) s \,ds\right  )  \, dr dw dL \notag \\ 
		&= 	\iiint_{\Omega_0}  \frac{e^{\lambda_0}}{|\varphi'|} f  g \,  dr dw dL + 4 \int_{\Rmin}^{\Rmax}  r^2 \rho_f \int_{r}^{\Rmax} e^{(3\lambda_0+\mu_0)(s)} p_g(s) s \, ds dr	\label{eq:conditionkernelB}
	\end{align}
	for every $g=g(E,L) \in \ker \T$. We first employ an integration by parts for the second term using~\eqref{eq:fieldeq1_lin} in order to recover $g$ from $p_g$:
	\begin{align*}
		&4 \int_{\Rmin}^{\Rmax} r^2 \rho_f \int_{r}^{\Rmax} e^{3\lambda_0+\mu_0} p_g(s) s \, ds dr =\frac 1 \pi  \int_{\Rmin}^{\Rmax} \left( re^{-2\lambda_0} \lambda_f \right )' \int_{r}^{\Rmax} e^{3\lambda_0+\mu_0} p_g(s) s \, ds dr \\
		&\quad = \frac 1 \pi  \int_{\Rmin}^{\Rmax} e^{\lambda_0+\mu_0}  \lambda_f p_g r^2 dr 
		= 	\iiint_{\Omega_0}    e^{\lambda_0+2\mu_0} \lambda_f \frac{w^{2}}E g(E,L) \, dr dw dL;
	\end{align*}
	note that there are no boundary terms when integrating by parts since $\lambda_f(\Rmin)=0$ by~\eqref{eq:lambdaeqexplicit}.
	Therefore, condition~\eqref{eq:conditionkernelB} is equivalent to 
	\begin{align*}
		0 &=   \iiint_{\Omega_0} g(E,L)e^{\lambda_0(r)}  \left ( \frac{f(r,w,L)}{|\varphi'(E,L)|} + e^{2\mu_0(r)} \lambda_f(r) \frac{w^2}{E} \right ) \,  dr dw dL\\
		&= \iint_{\Omega_0^{EL}}  g(E,L)T(E,L) \int_{0}^{1} \left ( \frac{f(\theta,E,L)}{|\varphi'(E,L)|} + e^{2\mu_0(R)} \lambda_f(R) \frac{W^2}{E} \right ) \, d\theta dE dL
	\end{align*}
	after changing from $(r,w,L)$ to action-angle type variables, cf.~\eqref{eq:aacoordsvolumeelements}. Since $g$ is an arbitrary function in $(E,L)$ and $T>0$, the inner integral must vanish almost everywhere and the claim follows.  
	
	As to the final inclusion, for $f\in \H$, i.e., odd-in-$w$ $f$ , we observe that $\lambda_f = 0$ and $\int_0^1 f(\theta,\cdot,\cdot) \, d\theta =0$ almost everywhere, see \eqref{eq:odd_in_w_theta}.
\end{proof}
One should compare this result with the characterization of $\ker(\T)^{\bot}$ in Proposition~\ref{prop:transport_char}~\ref{it:Tprop4}. For further analysis it is essential to characterize the image of~$\B$ which we do next. A useful tool for this is the existence of a right-inverse of $\B$.
\begin{definition}\label{def:Btildeinv}
	The operator $\widetilde{\B}^{-1} \colon \ker(\B)^\bot  \to \mathrm D(\T)$ is defined by
	\begin{multline*}
		\widetilde{\B}^{-1}f \coloneqq 	{\T}^{-1} \left (f+|\varphi'| e^{2\mu_0} \lambda_f \frac{w^2}{E}\right )
		\\+ 4\pi |\varphi'| E e^{-\lambda_0-\mu_0} \int_{r}^{\Rmax} e^{(3\lambda_0+\mu_0)(s)} p_{{\T}^{-1}(f+|\varphi'| e^{2\mu_0} \lambda_f \frac{w^2}{E})}(s) s \, ds .
	\end{multline*}
\end{definition}
We would of course prefer to give the actual inverse of $\B$, but we are not able to construct $\B^{-1}$ explicitly---we will see later why this is a difficult task. Besides, we omit the involved derivation of~$\widetilde{\B}^{-1}$, and simply verify that it is indeed a right-inverse of $\B$.

\begin{lemma}\label{lemma:Btildeinv_prop}
 	The operator  $\widetilde{\B}^{-1}$ is well-defined, bounded, reverses $w$-parity, and for every $f\in \ker(\B)^\bot$ it holds that $\B  \widetilde{\B}^{-1} f = f$. In particular $ \ker(\B)^\bot \subset \im (\B)$.
\end{lemma}
\begin{proof}
	 The operator $\widetilde{\B}^{-1}$ is well-defined and bounded which can be seen from Proposition~\ref{prop:transport_char}~\ref{it:Tprop2} and~\ref{it:Tprop7} as well as Lemmas~\ref{lemma:source_bounded} and~\ref{lemma:characterkerBbot}. 
	 For the reversal of $w$-parity, we let $f \in \ker(\B)^\bot$ be odd in $w$. Then  $\lambda_f=0$, we use that $\T^{-1}f$ is even in $w$, and we observe that the last term in $\widetilde{\B}^{-1}$ is even in $w$. Thus, $\widetilde{\B}^{-1}f$ is even in $w$. For even-in-$w$  $f$ the claim follows after noting that $p_{{\T}^{-1} (f+|\varphi'| e^{2\mu_0} \lambda_f \frac{w^2}{E})} =0$. 
	
	It remains to show the right-inverse property of $\widetilde{\B}^{-1}$. For $f\in\ker(\B)^\perp$, 
	\begin{align*}
		&(\B\widetilde{\B}^{-1}f) (\theta,E,L)  \\
		& = f(\theta, E,L) + |\varphi'| \frac{e^{2\mu_0}W^2}E \left ( \lambda_f +4\pi R e^{\mu_0+\lambda_0} j_{\widetilde{\B}^{-1}f} \right ) \\
		& - 4\pi  |\varphi'|  e^{2\mu_0-\lambda_0} W \left (  \partial_r \left ( e^{-\lambda_0-\mu_0} \int_{r}^{\Rmax} e^{3\lambda_0+\mu_0} p_{\T^{-1} (f+|\varphi'| e^{2\mu_0} \lambda_f \frac{w^2}E) }(s) s \, ds\right ) + R e^{2\lambda_0} p_{\widetilde{\B}^{-1}f} \right ).
	\end{align*}
	In order to show $\B \widetilde{\B}^{-1}f = f$, we prove the validity of the two equations
	\begin{align}
		\lambda_f &= -4\pi r e^{\mu_0+\lambda_0} j_{\widetilde{\B}^{-1}f} \label{eq:bhf_first}, \\
		 p_{\widetilde{\B}^{-1}f}&=   p_{{\T}^{-1} (f+|\varphi'| e^{2\mu_0} \lambda_f \frac{w^2}{E}) } \notag \\  &\quad +\frac 1r (\lambda_0'+\mu_0')  e^{-3\lambda_0-\mu_0}  \int_{r}^{\Rmax} e^{(3\lambda_0+\mu_0)(s)} p_{{\T}^{-1} (f+|\varphi'| e^{2\mu_0} \lambda_f \frac{w^2}{E})}(s) s \, ds   \label{eq:bhf_second}  .
	\end{align}
	We know from $j_{\widetilde{\B}^{-1}f} =j_{(\widetilde{\B}^{-1}f)_-}$  and~\eqref{eq:lambdadotmod} that for $g \coloneqq f+|\varphi'| e^{2\mu_0} \lambda_f \frac{w^2}{E }  \in \im(\T)$,
	\begin{align} 
		-4\pi & r e^{2\lambda_0 + 2\mu_0} j_{ \widetilde{\B}^{-1} f}=  -4\pi  r e^{2\lambda_0 + 2\mu_0} j_{\T^{-1} g}   =	\lambda_{ e^{\mu_0+\lambda_0} g}  = \frac{4\pi e^{2\lambda_0}}{r} \int_{\Rmin}^r e^{\mu_0+\lambda_0} \rho_g s^2 \, ds \nonumber  \\
		&= \frac{4\pi e^{2\lambda_0}}{r}   \int_{\Rmin}^r e^{\mu_0+\lambda_0} \left ( \rho_f + e^{\mu_0} \lambda_f \left ( \frac{\pi}{s^2} \int_0^\infty \int_\R  w^2 |\varphi'|  \, dw dL \right ) \right ) s^2\, ds \nonumber \\ 
		&= \frac{4\pi e^{2\lambda_0}}{r}\left ( \int_{\Rmin}^r e^{\mu_0+\lambda_0} \rho_f s^2 \, ds + \frac{1}{4\pi} \int_{\Rmin}^r  s e^{-2\lambda_0} \lambda_f  \partial_s \left ( e^{\mu_0+\lambda_0}\right ) \, ds   \right ) \nonumber \\
		& = e^{\mu_0+\lambda_0} \lambda_f, \label{eq:useful_identity_1}
	\end{align}
	where we applied~\eqref{eq:HLRidentity} and integrated by parts using~\eqref{eq:fieldeq1_lin}. This proves \eqref{eq:bhf_first}.
	
	To show that~\eqref{eq:bhf_second} holds as well we put the definition of ${\widetilde{\B}^{-1}f}$ into $p$ and get
	\begin{multline*}
		p_{\widetilde{\B}^{-1}f} = p_ {{\T}^{-1} (f+|\varphi'| e^{2\mu_0} \lambda_f \frac{w^2}{E})} \\
		\quad+ 4\pi e^{-\lambda_0}  \left ( \int_{r}^{\Rmax}  e^{(3\lambda_0+\mu_0)(s)} p_{{\T}^{-1} (f+|\varphi'| e^{2\mu_0} \lambda_f \frac{w^2}{E })}(s) s \, ds \right ) \frac{\pi}{r^2} \int_0^\infty \int_\R   w^2 |\varphi'| \, dw dL, 
	\end{multline*}
	which yields~\eqref{eq:bhf_second} after inserting~\eqref{eq:HLRidentity}.	We have thus proven $\B \widetilde{\B}^{-1}f = f$ and in particular conclude that $f\in \im (\B)$.
\end{proof}
The characterization of the orthogonal complement of $\ker (\B)$ together with Lemma~\ref{lemma:Btildeinv_prop} yields the following crucial result:
\begin{proposition}\label{prop:imBkerB}
	It holds that
	\[ 
	\im  (\B) = \ker(\mathcal{B})^\bot. 
	\]
In particular, the range of $\B$ is closed. 
\end{proposition}
\begin{proof}
	Since $\B$ is skew-adjoint, $\ker(\mathcal{B})^\bot = \overline{\im (\B)} $, cf.~\cite[Cor.~2.18~(iv)]{Brezis2011}. Furthermore, from Lemma~\ref{lemma:Btildeinv_prop} we have that $\ker(\B)^\bot \subset \im (\B)$ and the claim follows. 
\end{proof}
In \cite[Remark~4.15]{HaLiRe2020} it was noted without proof that $\im (\B)$ is closed in $H$ if Jeans' theorem holds for the steady state (meaning that the steady state has single-well structure in our terminology). We have now proven this remark in detail. 

\subsubsection{The inverse of $\B$ and $\B^2$}\label{ssc:inverseB_inverseBsq}

In order to construct the actual inverse of $\B$, we need to project elements of $H$ onto $\ker(\B)$, since $\B^{-1} f \in \im(\B)=\ker(\B)^\bot$ has to hold, if $\B^{-1}$ exists. We denote by $\Pi:H \to \ker(\B)$ the orthogonal projection onto $\ker(\B)$ which is the unique bounded and symmetric operator such that $\Pi =\mathrm{id}$ on $\ker(\B)$ and $\Pi=0$ on $\ker(\B)^\bot$. For the existence theory of such projections, see, e.g., \cite[Section~5.1]{Brezis2011} or \cite[Section~5.4]{HiSi}. Note that precisely this non-explicit projection is also used in \cite{HaLiRe2020}.

With the projection $\Pi$ and the skew-adjointness of $\B$, we can now determine the kernel and the image of $\B^2$.
\begin{lemma}\label{lemma:kernelBsq}
	The kernel and image of $\B^2$ are given by
	\[ \ker(\B^2) = \ker(\B), \quad \im  (\B^2)=\im (\B). \]
\end{lemma}
\begin{proof}
	For the first equality, we only need to show $\ker(\B^2) \subset \ker(\B)$ since the reverse inclusion is trivial. If $f\in \ker(\B^2)$ we have that
	\[ 
	0=\left \langle\B^2f,f\right \rangle_H = -\left \langle\B f,\B f\right \rangle_H = - \|\B f\|^2_H
	\]
	by the skew-symmetry of $\B$ 
	and thus $\B f=0$. For the second claim, we only have to prove  $\im (\B^2) \supset \im (\B)$. Let $f\in \im (\B)$, i.e., there exists $\tilde h \in \mathrm D(\T)$ such that $\B\tilde h =f$. Define 
	\[ 
	h \coloneqq (\mathrm{id}-\Pi) 	\tilde h \in \ker(\B)^\bot
	\]
	for which $\B h = \B	\tilde h = f$. Applying Proposition~\ref{prop:imBkerB} gives $h \in \im(\B)$. This implies that there exists $g \in \mathrm D(\T)$  with $\B g= h$ and therefore $\B^2 g = f$, i.e., $f\in \im  (\B^2)$.
\end{proof}
We now show that $\B^{-1}$ and $(\B^2)^{-1}$ exist. There are two main reasons why we need to calculate $(\B^2)^{-1}$ as explicitly as possible. First, we need to make sure that the spectrum of $\B^{2}$ does not contain zero when considering odd-in-$w$ functions. This facilitates the analysis of the spectrum of $\B^2$. In addition, the inverse of $\B^2$ is crucial to derive a Birman-Schwinger principle.

We first show that we can invert $\B$ on an appropriate set. The right-inverse $\widetilde{\B}^{-1}$ need not map elements of $\im (\B)$ back into $\im (\B)=\ker(\B)^\perp$ which would be necessary for the actual inverse of $\B$. We thus have to subtract the projection onto the kernel of $\B$.
\begin{lemma}\label{lemma:Binv}
	The operator $\B \colon \mathrm D(\T) \cap \ker(\B)^\bot  \to \im(\B)$ is bijective. Its inverse is bounded on $\im(\B)$, reverses $w$-parity, and is given by 
	\[ 
	\B^{-1} = (\mathrm{id}-\Pi)\widetilde{\B}^{-1},
	\]
	with $\widetilde{\B}^{-1}$ explicitly defined in Definition~\ref{def:Btildeinv}.
\end{lemma} 
\begin{proof}
	The well-definedness and boundedness of $(\mathrm{id}-\Pi)\widetilde{\B}^{-1} \colon \im(\B) \to H$ follow by Lemma~\ref{lemma:Btildeinv_prop} and Proposition~\ref{prop:imBkerB}. In addition, $\im((\mathrm{id}-\Pi)\widetilde{\B}^{-1})\subset\im(\mathrm{id}-\Pi)=\ker(\B)^\perp$, and for $f\in\mathrm D(\T)$ we have that $(\mathrm{id}-\Pi)f\in\mathrm D(\T)$ since $\im(\Pi)=\ker(\B)\subset\mathrm D(\T)$. Thus, $\im((\mathrm{id}-\Pi)\widetilde{\B}^{-1})\subset\mathrm D(\T) \cap \ker(\B)^\perp$. 
	The reversal of $w$-parity can be seen from the fact that $\widetilde{\B}^{-1}$ reverses and $\mathrm{id}-\Pi$ conserves $w$-parity; note that $\H\subset\ker(\B)^\perp$ and that $\ker(\B)$ consists only of functions even in $w$. 
	
	We now show that $\B$ as a mapping from $\mathrm D(\T) \cap \ker(\B)^\bot$ to $\im(\B)$ is bijective. First, we prove its surjectivity. Let $g\in \im (\B)$, i.e., there exists $f\in \mathrm D(\T)$ such that $\B f = g$. 
	We define $\tilde f \coloneqq (\mathrm{id}-\Pi)f$ and obtain that $\tilde f \in \mathrm D(\T)\cap\ker(\B)^\bot$ with $\B \tilde f = \B f = g$, i.e., $\B$ is surjective. 	Obviously,
	\[ 
	\ker(\B |_{\mathrm D(\T) \cap \ker(\B)^\bot}) = \ker(\B) \cap \ker(\B)^\bot = \{0\},
	\]
	which proves the injectivity. 
	
	It remains to show that $(\mathrm{id}-\Pi)\widetilde{\B}^{-1}$ is the inverse of $\B$. Lemma~\ref{lemma:Btildeinv_prop} and Proposition~\ref{prop:imBkerB} yield that
	\[ 
	\B (\mathrm{id}-\Pi) \widetilde{\B}^{-1} f =   \B \widetilde{\B}^{-1}f  = f, \quad f\in \ker(\B)^\bot=\im (\B).
	\]
	Therefore, 
	$ 
	(\mathrm{id}-\Pi) \widetilde{\B}^{-1} \colon \im(\B) \to \mathrm D(\T) \cap  \ker(\B)^\bot
	$ 
	is a right-inverse of 
	$
	\B \colon \mathrm D(\T) \cap  \ker(\B)^\bot \to \im(\B).
	$
	But since the latter mapping is already known to be bijective, we can conclude that \mbox{$\B^{-1} = 	(\mathrm{id}-\Pi) \widetilde{\B}^{-1} $}.
\end{proof}
We now show a similar result for $\B^2$. 
\begin{lemma}\label{lemma:Bsqinv}
	The operator 
	\[ 
		\B^{2} \colon \mathrm D(\T^2) \cap \ker(\B^2)^\bot \to \im(\B^2)
	\]
	is bijective. Its inverse is bounded on $\im(\B^2)$, symmetric, conserves $w$-parity, and is given by 
	\[ 
		\left (\B^2 \right )^{-1} = \B^{-1} \B^{-1} \eqqcolon \B^{-2}. 
	\] 
\end{lemma}
\begin{proof}	
	We consider  $f \in \im(\B^2) = \im(\B) $. Then $\B^{-1}f \in \ker(\B)^\bot \cap \mathrm D (\T)$ by Lemma~\ref{lemma:Binv}. 
	Together with Proposition~\ref{prop:imBkerB} we further deduce that
	$\B^{-2}f \in \ker(\B)^\bot \cap \mathrm D (\T)$. 
	Moreover, $\B \B^{-2}f \in \mathrm D(\T)$ implies that $\T \B^{-2}f \in \mathrm D(\T)$ similarly to Lemma~\ref{lemma:Bselfadjoint}~\ref{it:Bsa1} using Lemma~\ref{lemma:source_bounded}~\ref{it:scbd2}.
	Because of $\ker(\B)^\bot = \ker(\B^2)^\bot$, $\B^{-2} \colon \im(\B^2) \to  \mathrm D(\T^2) \cap \ker(\B^2)^\bot$ is well-defined. The fact that $\B^{-2}$ is bounded and the conservation of $w$-parity follow from Lemma~\ref{lemma:Btildeinv_prop} and Lemma~\ref{lemma:Binv}. 
	
	We now prove that $\B^{-2}$ is indeed the inverse of $\B^2$. Firstly, for $f\in \im(\B^2) = \im(\B)$ we immediately conclude that
	\[ 
	\B^2 \B^{-2}f= \B (\B \B^{-1} ) \B^{-1} f  = \B \B^{-1}f = f
	\]  
	since $\B^{-1} f\in \im(\B)$. 	Secondly, for $f \in \mathrm D(\T^2) \cap \ker(\B^2)^\bot $ we have
	\[ 
	\B^{-2} \B^2 f = {\B}^{-1} \left ( \B^{-1} \B \right ) \B f = {\B}^{-1} \B f = f 
	\]
	since $\B f \in \mathrm D(\T) \cap \ker(\B)^\bot$ and $f\in \mathrm D(\T)\cap \ker(\B)^\bot$ by Proposition~\ref{prop:imBkerB} and Lemma~\ref{lemma:kernelBsq}.	Consequently, $\B^{-2} = \left ( \B^2 \right )^{-1}$.
	
	The symmetry of $\B^{-2}$ can then be deduced from the self-adjointness of $\B^2$ and
	\[
		\langle f,\B^{-2}g\rangle_H=\langle\B^2\B^{-2}f,\B^{-2}g\rangle_H=\langle\B^{-2}f,g\rangle_H,\quad f,g\in \im(\B^2).\qedhere
	\]
\end{proof}

In particular, from the boundedness of $\B^{-1}$ we deduce a Poincar\'e-type inequality which can then be applied to establish the semi-boundedness of the spectrum of $\B^2$. An analogous estimate was shown in the non-relativistic case in \cite[Corollary~5.8]{HaReSt21}; however, the arguments there is simpler since the spectrum of the non-relativistic analogue of~$\B^2$ is explicitly known. 
\begin{cor}\label{cor:poincareB}
	There exists $C>0$ such that 
	\begin{equation}\label{eq:poincareB}
		\|\B f\|_{H} \geq C \|f\|_{H}, \quad f \in \mathrm D(\T) \cap \ker(\B)^\bot.
	\end{equation}
	Moreover, the spectrum of the self-adjoint operator $-\B^2|_{\H}\colon\mathrm D(\T^2)\cap\H\to\H$ is bounded from below by some $\epsilon>0$, i.e.,
	\[ 
	\sigma(-\B^2|_{\H} ) \subset[\epsilon,\infty[.
	\]
\end{cor}
\begin{proof}
	The estimate~\eqref{eq:poincareB} is an immediate consequence of the boundedness of $\B^{-1}$ shown in Lemma~\ref{lemma:Binv}. As for the spectral bound, recall that $\H\subset\ker(\B)^\perp$ by Lemma~\ref{lemma:characterkerBbot}. Thus, the skew-symmetry of $\B$ implies that
	\[
	\langle-\B^2f,f\rangle_{H}=\|\B f\|_H^2\geq C^2\|f\|_H^2,\quad f\in\mathrm D(\T^2)\cap\H,
	\]
	from which we obtain the spectral bound using~\cite[Prop.~5.12]{HiSi}; note that $-\B^2|_{\H}$ is already known to be self-adjoint by Lemma~\ref{lemma:Bselfadjoint}~\ref{it:Bsa1}. 
\end{proof}
To summarize, we collect all the derived properties of $\B$ for future reference in the following proposition, which should be compared to the similar results for the pure transport operator $\T$ in Proposition~\ref{prop:transport_char}.
\begin{proposition}\phantomsection\label{prop:Bproperties}
	\begin{enumerate}[label=(\alph*)]
		\item\label{it:Bprop1} $\B \colon \mathrm D(\T) \to H$ is well-defined and skew-adjoint as a densely defined operator on $H$. Moreover, $\B^2 \colon \mathrm D(\T^2) \to H$ is self-adjoint.
		\item\label{it:Bprop2} The kernel of $\T$ maps one-to-one to the kernel of $\B$ and
		\begin{equation}\label{eq:kernel_B}
			\ker(\B) = \left\{ g + 4\pi |\varphi'|E e^{-\lambda_0-\mu_0} \int_{r}^{\Rmax} e^{(3\lambda_0+\mu_0)(s)} p_g(s) s \,ds \mid g  \in \ker \T \right\} .
		\end{equation}
		\item\label{it:Bprop3} It holds that  $\im(\B)=\im(\B^2)$, $\ker(\B) = \ker(\B^2)$, $\H \subset \ker(\B)^\bot=\im(\B)$, and 
		\begin{equation}
			 \ker(\B)^\bot=  \left\{ f\in H \mid 	\int_0^1 \left (f + |\varphi'| e^{2\mu_0(R)} \lambda_f(R) \frac{W^2}{E}\right ) \, d\theta = 0 \text{ for a.e.\ } (E,L)\in \Omega_0^{EL} \right\},
		\end{equation}
		where the notation introduced at the start of Section~\ref{ssc:kerB_kerBbot} is used.
		\item\label{it:Bprop4} $\B$ reverses $w$-parity, i.e., $(\B f)_{\pm} = \B(f_{\mp})$ for $f\in \mathrm D(\T)$. Moreover, $\B^2$ conserves $w$-parity and the restricted operator $\B^2=\B^2|_{\H}\colon\mathrm D(\T^2)\cap\H\to\H$ is self-adjoint as a densely defined operator on $\H$. 
		\item\label{it:Bprop5}  The operators $\B \colon \mathrm D(\T) \cap \ker(\B)^\bot  \to \im(\B)$  and $\B^2 \colon \mathrm D(\T^2) \cap \ker(\B^2)^\bot  \to \im(\B^2)$ are bijective. The bounded and symmetric inverses are given by
		\begin{equation}\label{eq:Binv_Bsqinv}
			\B^{-1} = (\mathrm{id}-\Pi)\widetilde{\B}^{-1}, \quad (\B^2)^{-1} = \B^{-1} \B^{-1}
		\end{equation}
		with $\widetilde{\B}^{-1}$ explicitly defined in Definition~\ref{def:Btildeinv} and $\Pi\colon H\to H$ being the orthogonal projection onto $\ker(\B)$ introduced at the beginning of Section~\ref{ssc:inverseB_inverseBsq}.
		\item\label{it:Bprop6} There exists $\epsilon > 0$ such that 
		\begin{equation}
			\sigma(-\B^2|_{\H} ) \subset[\epsilon,\infty[.
		\end{equation} 		
	\end{enumerate}
\end{proposition}
\begin{proof}
	Parts~\ref{it:Bprop1} and~\ref{it:Bprop4} are proven in Lemma~\ref{lemma:Bselfadjoint}~\ref{it:Bsa1}. The explicit characterizations of $\ker(\B)$ and its orthogonal complement are derived in Lemmas~\ref{lemma:kernel_B} and~\ref{lemma:characterkerBbot}. The relations between $\ker(\B)$, $\im(\B)$, and the respective sets for $\B^2$ are due to Proposition~\ref{prop:imBkerB} and Lemma~\ref{lemma:kernelBsq}. The inverses of $\B$ and $\B^2$ are studied in Lemmas~\ref{lemma:Binv} and~\ref{lemma:Bsqinv}, while the spectral semi-boundedness of $-\B^2|_{\H}$ from part~\ref{it:Bprop6} is shown in Corollary~\ref{cor:poincareB}.
\end{proof}
\subsection{The residual operator $\Ri $}
Compared to the operator $\B$, the residual operator $\Ri $ is rather nice; recall Definition~\ref{def:TandB}~\ref{it:TBdef3}. We first show that $\Ri$ is non-negative and that we can determine its square root. 
\begin{lemma}\label{lemma:Rproperties}
	The operator $\Ri \colon H \to H$ is bounded, symmetric, and non-negative (in the sense of quadratic forms), i.e., $ \langle \Ri f, f \rangle \geq 0$ for $f\in H$. The operator 
	\[ 
		\sqrt \Ri \colon H \to H, \quad \sqrt{\Ri}f\coloneqq4\pi \sqrt{r} |\varphi'| e^{2\mu_0+\lambda_0} \sqrt{\frac{2r\mu_0'+1}{\mu_0'+\lambda_0'}} \, w j_f,
	\]
	is bounded, symmetric, non-negative, and on $H$ we have that $\sqrt{\Ri} \sqrt{\Ri}  = \Ri $. Moreover, $\sqrt{\Ri}f\in\H$ and $\Ri f\in\H$ for $f\in H$.  
\end{lemma}
\begin{proof}
	$\Ri$ is well-defined and bounded because of Lemma~\ref{lemma:source_bounded}~\ref{it:Bsa1} and since the metric coefficients and $\Omega_0$ are bounded. Next, we compute that
	\[
		\langle \Ri f, g \rangle_H =  16\pi^2 \int_{\Rmin}^{\Rmax}  e^{3\mu_0+\lambda_0} (2r\mu_0'+1)  j_f j_g  r^2 \, dr,\quad f,g\in H,
	\]
	which implies the symmetry and non-negativity of $\Ri$ since $\mu_0'\geq 0$, see~\eqref{eq:fieldeq2}. Similar arguments using~\eqref{eq:HLRidentity} yield the claims for $\sqrt \Ri$; note that $\mu_0'+\lambda_0'=4\pi re^{2\lambda_0}(\rho_0+p_0)>0$ for $r>0$ in the radial support of the steady state. 
	The oddness in $w$ of $\Ri f$ and $\sqrt{\Ri}f$ is due to the fact that $\varphi'=\varphi'(E(r,w,L),L)$ is even-in-$w$.
\end{proof}
In order to apply the Birman-Schwinger principle, we need that $\sqrt \Ri $ is relatively $\B ^2$-compact on the space of odd-in-$w$ functions, i.e., 
\[ 
\sqrt \Ri  \B^{-2} : \H \to \H 
\]
is compact, see \cite[Definition 14.1]{HiSi}; note that this operator is well-defined by Proposition~\ref{prop:Bproperties}~\ref{it:Bprop3},~\ref{it:Bprop5}. This relative compactness will allow us to control the essential spectrum of $\L$ later on.
\begin{lemma}\label{lemma:relcompact_Bsq_R}
	The operator $\sqrt\Ri|_{\H} $ is relatively $(\B^2|_{\H})$-compact. 
\end{lemma}
\begin{proof}
	Recall that $\B^2|_{\H}$ is self-adjoint with non-empty resolvent set by Proposition~\ref{prop:Bproperties}~\ref{it:Bprop4},~\ref{it:Bprop6}.  Hence, the relative $(\B^2|_{\H})$-compactness of $\sqrt{\Ri}|_{\H}$ is equivalent to
	\[ 
	\sqrt\Ri\colon \left(\mathrm D(\T ^2)\cap \H ,~ \|\B^2 \cdot\|_H + \|\cdot \|_H\right) \to \H
	\]
	being compact, cf.~\cite[III~Definition~2.15, Exercise~2.18.(1)]{EngNa2000}. Let $(f_n)_{n\in\N}\subset\mathrm D(\T^2)\cap\H$ be a sequence such that $(f_n)_{n\in\N}$ and $(\B ^2 f_n)_{n\in\N}$ are bounded in $\H$. The Poincar\'e-type estimate~\eqref{eq:poincareB} implies that $(\B f_n)_{n\in\N}$ is bounded in $H$ as well; note that $\B f_n\in\mathrm D(\T)\cap\im(\B)$. From~\eqref{eq:lambdadotB}, the compactness of $H\ni f \mapsto \lambda_f \in L^2([\Rmin,\Rmax])$, see Lemma~\ref{lemma:source_bounded}~\ref{it:scbd1}, and~\eqref{eq:HLRidentity} we obtain that $\sqrt \Ri f_n$ converges strongly in $\H$ up to a subsequence.
\end{proof}

\section{The Birman-Schwinger principle}\label{sc:birman_schwinger_principle}

We now apply a Birman-Schwinger type principle in order to derive a criterion for the existence of negative eigenvalues of $\L=-\B^2-\Ri$ which can be used for a large variety of steady states. Recall that the class of steady states used here is specified in Section~\ref{sc:ststconditions} and that all the operators and function spaces are defined in Section~\ref{sc:def_operators}. As a preparation we consider the following auxiliary family of operators:
\begin{definition}\label{def:L_delta}
	For $\gamma>0$ let
	\begin{equation*}
		\L_\gamma\coloneqq-\B^2-\frac1\gamma\Ri\colon \mathrm D(\T^2)\cap\H\to\H.
	\end{equation*}
\end{definition}
Note that we work solely on the space $\H$ of odd-in-$w$ functions since the Antonov operator $\L=\L_1$ covers only the evolution of the odd-in-$w$ part of the perturbation.

In order to arrive at a Birman-Schwinger principle, we have to analyze the functional analytic and spectral properties of these operators, in particular, the dependency of the spectrum on $\gamma$. This rather abstract investigation is conducted in the following section, the derivation of the actual Birman-Schwinger operator is then performed in Section~\ref{ssc:birman_schwinger_operator}.

\subsection{Analysis of the operators $\mathcal L_\gamma$}\label{appendix}

We first show that the operators are self-adjoint and that the essential spectrum $\sigma_{ess}$ of $\L_\gamma$ is independent of $\gamma$. Broadly speaking, the essential spectrum contains all the elements of the spectrum which are are not isolated eigenvalues of finite multiplicity; see, e.g., \cite[Chapter~7]{HiSi}.

\begin{lemma}\label{lemma:ess_spectrum_L}
	For $\gamma>0$ the operator $\L_\gamma\colon \mathrm D(\T^2)\cap\H\to\H$ is self-adjoint as a densely defined operator on $\H$ with essential spectrum given by
	\begin{equation*}
		\sigma_{ess}(\L_\gamma)= \sigma_{ess}(-\B^2|_{\H}) =\sigma_{ess}(\L).
	\end{equation*}
	Moreover, $\inf(\sigma_{ess}(\L))>0$. 
\end{lemma} 
\begin{proof}
	$\L_\gamma = -\B^2 - \frac1\gamma\Ri$ is well-defined by Proposition~\ref{prop:Bproperties}~\ref{it:Bprop4} and Lemma~\ref{lemma:Rproperties}. Furthermore, since $\B^2|_{\H}$ is self-adjoint and $\Ri$ is bounded and symmetric, $\L_\gamma$ is self-adjoint by the Kato-Rellich theorem \cite[Thm.~X.12]{ReSi2}.
	Lemma~\ref{lemma:relcompact_Bsq_R} implies that $\Ri|_{\H}=\sqrt{\Ri}\sqrt{\Ri}|_{\H}$ is relatively $(\B^2|_{\H})$-compact and thus the essential spectrum of $\L_\gamma$ equals the one of $-\B^2|_{\H}$ by Weyl's theorem~\cite[Thm.~14.6]{HiSi}. The fact that $\sigma_{ess}(\L)$ is bounded away from zero follows from Proposition~\ref{prop:Bproperties}~\ref{it:Bprop6}. 
\end{proof}

\begin{remark}
	The essential spectrum of the Antonov operator $\L$ can be determined explicitly since $\sigma_{ess}(-\B^2|_{\H})=\sigma_{ess}(-\T^2|_{\H})$ by the Weyl theorem along with further techniques. Using action-angle type variables then allows one to explicitly determine the essential spectrum of the transport operator similar to~\cite[Thm.~5.7]{HaReSt21} or~\cite[Lemma~B.12]{Kunze}. Moreover, the spectrum of $ -\T^2|_{\H}$ is purely essential. It is an open problem if this is the case with $-\B^2|_{\H}$ as well.
\end{remark}

Hence, it remains to analyze the behavior of the isolated eigenvalues of $\L_\gamma$ when varying~$\gamma$. This is achieved by a variational characterization of these eigenvalues:
\begin{defnprop}\label{thm:minmaxLdelta}
	For $\gamma>0$ and $n\in\N$ let 
	\begin{align*}
		\mu_n(\gamma)\coloneqq \sup_{g_1,\ldots,g_{n-1}\in\H}\left( \inf_{\substack{h\in\mathrm D(\T^2)\cap\H,~\|h\|_H=1,\\ h\perp g_1,\ldots,g_{n-1}}}\langle h,\L_\gamma h\rangle_H \right).
	\end{align*}
	Then $\mu_n(\gamma)$ is finite, and either
	\begin{enumerate}
		\item[(i)] $\mu_n(\gamma)<\inf\left(\sigma_{ess}(-\B^2|_{\H})\right)$. In this case there exist at least $n$ eigenvalues (counting multiplicities) of $\L_\gamma$ below  $\inf\left(\sigma_{ess}(-\B^2|_{\H})\right)$, and $\mu_n(\gamma)$ is the $n$-th smallest eigenvalue (counting multiplicities) of $\L_\gamma$.
	\end{enumerate}
	or
	\begin{enumerate}
		\item[(ii)] $\mu_n(\gamma)=\inf\left(\sigma_{ess}(-\B^2|_{\H})\right)$. In this case there exist at most $n-1$ eigenvalues (counting multiplicities) of $\L_\gamma$ below  $\inf\left(\sigma_{ess}(-\B^2|_{\H})\right)$, and $\mu_{n+j}(\gamma)=\inf\left(\sigma_{ess}(-\B^2|_{\H})\right)$ for $j\in\N$.
	\end{enumerate}
\end{defnprop}
\begin{proof}
	First note that $\L_\gamma$ is bounded from below since $-\B^2\geq0$ and $\Ri$ is bounded, more precisely,
	\begin{equation*}
		\langle h,\L_\gamma h\rangle_H = \|\B h\|_H^2-\frac1\gamma\langle h,\Ri h\rangle_H\geq -\frac{\|\Ri\|_{H\to H}}\gamma
	\end{equation*}
	for $h\in \mathrm D(\T^2)\cap\H$ with $\|h\|_H=1$. Then the statement is simply the min-max principle for semi-bounded, self-adjoint operators, see~\cite[Thm.~XIII.1]{ReSi4} or~\cite[Prop.~II.32]{Si}.
\end{proof}
Before proceeding we want to make clear what we mean by the \enquote{multiplicity} of an eigenvalue of a self-adjoint operator, see, e.g.,~\cite[Sc.~7.1]{HiSi} for a detailed discussion.
\begin{remark}\label{thm:multiplicity}
	Let $A$ be a self-adjoint operator defined on a dense subset $ \mathrm D(A)$ of some Hilbert-space $V$. Let $\lambda\in\R$ be an eigenvalue of $A$, i.e., there exists $v\in\mathrm D(A)\setminus\{0\}$ such that $Av=\lambda v$. The {\em multiplicity} of $\lambda$ is defined as $\dim\left(\ker\left(A-\lambda\,\mathrm{id}\right)\right)\in\N\cup\{\infty\}$. Note that the multiplicity of $\lambda$ is finite if $\lambda\notin\sigma_{ess}(A)$.
\end{remark}

Our goal is to understand the properties of the mappings $]0,\infty[\ni\gamma\mapsto\mu_n(\gamma)$ for $n\in\N$. The following lemma is related to~\cite[XIII~Problem~2]{ReSi4} and~\cite[Thm.~II.33]{Si}.
\begin{lemma}\label{thm:evmonoton}
	For fixed $n\in\N$ the mapping $]0,\infty[\ni\gamma\mapsto\mu_n(\gamma)$ is non-decreasing and 
	\begin{equation}
		\left|\mu_n(\gamma)-\mu_n(\beta)\right|\leq\left|\frac1\gamma-\frac1\beta\right|\,\|\Ri\|_{H\to H}\label{eq:estevs}
	\end{equation}
	for $\gamma,\beta>0$. In particular, $]0,\infty[\ni\gamma\mapsto\mu_n(\gamma)$ is continuous.
\end{lemma}
\begin{proof}
	For $\gamma>0$ and $h\in\mathrm D(\T^2)\cap\H$ with $\|h\|_H=1$ let
	\begin{equation*}
		f_h(\gamma)\coloneqq\langle h,\L_\gamma h\rangle_H.
	\end{equation*}
	We first prove the claimed properties for $f_h$. For $\gamma<\beta$, $\Ri\geq0$ in the sense of quadratic forms (cf.\ Lemma~\ref{lemma:Rproperties}) implies that
	\begin{equation*}
		f_h(\gamma)=\|\B h\|_H^2-\frac1\gamma\,\langle h,\Ri h\rangle_H\leq\|\B h\|_H^2-\frac1\beta\,\langle h,\Ri h\rangle_H=f_h(\beta)
	\end{equation*}
	for $h$ as above. Furthermore, by the Cauchy-Schwarz inequality,
	\begin{equation*}
		\left|f_h(\gamma)-f_h(\beta)\right|=\left|\frac1\gamma-\frac1\beta\right|\,\langle h,\Ri h\rangle_H\leq \left|\frac1\gamma-\frac1\beta\right|\,\|\Ri\|_{H\to H}
	\end{equation*}
	for any $\gamma,\beta>0$.
	Now the monotonicity and estimate easily carry over from $f_h$ to the sup-inf in the definition of $\mu_n$. 
\end{proof}
The monotonicity of $\mu_n$ corresponds to the fact that decreasing $\gamma>0$ means that we assign more weight to the non-positive term $-\frac1\gamma\,\Ri$ of the operator $\L_\gamma$, which leads to the spectrum of $\L_\gamma$ to be shifted towards more negative values. In fact, the monotonicity from the previous lemma is even strict if $\mu_n$ departs from $\inf\left(\sigma(-\B^2|_{\H})\right)>0$; recall Proposition~\ref{prop:Bproperties}~\ref{it:Bprop6}.

\begin{lemma}\label{thm:evstrictlyincreasing}
	Fix $n\in\N$ and suppose that there exists $\gamma_0>0$ such that $\mu_n(\gamma_0)<\inf\left(\sigma(-\B^2|_{\H})\right)$. Then $]0,\gamma_0]\ni\gamma\mapsto\mu_n(\gamma)$ is (strictly) increasing.
\end{lemma}
\begin{proof}
	The following proof is related to \cite[proof of Thm.~12.1]{LiLo01}.
	First observe that Proposition~\ref{thm:minmaxLdelta} and Lemma~\ref{thm:evmonoton} imply that $\mu_j(\gamma)<\inf\left(\sigma(-\B^2|_{\H})\right)$ for all $1\leq j\leq n$ and $0<\gamma\leq\gamma_0$. In particular, for all such $j$ and $\gamma$ we know that $\mu_j(\gamma)$ is an eigenvalue of $\L_\gamma$. Choosing orthonormal eigenfunctions, we deduce that for every $0<\gamma\leq\gamma_0$ and $1\leq j\leq n$ there exists $h_j^\gamma\in \mathrm D(\T^2)\cap\H$ such that $\L_\gamma h_j^\gamma=\mu_j(\gamma)\,h_j^\gamma$, $\|h_j^\gamma\|_H=1$, and $h_i^\gamma\perp h_j^\gamma$ for $i\neq j$.
	
	Now fix $0<\gamma<\beta\leq\gamma_0$ and let $c_1,\ldots,c_n\in\R$ be such that $\tilde h\coloneqq\sum_{j=1}^nc_jh_j^\beta$ satisfies $1=\|\tilde h\|_H^2=\sum_{j=1}^nc_j^2$ and $\tilde h\perp h_1^\gamma,\ldots,h_{n-1}^\gamma$. Then
	\begin{equation}\label{eq:tildehquadraticform}
		\langle \tilde h,\L_\beta\tilde h\rangle_H = \sum_{j=1}^nc_j^2\,\mu_j(\beta)\leq\mu_n(\beta)\,\sum_{j=1}^nc_j^2 = \mu_n(\beta)<\inf\left(\sigma(-\B^2|_{\H})\right),
	\end{equation}
	which implies that $\langle\tilde h,\Ri\tilde h\rangle_H>0$ since $\|\B\tilde h\|_H^2\geq \inf\left(\sigma(-\B^2|_{\H})\right)$, cf.~\cite[Prop.~5.12]{HiSi}. Thus,
	\begin{equation}\label{eq:tildehquadraticformstrict}
		\langle \tilde h,\L_\gamma\tilde h\rangle_H=\|\B\tilde h\|_H^2-\frac1\gamma\langle\tilde h,\Ri\tilde h\rangle_H < \|\B\tilde h\|_H^2-\frac1\beta\langle\tilde h,\Ri\tilde h\rangle_H=\langle \tilde h,\L_\beta\tilde h\rangle_H.
	\end{equation}
	The last step is to observe that the supremum in $\mu_n(\gamma)$, see Definition~\ref{thm:minmaxLdelta}, is attained when choosing $g_1,\ldots,g_{n-1}=h_1^\gamma,\ldots,h_{n-1}^\gamma$.
	Hence, using~\eqref{eq:tildehquadraticform} and~\eqref{eq:tildehquadraticformstrict} yields that
	\begin{equation*}
		\mu_n(\gamma)=\inf_{\substack{h\in \mathrm D(\T^2)\cap\H,~\|h\|_H=1,\\ h\perp h_1^\gamma,\ldots,h_{n-1}^\gamma}}\langle h,\L_\gamma h\rangle_H\leq\langle\tilde h,\L_\gamma\tilde h\rangle_H<\langle\tilde h,\L_\beta\tilde h\rangle_H\leq\mu_n(\beta).\qedhere
	\end{equation*}
\end{proof}
We have seen in the proof of Lemma~\ref{thm:evstrictlyincreasing} that strict monotonicity of an eigenvalue can only be expected if an eigenvalue $\mu_n(\gamma)$ departs from $\inf\left(\sigma(-\B^2|_{\H})\right)>0$ as $\gamma$ decreases. We do not know if this happens for $n\geq2$, but we get the following result for the smallest eigenvalue:
\begin{lemma}\label{thm:ev1deltatominusinfty}
	It holds that
	$\lim_{\gamma\to0}\mu_1(\gamma)=-\infty$.	
\end{lemma} 
\begin{proof}
	Since 
	\begin{equation*}
		\mu_1(\gamma)=\inf_{\substack{h\in \mathrm D(\T^2)\cap\H,\\\|h\|_H=1}}\langle h,\L_\gamma h\rangle_H = \inf_{\substack{h\in \mathrm D(\T^2)\cap\H,\\\|h\|_H=1}}\left ( \|\B h\|_H^2-\frac1\gamma\langle h,\Ri h\rangle_H\right )
	\end{equation*}
	for $\gamma>0$, we just have to fix some $\tilde h\in\mathrm D(\T^2)\cap\H$ with $\|\tilde h\|_H=1$ and $\langle \tilde h,\Ri\tilde h\rangle_H\neq0$ to deduce that
	\begin{equation*}
		\mu_1(\gamma)\leq \|\B\tilde h\|_H^2-\frac1\gamma\langle\tilde h,\Ri\tilde h\rangle_H\to-\infty ,\quad\gamma\to0.\qedhere
	\end{equation*} 
\end{proof}
On the other hand, the limiting behavior for the eigenvalues $\mu_n(\gamma)$ as $\gamma$ goes to infinity is rather simple:
\begin{lemma}\label{thm:evsdeltatoinfty}
	For every $n\in\N$, $\lim_{\gamma\to\infty}\mu_n(\gamma)\geq\inf\left(\sigma(-\B^2|_{\H})\right)>0$.
\end{lemma}
\begin{proof}
	First note that the limit exists by Lemma~\ref{thm:evmonoton}. Furthermore, for every $h\in  \mathrm D(\T^2)\cap\H$ with $\|h\|_H=1$ we have the estimate
	\begin{equation*}
		\langle h,\L_\gamma h\rangle_H=\|\B h\|_H^2-\frac1\gamma\langle h,\Ri h\rangle_H\geq \inf\left(\sigma(-\B^2|_{\H})\right)-\frac{\|\Ri\|_{H\to H}}\gamma.\qedhere
	\end{equation*}
\end{proof}
The monotonicity of $\mu_n$ allows us to translate the number of negative eigenvalues of $\L=\L_1$ into the position of the zeros of the mappings $]0,\infty[\ni\gamma\mapsto\mu_n(\gamma)$ for $n\in\N$. This is why we define the following quantities:
\begin{defnrem}\label{def:gammastar}
	For $n\in\N$ we define $\gamma^\ast_n\in[0,\infty[$ as follows:
	\begin{enumerate}[label=(\roman*)]
		\item\label{it:defgammastar1} If $\mu_n(\gamma)>0$ for every $\gamma>0$, let $\gamma^\ast_n\coloneqq0$.
		\item\label{it:defgammastar2} Otherwise, define $\gamma^\ast_n$ via $\mu_n(\gamma^\ast_n)=0$.
	\end{enumerate}
	In the case \ref{it:defgammastar2}, $\gamma^\ast_n>0$ is uniquely determined due to the Lemmas~\ref{thm:evmonoton}, \ref{thm:evstrictlyincreasing}, and \ref{thm:evsdeltatoinfty}. By Proposition~\ref{thm:minmaxLdelta} zero is an eigenvalue of $\L_{\gamma^\ast_n}$.
\end{defnrem}
Note that $\gamma_{n+1}^\ast\leq\gamma_n^\ast$ for $n\in\N$. Furthermore, the analysis of the present section yields the following key result:
\begin{proposition}\label{thm:evsLanddeltastar} It holds that
	\begin{equation*}
		\#\{\mathrm{negative~eigenvalues~of~}\L\mathrm{~(counting~multiplicities)}\}=\#\{n\in\N\mid\gamma_n^\ast>1\}.
	\end{equation*}
	Here, negative means $<0$, and the multiplicity of an eigenvalue is explained in Remark~\ref{thm:multiplicity}. Note that each negative eigenvalue of $\L$ has finite multiplicity since $\inf\left(\sigma_{ess}(\L)\right)>0$. Nonetheless, the number of negative eigenvalues could be infinite in principal.
\end{proposition}
\begin{proof}
	By Definition~\ref{thm:minmaxLdelta},
	\begin{equation*}
		\#\{\text{negative eigenvalues of }\L\text{ (counting multiplicities)}\}=\#\{n\in\N\mid\mu_n(1)<0\}.
	\end{equation*} 
	Lemmas~\ref{thm:evmonoton}, \ref{thm:evstrictlyincreasing}, and \ref{thm:evsdeltatoinfty} imply that $\mu_n(1)<0$ for some $n\in\N$ is equivalent to $\gamma_n^\ast>1$.
\end{proof}

\subsection{The Birman-Schwinger operator}\label{ssc:birman_schwinger_operator}

We have now collected all necessary tools to establish the connection between the spectrum of $\L$ and the following operator.

\begin{definition}
 	The operator
	\begin{equation*}
		Q\coloneqq-\sqrt\Ri\,\B^{-2}\,\sqrt\Ri\colon\H\to\H
	\end{equation*}
	is the {\em Birman-Schwinger operator} associated to $\L$. 
\end{definition}
Recall Proposition~\ref{prop:Bproperties}~\ref{it:Bprop5} and Lemma~\ref{lemma:Rproperties} for the definition of $\B^{-2}$ and $\sqrt{\Ri}$, respectively, and note that $Q$ is well-defined since $\H\subset\ker(\B^2)^\perp$ by Proposition~\ref{prop:Bproperties}~\ref{it:Bprop3}.
Observe that $Q$ looks different from the analogue operator $Q_\lambda$ with $\lambda=0$ in \cite[(8.1)]{HaReSt21}; the operator there corresponds to $\Ri\,\B^{-2}$ in our setting. However, using the square root of $\Ri$ has the advantage that our $Q$ is symmetric, which is not the case for $Q_\lambda$ from \cite{HaReSt21}. Such a \enquote{symmetric Birman-Schwinger operator} is common in quantum mechanics, see, e.g.,~\cite[Thm.~XIII.10]{ReSi4} or \cite[Thm.~12.4]{LiLo01}. 

We now derive the connection between the eigenvalues of $\L_\gamma$ and $Q$.

\begin{proposition}[Birman-Schwinger principle]\label{thm:Birman}
	Let $\gamma>0$. Then $0$ is an eigenvalue of $\L_\gamma$ if and only if $\gamma$ is an eigenvalue of $Q$. 
	
	In this case, the multiplicities of the these eigenvalues are equal, and the associated eigenfunctions can be transformed explicitly into one another:
	\begin{enumerate}[label=(\alph*)]
		\item If $f\in \mathrm D(\T^2)\cap\H$ is an eigenfunction of $\L_\gamma$ to the eigenvalue $0$, then
		\begin{equation}\label{eq:hfromg}
			g\coloneqq\sqrt\Ri\,f \in\H
		\end{equation}
		defines an eigenfunction of $Q$ to the eigenvalue $\gamma$. 
		\item If $g\in\H$ is an eigenfunction of $Q$ to the eigenvalue $\gamma$, then
		\begin{equation}\label{eq:gfromh}
			f\coloneqq -\B^{-2}\sqrt\Ri\,g\in\mathrm D(\T^2)\cap\H
		\end{equation}
		defines an eigenfunction of $\L_\gamma$ to the eigenvalue $0$. 
	\end{enumerate}
\end{proposition} 
\begin{proof}
	Let $f\in \mathrm D(\T^2)\cap\H$ be a solution of $\L_\gamma f=0$, i.e., $-\gamma\,\B^2f=\Ri\,f$. Applying $-\sqrt\Ri\,\B^{-2}$ onto the latter equation and writing $\Ri=\sqrt\Ri\,\sqrt\Ri$ then yields that
	\begin{equation*}
		\gamma\,g = \gamma\,\sqrt\Ri\,f= Q\left(\sqrt\Ri\,f\right)=Qg,
	\end{equation*}  
	with $g$ defined by~\eqref{eq:hfromg}. Moreover, the eigenfunction identity for $f$ can be written as $f=-\frac1\gamma\B^{-2}\Ri f=-\frac1\gamma\B^{-2}\sqrt{\Ri}g$, which shows $g\neq0$ since $f\neq0$.  
	
	Conversely, if $g\in\H$ is a solution of the eigenvalue equation $Qg=\gamma g$ and $f\in \mathrm D(\T^2)\cap\H$ is defined via~\eqref{eq:gfromh}, then 
	\begin{equation*}
		\L_\gamma f= -\B^2f-\frac1\gamma\,\Ri f=\sqrt\Ri\,g-\frac1\gamma\,\sqrt\Ri\,Qg=0.
	\end{equation*} 
	In particular, applying $\sqrt{\Ri}$ on~\eqref{eq:gfromh} yields that $\sqrt{\Ri}f=Qg=\gamma g\neq0$, i.e., $f\neq0$. 
	
	Linear independence of eigenfunctions is preserved by the transformations \eqref{eq:hfromg} and \eqref{eq:gfromh}.
\end{proof}
Although Proposition~\ref{thm:Birman} only establishes a connection between zero eigenvalues of $\L_\gamma$ and eigenvalues of $Q$, we can apply Proposition~\ref{thm:evsLanddeltastar} to infer the following quantitative control on the number of negative eigenvalues of $\L=\L_1$. 
\begin{proposition}\label{thm:evsLandQ} It holds that
	\begin{align*}
		\#\{\mathrm{negative~eigenvalues~of~}\L\}=\#\{\mathrm{eigenvalues}>1\mathrm{~of~}Q\}.
	\end{align*}
	In both sets we count the eigenvalues including their multiplicities.
\end{proposition}
At first glance, Proposition~\ref{thm:evsLandQ} does not seem to be of any help since it simply translates the original eigenvalue problem for $\L$ into another eigenvalue problem. However, from a functional analysis point of view, $Q$ is much nicer than the unbounded operator $\L$.
\begin{lemma}\label{thm:Qprop}
	The operator $Q=-\sqrt\Ri\,\B^{-2}\,\sqrt\Ri\colon\H\to\H$ is linear, bounded, symmetric, non-negative, and compact.
\end{lemma}
\begin{proof} 
	This proof relies on the properties of $\B^2$ and $\sqrt \Ri$ shown in  Proposition~\ref{prop:Bproperties} and Lemma~\ref{lemma:Rproperties}.
	$Q$ is linear, bounded, and symmetric, since $\B^{-2}$ and $\sqrt\Ri$ have these three properties.
	In addition, $\sqrt\Ri$ being relatively $\B^2$-compact, see Lemma~\ref{lemma:relcompact_Bsq_R}, means that $\sqrt\Ri\,\B^{-2}$ is a compact operator. 	Thus, $Q$ is compact as the composition of a compact and a bounded operator.
	
	To see the non-negativity of $Q$, we use the symmetry of $\sqrt \Ri$ and $\B^{-1}$ to obtain that
	\[ 
		\langle Qf,f \rangle_H = \langle -\B^{-2} \sqrt \Ri f, \sqrt \Ri f\rangle_H = \| \B^{-1} \sqrt \Ri f\|_H^2 \geq 0, \quad f\in \H. \qedhere
	\] 
\end{proof}
Of course, $Q$ having all these nice properties extends our understanding of its spectrum immensely. For example, we immediately obtain that $\sigma(Q)\setminus\{0\}$ is contained in $]0,\infty[$ and consists of discrete eigenvalues of finite multiplicity with only possible accumulation point at zero which might be an eigenvalue of infinite multiplicity. Together with Proposition~\ref{thm:evsLandQ} we now know that $\L$ can only have a finite number of negative eigenvalues.

But that is not all. The specific structure of $Q=-\sqrt\Ri\,\B^{-2}\,\sqrt\Ri$ allows us to limit the hunt for eigenvalues $>1$ of $Q$ to a reduced setting.

\subsection{The Mathur operator}\label{sc:mathur_operator}
\subsubsection{Definition of the Mathur operator}\label{ssc:mathur_definition}

The reduction process is based on the following simple observation which goes back to {Mathur}~\cite{Ma}.
\begin{remark}\label{thm:Mathursobservation}
	If $f\in\H$ is an eigenfunction of $Q=-\sqrt\Ri\,\B^{-2}\,\sqrt\Ri$ corresponding to a non-zero eigenvalue, then $f\in\im\left(\sqrt\Ri\right)$.
\end{remark} 
The beautiful thing is that functions in $\im\left(\sqrt\Ri\right)$ have a particularly nice structure. More precisely, 
\begin{equation*}
	\im\left(\sqrt\Ri\right)\subset\left\{f=f(r,w,L)=\vert\varphi'(E,L)\vert\,w\, \alpha_0(r) \,F(r)\text{ a.e.}\mid F\in L^2([\Rmin,\Rmax])\right\},
\end{equation*}
where $\Rmin$ and $\Rmax$ denote the minimal and maximal radii of the steady state and 
\begin{equation}
	\alpha_0(r) \coloneqq \frac{e^{ (\frac{\lambda_0}{2}+\frac{\mu_0}{2} )(r)}}{\sqrt{r(\lambda_0'+\mu_0')(r)}},\quad r\in]\Rmin,\infty[.
\end{equation}
Moreover, if $f,g$ are of the form $f(r,w,L)=\vert\varphi'(E,L)\vert\,w\,\alpha_0(r)\,F(r)$ and $g(r,w,L)=\vert\varphi'(E,L)\vert\,w\,\alpha_0(r)\,G(r)$, then
\begin{equation}\label{eq:imRisoL2}
	\langle f,g\rangle_H = \langle F,G \rangle_{L^2([\Rmin,\Rmax])}
\end{equation}
by~\eqref{eq:HLRidentity}. 
Based on these observations, 
the reduced operator is defined as follows.
\begin{definition}\label{def:mathur_op}
	Let $F\in L^2([\Rmin,\Rmax])$ and define $f\in\H$ by
	\begin{align}\label{eq:gseperatedG}
		f(r,w,L)\coloneqq\vert\varphi'(E,L)\vert\,w\,\alpha_0(r)\,F(r)\quad\text{for a.e.\ } (r,w,L)\in\Omega_0.
	\end{align}
	Since $Qf\in\im\left(\sqrt\Ri\right)\subset\H$, there exists a unique $G\in L^2([\Rmin,\Rmax])$ such that
	\begin{align*}
		Qf(r,w,L)=\vert\varphi'(E,L)\vert\,w\,\alpha_0(r)\,G(r)\quad\text{for a.e.\ } (r,w,L)\in\Omega_0.
	\end{align*}
	The resulting mapping
	\begin{align*}
		\M\colon L^2([\Rmin,\Rmax])\to L^2([\Rmin,\Rmax]),\;F\mapsto G
	\end{align*}
	is the {\em reduced operator} or {\em Mathur operator}.
\end{definition}
As already indicated by Remark~\ref{thm:Mathursobservation}, non-zero eigenvalues of $Q$ and $\M$ are equivalent to each other.
\begin{lemma}\label{thm:evsQandM}
	Let $\gamma\in\R\setminus\{0\}$. Then $\gamma$ is an eigenvalue of $Q$ if and only if $\gamma$ is an eigenvalue of $\M$. In this case, the multiplicities of these eigenvalues are equal.
\end{lemma}
\begin{proof}
	The equivalence of eigenvalues is essentially given by Remark~\ref{thm:Mathursobservation}. 
	The fact that the multiplicities of an eigenvalue of $Q$ and $\M$ are the same follows by~\eqref{eq:imRisoL2} since orthogonality of eigenfunctions is conserved.
\end{proof}
Having this lemma and Proposition~\ref{thm:evsLandQ} in mind, we now want to analyze the spectrum of $\M$. It seems reasonable that this is easier than the analogous spectral analysis of $Q$ since $\M$ acts on a function space consisting of functions of one variable only (compared to three variables in the case of $Q$); this is why we call $\M$ {\em reduced}.
Still, using~\eqref{eq:imRisoL2} it is easy to verify that $\M$ inherits all the functional analytic properties of $Q$ from Lemma~\ref{thm:Qprop}.
\begin{proposition}\label{thm:Mprop}
	The Mathur operator $\M\colon L^2([\Rmin,\Rmax])\to L^2([\Rmin,\Rmax])$ is a linear, bounded, symmetric, non-negative, and compact operator.
\end{proposition}

\subsubsection{Explicit representation of the Mathur operator}\label{ssc:mathur_operator_explicit}
Up to this point, the Birman-Schwinger operator and the Mathur operator are given as abstract objects which do not seem particularly useful for applications. One would expect that we have to know $\mathcal B^{-1}$ or the projection $\Pi$, defined in Section~\ref{ssc:inverseB_inverseBsq}, explicitly to infer further properties of the Birman-Schwinger operator $Q=-\sqrt{\Ri} \, \B^{-2} \sqrt{\Ri}$ similar to~\cite{HaReSt21}. Rather surprisingly, knowledge about $\widetilde{\B}^{-1}$ is sufficient to bring $\M$ into a handy form, cf.~Definition~\ref{def:Btildeinv}.

As seen in the last subsection, we need to consider functions $f\in\H$ of the form $ f= |\varphi'| w \alpha_0(r) F(r)$ with $F\in L^2([\Rmin,\Rmax])$. For such an ansatz we observe
\[ 
\sqrt \Ri \, f = \sqrt \Ri \ \left (|\varphi'| w \alpha_0 F\right ) = |\varphi' | w \beta_0 F
\]
by~\eqref{eq:HLRidentity} and Lemma~\ref{lemma:Rproperties}, where we have introduced
\begin{equation}\label{eq:d0_def}
	\beta_0(r) \coloneqq e^{\frac{3\mu_0}{2}- \frac{\lambda_0}{2}} \frac{\sqrt{2r \mu_0'+1}}{r}, \quad r\in ]\Rmin,\infty[. 
\end{equation}
In the following calculation we employ the same notational conventions as in Section~\ref{sc:properties_operators}.
\vspace{3mm}\\
\underline{\textit{Step 1: Computing ${\B}^{-1}\sqrt \Ri \, f$}}
\vspace{3mm}

\noindent Since $\sqrt{\Ri}f$ is odd in $w$, we have
\begin{equation}
	\widetilde{\B}^{-1} \sqrt \Ri \, f = \T^{-1} \sqrt \Ri \, f + 4\pi |\varphi'| E e^{-\lambda_0-\mu_0} \int_r^{\Rmax} e^{(3\lambda_0+\mu_0)(s)} p_{ \T^{-1}\sqrt \Ri \, f}(s) s \, ds; \label{eq:TinvsqrtR}
\end{equation}
recall Definition~\ref{def:Btildeinv}, $\lambda_{\sqrt{\Ri}f}=0$, and that the expression~\eqref{eq:TinvsqrtR} is even in $w$.
We now apply Proposition~\ref{prop:transport_char}~\ref{it:Tprop7} to express the first term in action-angle type variables; we use the same notation as introduced at the start of Section~\ref{ssc:kerB_kerBbot}. Having~\eqref{eq:even_in_w_theta} in mind, we only consider $\theta\in[0,\frac12]$. Then, using the oddness-in-$w$ of $\sqrt{\Ri}f$, Fubini's theorem, and changing variables via $s=R(\tau,E,L)$ or $\tau=\theta(s,E,L)$, cf.~\eqref{eq:defthetarEL}, yields that
\begin{align*}
	(\T^{-1}\sqrt \Ri \,&f)(\theta,E,L) = 	\left (\T^{-1}(|\varphi'| w \beta_0 F)\right )(\theta,E,L)\\
	&= -T(E,L) \Bigg ( \int_0^\theta |\varphi'(E,L)| W(\tau,E,L) \beta_0(R(\tau,E,L)) F(R(\tau,E,L)) \, d\tau \\
	&\quad  + \int_0^1 \tau  |\varphi'(E,L)|  W(\tau,E,L) \beta_0(R(\tau,E,L)) F(R(\tau,E,L))  d\tau\Bigg ) \\
	&= - \int_{r_-(E,L)}^{R(\theta,E,L)} |\varphi'(E,L)|  E e^{(\lambda_0-2\mu_0)(s)} \beta_0(s) F(s) \, ds \\
	&\quad  -  \int_{r_-(E,L)}^{r_+(E,L)} (2\theta(s,E,L)-1) |\varphi'(E,L)|  E e^{(\lambda_0-2\mu_0)(s)} \beta_0(s) F(s) \, ds \\
	&=  |\varphi'(E,L)|  E \int^{\Rmax}_{R(\theta,E,L)} e^{\lambda_0-2\mu_0} \beta_0 F \, ds \\
	&\quad - |\varphi'(E,L)|  E \left ( \int^{\Rmax}_{r_+(E,L)}  e^{\lambda_0-2\mu_0} \beta_0 F \, ds + 2 \int_{r_-(E,L)}^{r_+(E,L)} \theta(s,E,L)  e^{\lambda_0-2\mu_0} \beta_0 F \, ds \right ) \\
	&=  |\varphi'(E,L)|  E \int^{\Rmax}_{R(\theta,E,L)}  e^{\lambda_0-2\mu_0} \beta_0F \, ds + h_F(E,L),
\end{align*}
for $\theta \in [0,\frac 12]$ and a.e.\ $(E,L) \in \Omega_0^{EL}$ with $h_F(E,L)$ defined appropriately; note that $h_F\in H$ by~\eqref{eq:phi_prime_bound}.  As to the second term in \eqref{eq:TinvsqrtR}, eqn.~\eqref{eq:HLRidentity} yields that
\[ 
p_{ \T^{-1} \sqrt \Ri f}(s) = \frac{e^{-2\lambda_0(s)} (\lambda_0'+\mu_0')(s)}{4\pi s} \int^{\Rmax}_{s}  e^{\lambda_0-2\mu_0} \beta_0 F \, d\sigma  + p_{h_F}(s), \quad s \in [\Rmin,\Rmax].
\]
Altogether, we can now compute that
\begin{align*}
	\widetilde{\B}^{-1} \sqrt \Ri \,f &=  |\varphi'|  E  \left ( \int^{\Rmax}_{r}  e^{\lambda_0-2\mu_0} \beta_0 F \, ds +  e^{-\lambda_0-\mu_0} \int_{r}^{\Rmax} (e^{\lambda_0+\mu_0})' \int_s^{\Rmax} e^{\lambda_0-2\mu_0}  \beta_0 F \, d\sigma ds \right ) \\
	& \quad + h_F + 4\pi   |\varphi'|  E e^{-\lambda_0-\mu_0}  \int_{r}^{\Rmax} e^{3\lambda_0+\mu_0} p_{h_F} s\, ds\\
	&=  |\varphi'|  E e^{-\lambda_0-\mu_0} \int_{r}^{\Rmax}e^{2\lambda_0-\mu_0}  \beta_0 F \, ds  \\
	&\quad + h_F + 4\pi   |\varphi'|  E e^{-\lambda_0-\mu_0}  \int_{r }^{\Rmax} e^{3\lambda_0+\mu_0} p_{h_F} s\, ds
\end{align*}
after integrating by parts in the last step. The last two terms in this equation constitute an element of $\ker(\B)$ with generator $h_F$ in the sense of Lemma~\ref{lemma:kernel_B} since $h_F$ is a function of $(E,L)$ only; recall that $h_F$ is an element of $H$. In particular, since $\Pi$ is the orthogonal projection onto $\ker(\B)$, we get with \eqref{eq:Binv_Bsqinv} that
\begin{equation}\label{eq:Binv_seperate}
	\B^{-1} \sqrt \Ri \,f = (\mathrm{id}- \Pi) 	\widetilde{\B}^{-1} \sqrt \Ri \,f   =(\mathrm{id}- \Pi)\left ( |\varphi'|  E e^{-\lambda_0-\mu_0} \int_{r}^{\Rmax}e^{2\lambda_0-\mu_0}  \beta_0 F \, ds\right).
\end{equation}
\vspace{3mm}
\underline{\textit{Step 2: The Birman-Schwinger operator $Q$}}
\vspace{3mm}

\noindent We now insert this result into $Q$, i.e.,
\begin{align}
	Qf &= -\sqrt \Ri\, \B^{-2} \sqrt \Ri \, f = -4\pi \sqrt r |\varphi'| e^{2\mu_0+\lambda_0} \sqrt{\frac{2r\mu_0'+1}{\lambda_0'+\mu_0'}}\, w \, j_{\B^{-2} \sqrt \Ri \, f} \nonumber\\
	&= |\varphi'| w e^{\mu_0} \sqrt{\frac{2r\mu_0'+1}{r(\lambda_0'+\mu_0')}}\, \lambda_{\B^{-1}\sqrt \Ri\, f}. \label{eq:Qf_first}
\end{align}
For the last equality we have applied \eqref{eq:lambdadotB} to get rid of one $\B^{-1}$-term. We emphasize that this is a crucial step for our investigation: Firstly, \eqref{eq:lambdadotB} allows us to massively reduce the complexity of $Qf$ as we do not have to calculate $\B^{-2} \sqrt \Ri\, f$ which might not even be possible explicitly. Secondly, $\lambda_{\B^{-1}\sqrt \Ri\, f}$ is a nice term in itself, as it introduces an integration in $(r,w,L)$ of $\B^{-1}\sqrt \Ri\, f$ and we can thus write it as a scalar product. At first glance, this does not seem to help with the problem of not being able to determine $\B^{-1}$ explicitly. However, the symmetry of $\Pi$ facilicates the decoupling of the projection from the unknown~$F$.  More precisely, for $r\in [\Rmin,\Rmax]$,
\begin{multline*}
	\lambda_{\B^{-1}\sqrt \Ri\, f}(r)= \frac{e^{2\lambda_0(r)}}{r} \left \langle  |\varphi'|E e^{-\lambda_0-\mu_0} \mathds{1}_{[\Rmin,r]},  \B^{-1}\sqrt \Ri\, f  \right \rangle_H \\
	= \frac{e^{2\lambda_0(r)}}{r} \left \langle (\mathrm{id}-\Pi) \left ( |\varphi'|E e^{-\lambda_0-\mu_0}\mathds{1}_{[\Rmin,r]} \right ), |\varphi'|  E e^{-\lambda_0-\mu_0} \int_{s}^{\Rmax}e^{2\lambda_0-\mu_0}  \beta_0 F \, d\sigma \right \rangle_H,
\end{multline*}
where we inserted~\eqref{eq:lambdaeqexplicit} and~\eqref{eq:Binv_seperate}, and used the symmetry of $\Pi$. In addition, writing the scalar product as an integral and integrating by parts yields that
\begin{align*}
	\lambda&_{\B^{-1}\sqrt \Ri\, f}(r) = \frac{e^{2\lambda_0(r)}}{r} \int_{\Rmin}^{\Rmax} 4\pi s^2 \rho_{(\mathrm{id}-\Pi) \left ( |\varphi'|E e^{-\lambda_0-\mu_0} \mathds{1}_{[\Rmin,r]}\right )  }(s) \int_s^{\Rmax} e^{2\lambda_0-\mu_0} \beta_0  F \, d\sigma ds\\
	&= \frac{e^{2\lambda_0(r)}}{r} \int_{\Rmin}^{\Rmax} \partial_s \left (se^{-2\lambda_0}  \lambda_{(\mathrm{id}-\Pi) \left ( |\varphi'|E e^{-\lambda_0-\mu_0} \mathds{1}_{[\Rmin,r]}\right ) }(s)\right ) \int_s^{\Rmax} e^{2\lambda_0-\mu_0} \beta_0  F \, d\sigma ds \\
	&= \frac{e^{2\lambda_0(r)}}{r} \int_{\Rmin}^{\Rmax} s e^{-\mu_0(s)} \lambda_{(\mathrm{id}-\Pi) \left ( |\varphi'|E e^{-\lambda_0-\mu_0} \mathds{1}_{[\Rmin,r]}\right ) }(s) \beta_0(s) F(s) \, ds;
\end{align*}
recall~\eqref{eq:fieldeq1_lin} and note that the boundary terms at $s=\Rmin$ and $s=\Rmax$ vanish.  It turns out that we can reveal a \enquote{hidden} symmetry in the projection term, more precisely, 
\begin{multline*}
	s e^{-2\lambda_0(s)}   \lambda_{(\mathrm{id}-\Pi) \left ( |\varphi'|E e^{-\lambda_0-\mu_0} \mathds 1_{[\Rmin,r]}\right ) }(s)  \\
	= \left \langle (\mathrm{id}-\Pi) \left ( |\varphi'|E e^{-\lambda_0-\mu_0} \mathds{1}_{[\Rmin,r]}\right ), |\varphi'|  E e^{-\lambda_0-\mu_0} \mathds{1}_{[\Rmin,s]}  \right \rangle_H \eqqcolon I(r,s)
\end{multline*}
is obviously symmetric in $r,s\geq \Rmin$ since $\Pi$ is symmetric. We put these results into $\eqref{eq:Qf_first}$ and obtain that
\[ 
Qf = |\varphi'| w \frac{e^{2\lambda_0+\mu_0} }{r}\sqrt{\frac{2r\mu_0'+1}{r(\lambda_0'+\mu_0')}}\, \int_{\Rmin}^{\Rmax}  e^{\frac{\mu_0(s)}{2}+\frac{3\lambda_0(s)}{2}}\frac{\sqrt{2s\mu_0'(s)+1}}s
  \, I(r,s)  F(s) \, ds;
\]
recall the definition of $\beta_0$ in~\eqref{eq:d0_def}.
\vspace{3mm} \\
\underline{\textit{Step 3: The Mathur operator $\M$}}
\vspace{3mm}

\noindent In order to explicitly derive the Mathur operator $\M$ introduced in Definition~\ref{def:mathur_op}, we adjust for the factor $\alpha_0$ and get the following result: 

\begin{proposition}\label{prop:Mathur_Kernel}
	For  $G\in L^2([\Rmin,\Rmax])$ we have
	\[ 
	(\M G)(r) =  \int_{\Rmin}^{\Rmax} K(r,s) G(s) \, ds, \quad r\in[\Rmin,\Rmax],
	\]
	where the kernel $K \in L^2([\Rmin,\Rmax]^2)$ is defined as
	\begin{equation}\label{eq:kernel_K}
		K(r,s) = 	e^{ \frac{\mu_0(r)}{2} + \frac{3\lambda_0(r)}{2} }  e^{\frac{\mu_0(s)}{2}+\frac{3\lambda_0(s)}{2}}   \frac{\sqrt{2r\mu_0'(r)+1}\sqrt{2s\mu_0'(s)+1}}{rs}  \, I(r,s),
	\end{equation}
	with $I$ given by  
	\begin{equation}
		I(r,s) = \left \langle (\mathrm{id}-\Pi) \left ( |\varphi'|E e^{-\lambda_0-\mu_0}  \mathds{1}_{[\Rmin,r]} \right ), |\varphi'|  E e^{-\lambda_0-\mu_0}  \mathds{1}_{[\Rmin,s]}  \right \rangle_H
	\end{equation}
	for $r,s \in [\Rmin,\Rmax]$. The kernel is symmetric, i.e., $K(r,s)=K(s,r)$.
	In particular, $\M$ is a Hilbert-Schmidt operator, see~\cite[Thm.~VI.22 et~seq.]{ReSi1}.
\end{proposition}
\begin{proof}
	The existence of and the formula for the kernel follow from the calculations above. It remains to show that $K \in L^2([\Rmin,\Rmax]^2)$, in particular in the case $\Rmin =0$. By bounding the various steady state quantities by a constant $C>0$, which may change from line to line, we get
	\begin{align}
		\| K \|^2_{L^2([\Rmin,\Rmax]^2)} \leq C \int_{\Rmin}^{\Rmax} \int_{\Rmin}^{\Rmax} \frac{I(r,s)^2}{r^2s^2} \, dr ds. \label{eq:KinL2}
	\end{align}
	The Cauchy-Schwarz inequality, $\|\mathrm{id}-\Pi \| = 1$, and the estimate~\eqref{eq:phi_prime_bound} imply that
	\begin{align*}
		I(r,s)^2 &\leq \| |\varphi'|E e^{-\lambda_0-\mu_0} \mathds{1}_{[\Rmin,r]} \|_H^2\, \| |\varphi'|E e^{-\lambda_0-\mu_0} \mathds{1}_{[\Rmin,s]} \|_H^2 \\
		&\leq  C \left ( \int_{\Rmin}^r \sigma^2 \, d\sigma\right )  \left ( \int_{\Rmin}^s \sigma^2 \, d\sigma\right ) 
	\end{align*}
	for $r,s \in [\Rmin,\Rmax]$. Together with~\eqref{eq:KinL2} we conclude that $K \in L^2([\Rmin,\Rmax]^2)$.
\end{proof}
\begin{remark}
	Obviously, the Mathur operator $\M$ can be extended to an operator $\M\colon L^2([0,\infty[) \to L^2([0,\infty[)$ by setting the kernel $K$ to zero on $\R^2\setminus\left ([\Rmin,\Rmax]^2\right )$. All properties observed above stay valid for this extension. In fact, $K$ vanishes on $\partial [\Rmin,\Rmax]^2$ since $|\varphi'| E e^{-\lambda_0-\mu_0}  \in  \ker(\B)$. In addition, it can be shown that $K$ is continuous on $[\Rmin,\Rmax]^2$ such that this extension is actually continuous as well.
\end{remark}
An explicit bound on the number of negative eigenvalues of $\L$ is now given by the following properties of Hilbert-Schmidt operators. On the one hand, the so-called \enquote{Hilbert-Schmidt norm} $\|\cdot\|_{HS}$ of $\M$ is given by
\begin{equation}\label{eq:HSnormkernel}
	\|\M\|_{HS}^2 = \|K\|_{L^2([0,\infty[^2)}^2 = \int_0^{\infty}\int_0^{\infty} \vert K(r,s)\vert^2 \,dr ds,
\end{equation}
see \cite[Thm.~VI.23]{ReSi1}; we again emphasize that $K$ is supported on $[\Rmin, \Rmax]^2$ and extended by zero. On the other hand, let $\lambda_1\geq\lambda_2\geq\ldots\geq0$ denote the eigenvalues of $\M$ respecting multiplicities, i.e., we repeat each eigenvalue according to its multiplicity. Choosing an $L^2([0,\infty[)$ orthonormal basis of eigenfunctions to these eigenvalues---which is possible by the Hilbert-Schmidt theorem~\cite[Thm.~VI.16]{ReSi1}---and using~\cite[Thm.~VI.22(b)]{ReSi1} together with $\M^\ast=\M$ then yields that
\begin{equation}\label{eq:HSnormsum}
	\|\M\|_{HS}^2 = \sum_{j=1}^\infty\lambda_j^2 ;
\end{equation}
where we extend $\N\ni j\mapsto\lambda_j$ by $0$ if necessary.

Furthermore, since $\M$ is symmetric, non-negative, and compact by Proposition~\ref{thm:Mprop}, \cite[Prop.~5.12]{HiSi} and~\cite[Thm.~VI.6]{ReSi1} imply that
\begin{equation}\label{eq:Mspectralradius}
	\sup\left(\sigma(\M)\right)=\max\left(\sigma(\M)\right)=\|\M\|,
\end{equation}
where $\|\cdot\|$ denotes the operator norm on $L^2([0,\infty[)$, i.e.,
\begin{align}
	\|\M\|\coloneqq&\sup\{\|\M G\|_{L^2([0,\infty[)}\mid G\in L^2([0,\infty[),\;\|G\|_{L^2([0,\infty[)}=1\}\nonumber\\
	=&\sup\{\langle G,\M G\rangle_{L^2([0,\infty[)}\mid G\in L^2([0,\infty[),\;\|G\|_{L^2([0,\infty[)}=1\};\label{eq:operatornorm_def}
\end{align}
the latter equality is due to the symmetry and non-negativity of $\M$.

\subsection{Results on stability}
We now formulate and prove the main results for the steady states as specified in Section~\ref{sc:ststconditions}.
\begin{theorem}[A reduced variational principle]\label{thm:mathur_stability_2}
	$\L$ has a negative eigenvalue if, and only if, the Mathur operator $\M$ has an eigenvalue greater than one, i.e., 
	\[
	\| \M \| = \sup \limits_{ \substack{G\in L^2([0,\infty[) \\ \|G\|_{2}=1} }  \int_{0}^{\infty}\int_{0}^{\infty} K(r,s) G(r)G(s) \, ds dr > 1.
	\]
	Moreover, zero is the smallest eigenvalue of $\L$ if, and only if, $\| \M \| = 1$. 
\end{theorem}
\begin{proof}
	Combine Propositions~\ref{thm:Birman},~\ref{thm:evsLandQ}, and~\ref{prop:Mathur_Kernel}, Lemma~\ref{thm:evsQandM}, and eqns.~\eqref{eq:Mspectralradius}, \eqref{eq:operatornorm_def}.
\end{proof}

We feel obliged to add that this criterion might be equivalent to the existence of a negative direction of the reduced operator in~\cite{HaLiRe2020}. In fact, an equivalence of the Mathur operator and some other reduced operator has been observed in the context of the Vlasov-Poisson system in~\cite[Chapter~5]{Kunze}. However, as in the context of the Vlasov-Poisson system, the Mathur operator is much nicer from a functional analysis point of view.

\begin{cor}[A stability criterion]\label{cor:stab_cond}
	The Antonov operator $\L$ has no non-positive eigenvalues if \mbox{$\|K\|_{L^2([0,\infty[^2)}<1$}.
\end{cor}
\begin{proof}
	Observe $\|\M\|\leq\|\M\|_{HS}$, which follows directly from~\eqref{eq:HSnormsum} and~\eqref{eq:Mspectralradius}, but is also stated in \cite[Thm.~VI.22(d)]{ReSi1}. The statement then follows by~\eqref{eq:HSnormkernel} together with the previous theorem.
\end{proof}

\begin{theorem}[A Birman-Schwinger bound on the number of growing modes]\label{thm:birman_bound}
	It holds that
	\begin{align*}
		\#\{\mathrm{negative~eigenvalues~of~}\L\mathrm{~(counting~multiplicities)}\}<\|K\|_{L^2([0,\infty[^2)}^2.
	\end{align*}
	In particular, $K\not\equiv0$. 
\end{theorem}
\begin{proof}
	First note that Lemma~\ref{thm:ev1deltatominusinfty} yields that $\gamma_1^\ast>0$, recall Definition~\ref{def:gammastar}, which together with Proposition~\ref{thm:Birman} and Lemma~\ref{thm:evsQandM} implies that $\M$ has at least one non-zero eigenvalue, and thus $K\not\equiv0$.
	Furthermore, Proposition~\ref{thm:evsLandQ} and Lemma~\ref{thm:evsQandM} imply that
	\begin{multline*}
		\#\{\text{negative~eigenvalues~of~}\L\text{~(counting~multiplicities)}\}\\ = \#\{\text{eigenvalues~greater~one~of~}\M\text{~(counting~multiplicities)}\}.
	\end{multline*}
	Hence, if $\M$ has no eigenvalues $>1$, $\L$ has no negative eigenvalues and the statement is trivial. Otherwise, let $\lambda_1\geq\lambda_2\geq\ldots\geq0$ denote the eigenvalues of $\M$ respecting multiplicities, i.e., we repeat each eigenvalue according to its multiplicity. Then, by~\eqref{eq:HSnormkernel} and~\eqref{eq:HSnormsum},
	\begin{align*}
		\#&\{\mathrm{eigenvalues~greater~one~of~}\M\mathrm{~(counting~multiplicities)}\}=\#\{j\in\N\mid\lambda_j>1\}\\
		&<\sum_{j\in\N,~\lambda_j>1}\lambda_j^2\leq\sum_{j\in\N}\lambda_j^2=\|\M\|_{HS}^2=\|K\|_{L^2([0,\infty[^2)}^2.\qedhere
	\end{align*}

\end{proof}
We now combine these statements to prove the first main result stated in the introduction:
\begin{proof}[Proof of Theorem~\ref{thm:mathur_stability}]
	First note that $\inf\left(\sigma_{ess}(\L)\right)>0$ by Lemma~\ref{lemma:ess_spectrum_L}, and $\sigma(\L)\setminus\sigma_{ess}(\L)$ consists of isolated eigenvalues of finite multiplicity by  definition of the essential spectrum. Thus, linear stability is equivalent to~$\L$ having no non-positive eigenvalues; recall Definition~\ref{def:linear_stab}. Theorem~\ref{thm:mathur_stability_2} then yields~\ref{it:stab1}. Parts~\ref{it:stab2} and~\ref{it:stab3} follow by Theorem~\ref{thm:birman_bound} and Corollary~\ref{cor:stab_cond}, respectively; note that each negative eigenvalue of $\L$ with multiplicity~$n$ corresponds to~$n$ exponentially growing modes, see Definition~\ref{def:linear_stab} and Remark~\ref{remark:linear_stab}.
\end{proof}

\section{Linear stability of matter shells around a Schwarzschild black hole}\label{sc:stab_blackhole_solutions}
We now apply the methods derived in Section~\ref{sc:birman_schwinger_principle} and, in particular, the reduced variational principle from Theorem~\ref{thm:mathur_stability}. We prove that for $0<\delta \ll  1$ the steady states $f^\delta$ with a Schwarzschild-singularity of fixed mass $M$ constructed in Section~\ref{sc:stst_hole} are linearly stable. 

\begin{theorem}\label{thm:shell_stab_2}
	For fixed choices of $\chi$, $l$, $\Phi$, $L_0$, $r_0$, $\eta_0$, and $E^0$ as specified in Section~\ref{sc:stst_hole} let $(f^\delta)_{\delta>0}$ be the resulting family of static solutions provided by Proposition~\ref{prop:lochexistencestst}. Consider only the case where $\Phi$ is continuously differentiable on $\R$. Then there exists $\delta_0 >0$ such that for every $0<\delta \leq \delta_0$ the static solution $f^\delta$ is linearly stable in the sense of Definition~\ref{def:linear_stab}.
\end{theorem}
\begin{proof}
	In order to apply Theorem~\ref{thm:mathur_stability} we have to check that the steady states $f^\delta$ satisfy the assumptions from Section~\ref{sc:ststconditions}, at least for small values of $\delta>0$. 
	In Proposition~\ref{prop:lochjeans} and Lemma~\ref{lemma:lochperiodsbounded} we have shown that~\ref{it:stst1} and~\ref{it:stst2} hold for $0<\delta \leq \delta_0$ for $\delta_0>0$ sufficiently small. 
	Condition~\ref{it:stst3} is fulfilled by assumption and \ref{it:stst4} is valid as seen in Remark~\ref{remark:single_well_numerics}~\ref{remark:S4condition}. 
	
	From now on we employ the notation from Section~\ref{sc:stst_hole}, i.e., all quantities which depend on the steady state $f^\delta$ are denoted with a superscript $\delta$. Theorem~\ref{thm:mathur_stability}~\ref{it:stab3} yields that $f^\delta$ is linearly stable if $\|K^\delta\|_{L^2([0,\infty[^2)}<1$, where, by Proposition~\ref{prop:Mathur_Kernel}, 
	\begin{equation}\label{eq:Kdelta_def}
		K^\delta(r,s) = e^{ \frac{\mu^\delta(r)}{2}+\frac{3\lambda^\delta(r)}{2}}  e^{\frac{\mu^\delta(s)}{2}+\frac{3\lambda^\delta(s)}{2}}   \frac{\sqrt{2r\left (\mu^\delta\right )' (r)+1}\, \sqrt{2s\left (\mu^\delta\right )' (s)+1}}{rs}  \, I^\delta(r,s)
	\end{equation}
	and 
	\begin{equation*}
		I^\delta(r,s) = \left \langle (\mathrm{id}-\Pi^\delta) \left ( \delta |\varphi'|E e^{-\lambda^\delta-\mu^\delta}  \mathds{1}_{[\Rmin^\delta,r]} \right ), \delta | \varphi'|  E e^{-\lambda^\delta-\mu^\delta}  \mathds{1}_{[\Rmin^\delta,s]}  \right \rangle_{H^\delta}
	\end{equation*}
	for $r,s>2M$ with $K^\delta$, $I^\delta$ extended by zero onto $[0,\infty[^2$ and $K^\delta$, $I^\delta$ are supported inside $[\Rmin^\delta,\Rmax^\delta]^2$, where $0<\Rmin^\delta<\Rmax^\delta$ are the radial bounds of the steady state. Furthermore, note that we have to use $\delta \varphi$ instead of $\varphi$ when comparing with Section~\ref{sc:birman_schwinger_principle} and that $\varphi$ depends on $\delta$ via the cut-off energy $E^\delta$, but we suppress this dependency in our notation. The orthogonal projection onto $\ker(\B^\delta)$ introduced at the start of Section~\ref{ssc:inverseB_inverseBsq} is denoted by $\Pi^\delta$.
	We now estimate $I^\delta$ with the Cauchy-Schwarz inequality as follows:
	\begin{equation*}
		|I^\delta(r,s) | \leq   \left \| \delta |\varphi'| E e^{-\lambda^\delta-\mu^\delta}  \mathds{1}_{[\Rmin^\delta,r]} \right \|_{H^\delta}  \, \left \| \delta |\varphi'| E e^{-\lambda^\delta-\mu^\delta} \mathds{1}_{[\Rmin^\delta,s]} \right \|_{H^\delta},
	\end{equation*}
	where we have used $\|\mathrm{id} - \Pi^\delta\|= 1$ since $\Pi^\delta$ (and thus also $\mathrm{id} - \Pi^\delta$) is an orthogonal projection. The fact that $[\Rmin^\delta,\Rmax^\delta] \subset [\Rmin^0,\Rmax^0]$ from~\eqref{eq:lochrhopcompactsupport} implies that for  $(r,s) \in ]0,\infty[^2$ we have that
	\begin{equation*}
		\frac{|I^\delta(r,s)|}{rs} \leq \frac{1}{(\Rmin^0)^2}\,\left \| \delta |\varphi'| E e^{-\lambda^\delta-\mu^\delta}  \right \|_{H^\delta}^2 = C_0\delta\iiint_{\Omega^\delta}  |\varphi'|\, e^{-\lambda^\delta-2\mu^\delta} E(\sigma,w,L)^2  \, d\sigma dw dL
	\end{equation*} 
	for some $\delta$-independent constant $C_0>0$ which may change from line to line; recall that $\frac{e^{\lambda^\delta}}{\delta|\varphi'|}$ is the integral-weight in the present Hilbert space $H^\delta$. From Lemma~\ref{lemma:lochyconvergence} we know that $\mu^\delta$, $(\mu^\delta)'$, and $\lambda^\delta$ converge to the pure Schwarzschild quantities $\mu^0$, $(\mu^0)'$, and $\lambda^0$, respectively, uniformly on $]2M,\infty[$. Moreover, $E\leq E^\delta<1$ on the steady state support and the area of integration 
	$
		\Omega^\delta=\{(r,w,L)\mid r>4M,~E(r,w,L)<E^\delta\}
	$
	is bounded uniformly in $\delta\in]0,\delta_0]$; the radial boundedness follows by~\eqref{eq:lochrhopcompactsupport} and the $L$-bound can be obtained as in the proof of Proposition~\ref{prop:lochjeans}, which then imply the uniform $w$-boundedness.
	
	Since $\Phi'$ is continuous, we conclude that
	\begin{equation*}
		\frac{|I^\delta(r,s)|}{rs} \leq C_0 \delta, \quad (r,s) \in ]0,\infty[^2,
	\end{equation*}
	after possibly shrinking $\delta_0$. Inserting the latter estimate into~\eqref{eq:Kdelta_def} and again using the uniform bounds yields that
	\begin{equation*}
		\| K^\delta \|_{L^2([0,\infty[^2)} \leq C_0 \delta .
	\end{equation*}
	For $\delta_0 < \frac{1}{C_0}$  we can now apply Theorem~\ref{thm:mathur_stability}~\ref{it:stab3} and obtain the linear stability of $f^\delta$ for $0<\delta\leq \delta_0$.
\end{proof}
It should be mentioned that an alternate way to derive a coercivity estimate for $\L^\delta$, if $0<\delta\ll1$, is to use the methods from~\cite{HaRe2013,HaRe2014}, which simplify in the case of a Schwarzschild-singularity. This approach leads to a result similar to Theorem~\ref{thm:shell_stab_2}. 

The proof of Theorem~\ref{thm:shell_stab} now consists of merely gathering the results above:

\begin{proof}[Proof of Theorem~\ref{thm:shell_stab}]
	The stability property follows from Theorem~\ref{thm:shell_stab_2}. The fact that $\mu^\delta \to \frac 12 \ln({1-\frac{2M}{r}}) $ and  $\lambda^\delta \to -\frac 12 \ln({1-\frac{2M}{r}})$ uniformly on $]2M,\infty[$ as $\delta \to 0$ was shown in Lemma~\ref{lemma:lochyconvergence}. The pointwise convergence of $f^\delta$ to zero is a simple consequence of $f^\delta = \delta \varphi$ on the support of $f^\delta$. 
\end{proof}

\end{document}